\documentclass[10pt]{amsart}
\usepackage{amsmath}
\usepackage{amssymb}
\usepackage{amsbsy, amsthm,amstext, amsopn}
\usepackage[all]{xy}
\usepackage{amsfonts}
\usepackage{amscd}
\usepackage{stmaryrd}
\usepackage{color}
\usepackage[numbers, square, sort, comma]{natbib}

\usepackage{graphicx}

\newtheorem{thm}{Theorem} [section]
\newtheorem{lemma}[thm]{Lemma}

\newtheorem{corollary}[thm]{Corollary}
\newtheorem{prop}[thm]{Proposition}

\theoremstyle{definition}

\newtheorem{defn}[thm]{Definition}

\newtheorem{example}[thm]{Example}

\newtheorem{ansatz}[thm]{Ansatz}

\newtheorem{remark}[thm]{Remark}

\begin{document}

\numberwithin{equation}{section}

\newcommand{\hs}{\mbox{\hspace{.4em}}}
\newcommand{\ds}{\displaystyle}
\newcommand{\bd}{\begin{displaymath}}
\newcommand{\ed}{\end{displaymath}}
\newcommand{\bcd}{\begin{CD}}
\newcommand{\ecd}{\end{CD}}

\newcommand{\on}{\operatorname}
\newcommand{\proj}{\operatorname{Proj}}
\newcommand{\bproj}{\underline{\operatorname{Proj}}}

\newcommand{\spec}{\operatorname{Spec}}
\newcommand{\Spec}{\operatorname{Spec}}
\newcommand{\bspec}{\underline{\operatorname{Spec}}}
\newcommand{\pline}{{\mathbf P} ^1}
\newcommand{\aline}{{\mathbf A} ^1}
\newcommand{\pplane}{{\mathbf P}^2}
\newcommand{\aplane}{{\mathbf A}^2}
\newcommand{\coker}{{\operatorname{coker}}}
\newcommand{\ldb}{[[}
\newcommand{\rdb}{]]}

\newcommand{\Sym}{\operatorname{Sym}^{\bullet}}
\newcommand{\Symp}{\operatorname{Sym}}
\newcommand{\Pic}{\bf{Pic}}
\newcommand{\Aut}{\operatorname{Aut}}
\newcommand{\PAut}{\operatorname{PAut}}

\newcommand{\too}{\twoheadrightarrow}
\newcommand{\C}{{\mathbf C}}
\newcommand{\Z}{{\mathbf Z}}
\newcommand{\Q}{{\mathbf Q}}
\newcommand{\Cx}{{\mathbf C}^{\times}}
\newcommand{\Cbar}{\overline{\C}}
\newcommand{\Cxbar}{\overline{\Cx}}
\newcommand{\cA}{{\mathcal A}}
\newcommand{\cS}{{\mathcal S}}
\newcommand{\cV}{{\mathcal V}}
\newcommand{\cM}{{\mathcal M}}
\newcommand{\bA}{{\mathbf A}}
\newcommand{\cB}{{\mathcal B}}
\newcommand{\cC}{{\mathcal C}}
\newcommand{\cD}{{\mathcal D}}
\newcommand{\D}{{\mathcal D}}
\newcommand{\cs}{{\mathbf C} ^*}
\newcommand{\boldc}{{\mathbf C}}
\newcommand{\cE}{{\mathcal E}}
\newcommand{\cF}{{\mathcal F}}
\newcommand{\bF}{{\mathbf F}}
\newcommand{\cG}{{\mathcal G}}
\newcommand{\G}{{\mathbb G}}
\newcommand{\cH}{{\mathcal H}}
\newcommand{\CI}{{\mathcal I}}
\newcommand{\cJ}{{\mathcal J}}
\newcommand{\cK}{{\mathcal K}}
\newcommand{\cL}{{\mathcal L}}
\newcommand{\baL}{{\overline{\mathcal L}}}

\newcommand{\fI}{{\mathfrak I}}
\newcommand{\fJ}{{\mathfrak J}}
\newcommand{\fF}{{\mathfrak F}}
\newcommand{\Mf}{{\mathfrak M}}
\newcommand{\bM}{{\mathbf M}}
\newcommand{\bm}{{\mathbf m}}
\newcommand{\cN}{{\mathcal N}}
\newcommand{\theo}{\mathcal{O}}
\newcommand{\cP}{{\mathcal P}}
\newcommand{\cR}{{\mathcal R}}
\newcommand{\Pp}{{\mathbb P}}
\newcommand{\boldp}{{\mathbf P}}
\newcommand{\boldq}{{\mathbf Q}}
\newcommand{\bbL}{{\mathbf L}}
\newcommand{\cQ}{{\mathcal Q}}
\newcommand{\cO}{{\mathcal O}}
\newcommand{\Oo}{{\mathcal O}}
\newcommand{\cY}{{\mathcal Y}}
\newcommand{\OX}{{\Oo_X}}
\newcommand{\OY}{{\Oo_Y}}
\newcommand{\otY}{{\underset{\OY}{\ot}}}
\newcommand{\otX}{{\underset{\OX}{\ot}}}
\newcommand{\cU}{{\mathcal U}}\newcommand{\cX}{{\mathcal X}}
\newcommand{\cW}{{\mathcal W}}
\newcommand{\boldz}{{\mathbf Z}}
\newcommand{\qgr}{\operatorname{q-gr}}
\newcommand{\gr}{\operatorname{gr}}
\newcommand{\rk}{\operatorname{rk}}
\newcommand{\SH}{{\underline{\operatorname{Sh}}}}
\newcommand{\End}{\operatorname{End}}
\newcommand{\uEnd}{\underline{\operatorname{End}}}
\newcommand{\Hom}{\operatorname{Hom}}
\newcommand{\uHom}{\underline{\operatorname{Hom}}}
\newcommand{\uHomY}{\uHom_{\OY}}
\newcommand{\uHomX}{\uHom_{\OX}}
\newcommand{\Ext}{\operatorname{Ext}}
\newcommand{\bExt}{\operatorname{\bf{Ext}}}
\newcommand{\Tor}{\operatorname{Tor}}

\newcommand{\inv}{^{-1}}
\newcommand{\airtilde}{\widetilde{\hspace{.5em}}}
\newcommand{\airhat}{\widehat{\hspace{.5em}}}
\newcommand{\nt}{^{\circ}}
\newcommand{\del}{\partial}

\newcommand{\supp}{\operatorname{supp}}
\newcommand{\GK}{\operatorname{GK-dim}}
\newcommand{\hd}{\operatorname{hd}}
\newcommand{\id}{\operatorname{id}}
\newcommand{\res}{\operatorname{res}}
\newcommand{\lrar}{\leadsto}
\newcommand{\im}{\operatorname{Im}}
\newcommand{\HH}{\operatorname{H}}
\newcommand{\TF}{\operatorname{TF}}
\newcommand{\Bun}{\operatorname{Bun}}

\newcommand{\F}{\mathcal{F}}
\newcommand{\Ff}{\mathbb{F}}
\newcommand{\nthord}{^{(n)}}
\newcommand{\Gr}{{\mathfrak{Gr}}}

\newcommand{\Fr}{\operatorname{Fr}}
\newcommand{\GL}{\operatorname{GL}}
\newcommand{\gl}{\mathfrak{gl}}
\newcommand{\SL}{\operatorname{SL}}
\newcommand{\ff}{\footnote}
\newcommand{\ot}{\otimes}
\def\Ext{\operatorname {Ext}}
\def\Hom{\operatorname {Hom}}
\def\Ind{\operatorname {Ind}}
\def\bbZ{{\mathbb Z}}

\newcommand{\nc}{\newcommand}
\nc{\ol}{\overline} \nc{\cont}{\on{cont}} \nc{\rmod}{\on{mod}}
\nc{\Mtil}{\widetilde{M}} \nc{\wb}{\overline} 
\nc{\wh}{\widehat}  \nc{\mc}{\mathcal}
\nc{\mbb}{\mathbb}  \nc{\K}{{\mc K}} \nc{\Kx}{{\mc K}^{\times}}
\nc{\Ox}{{\mc O}^{\times}} \nc{\unit}{{\bf \on{unit}}}
\nc{\boxt}{\boxtimes} \nc{\xarr}{\stackrel{\rightarrow}{x}}

\nc{\Ga}{\G_a}
 \nc{\PGL}{{\on{PGL}}}
 \nc{\PU}{{\on{PU}}}

\nc{\h}{{\mathfrak h}} \nc{\kk}{{\mathfrak k}}
 \nc{\Gm}{\G_m}
\nc{\Gabar}{\wb{\G}_a} \nc{\Gmbar}{\wb{\G}_m} \nc{\Gv}{G^\vee}
\nc{\Tv}{T^\vee} \nc{\Bv}{B^\vee} \nc{\g}{{\mathfrak g}}
\nc{\gv}{{\mathfrak g}^\vee} \nc{\BRGv}{\on{Rep}\Gv}
\nc{\BRTv}{\on{Rep}T^\vee}
 \nc{\Flv}{{\mathcal B}^\vee}
 \nc{\TFlv}{T^*\Flv}
 \nc{\Fl}{{\mathfrak Fl}}
\nc{\BRR}{{\mathcal R}} \nc{\Nv}{{\mathcal{N}}^\vee}
\nc{\St}{{\mathcal St}} \nc{\ST}{{\underline{\mathcal St}}}
\nc{\Hec}{{\bf{\mathcal H}}} \nc{\Hecblock}{{\bf{\mathcal
H_{\alpha,\beta}}}} \nc{\dualHec}{{\bf{\mathcal H^\vee}}}
\nc{\dualHecblock}{{\bf{\mathcal H^\vee_{\alpha,\beta}}}}
\newcommand{\ramBun}{{\bf{Bun}}}
\newcommand{\ramBuno}{\ramBun^{\circ}}

\nc{\Buntheta}{{\bf Bun}_{\theta}} \nc{\Bunthetao}{{\bf
Bun}_{\theta}^{\circ}} \nc{\BunGR}{{\bf Bun}_{G_\BR}}
\nc{\BunGRo}{{\bf Bun}_{G_\BR}^{\circ}}
\nc{\HC}{{\mathcal{HC}}}
\nc{\risom}{\stackrel{\sim}{\to}} \nc{\Hv}{{H^\vee}}
\nc{\bS}{{\mathbf S}}
\def\BRep{\operatorname {Rep}}
\def\Conn{\operatorname {Conn}}

\nc{\Vect}{{\operatorname{Vect}}}
\nc{\Hecke}{{\operatorname{Hecke}}}

\newcommand{\ZZ}{{Z_{\bullet}}}
\nc{\HZ}{{\mc H}\ZZ} \nc{\eps}{\epsilon}

\nc{\CN}{\mathcal N} \nc{\BA}{\mathbb A}

 \nc{\BB}{\mathbb B}

\nc{\ul}{\underline}

\nc{\bn}{\mathbf n} \nc{\Sets}{{\on{Sets}}} \nc{\Top}{{\on{Top}}}
\nc{\IntHom}{{\mathcal Hom}}

\nc{\Simp}{{\mathbf \Delta}} \nc{\Simpop}{{\mathbf\Delta^\circ}}

\nc{\Cyc}{{\mathbf \Lambda}} \nc{\Cycop}{{\mathbf\Lambda^\circ}}

\nc{\Mon}{{\mathbf \Lambda^{mon}}}
\nc{\Monop}{{(\mathbf\Lambda^{mon})\circ}}

\nc{\Aff}{{\on{Aff}}} \nc{\Sch}{{\on{Sch}}}

\nc{\bul}{\bullet}
\nc{\module}{{\operatorname{-mod}}}

\nc{\dstack}{{\mathcal D}}

\nc{\BL}{{\mathbb L}}

\nc{\BD}{{\mathbb D}}

\nc{\BR}{{\mathbb R}}

\nc{\BT}{{\mathbb T}}

\nc{\SCA}{{\mc{SCA}}}
\nc{\DGA}{{\mc DGA}}

\nc{\DSt}{{DSt}}

\nc{\lotimes}{{\otimes}^{\mathbf L}}

\nc{\bs}{\backslash}

\nc{\Lhat}{\widehat{\mc L}}

\newcommand{\Coh}{\on{Coh}}

\nc{\QCoh}{QC}
\nc{\QC}{QC}
\nc{\Perf}{\on{Perf}}
\nc{\Cat}{{\on{Cat}}}
\nc{\dgCat}{{\on{dgCat}}}
\nc{\bLa}{{\mathbf \Lambda}}

\nc{\BRHom}{\mathbf{R}\hspace{-0.15em}\on{Hom}}
\nc{\BREnd}{\mathbf{R}\hspace{-0.15em}\on{End}}
\nc{\colim}{\on{colim}}
\nc{\oo}{\infty}
\nc{\Mod}{\on{Mod} }

\nc\fh{\mathfrak h}
\nc\al{\alpha}
\nc\la{\alpha}
\nc\BGB{B\bs G/B}
\nc\QCb{QC^\flat}
\nc\qc{\on{QC}}

\def\w{\wedge}
\nc{\vareps}{\varepsilon}

\nc{\fg}{\mathfrak g}

\nc{\Map}{\on{Map}} \nc{\fX}{\mathfrak X}

\nc{\ch}{\check}
\nc{\fb}{\mathfrak b} \nc{\fu}{\mathfrak u} \nc{\st}{{st}}
\nc{\fU}{\mathfrak U}
\nc{\fZ}{\mathfrak Z}
\nc{\fB}{\mathfrak B}

 \nc\fc{\mathfrak c}
 \nc\fs{\mathfrak s}

\nc\fk{\mathfrak k} \nc\fp{\mathfrak p}
\nc\fq{\mathfrak q}

\nc{\BRP}{\mathbf{RP}} \nc{\rigid}{\text{rigid}}
\nc{\glob}{\text{glob}}

\nc{\cI}{\mathcal I}

\nc{\La}{\mathcal L}

\nc{\quot}{/\hspace{-.25em}/}

\nc\aff{\mathit{aff}}
\nc\BS{\mathbb S}

\nc\Loc{{\mc Loc}}
\nc\Tr{{\on{Tr}}}
\nc\Ch{{\mc Ch}}

\nc\ftr{{\mathfrak {tr}}}
\nc\fM{\mathfrak M}

\nc\Id{\operatorname{Id}}

\nc\bimod{\on{-bimod}}

\nc\ev{\operatorname{ev}}
\nc\coev{\operatorname{coev}}

\nc\pair{\operatorname{pair}}
\nc\kernel{\operatorname{kernel}}

\nc\Alg{\operatorname{Alg}}

\nc\init{\emptyset_{\text{\em init}}}
\nc\term{\emptyset_{\text{\em term}}}

\nc\Ev{\on{Ev}}
\nc\Coev{\on{Coev}}

\nc\es{\emptyset}
\nc\m{\text{\it min}}
\nc\M{\text{\it max}}
\nc\cross{\text{\it cr}}
\nc\tr{\on{tr}}

\nc\perf{\on{-perf}}
\nc\inthom{\mathcal Hom}
\nc\intend{\mathcal End}

\newcommand{\Sh}{\mathit{Sh}}

\nc{\Comod}{\on{Comod}}
\nc{\cZ}{\mathcal Z}

\def\interiorsymbol {\on{int}}

\nc\frakf{\mathfrak f}
\nc\fraki{\mathfrak i}
\nc\frakj{\mathfrak j}
\nc\BP{\mathbb P}
\nc\stab{st}
\nc\Stab{St}

\nc\fN{\mathfrak N}
\nc\fT{\mathfrak T}
\nc\fV{\mathfrak V}

\nc\Ob{\on{Ob}}

\nc\fC{\mathfrak C}
\nc\Fun{\on{Fun}}

\nc\Null{\on{Null}}

\nc\BC{\mathbb C}

\nc\loc{\on{Loc}}

\nc\hra{\hookrightarrow}
\nc\fL{\mathfrak L}
\nc\R{\mathbb R}
\nc\CE{\mathcal E}

\nc\sK{\mathsf K}
\nc\sL{\mathsf L}
\nc\sC{\mathsf C}

\nc\Cone{\mathit Cone}

\nc\fY{\mathfrak Y}
\nc\fe{\mathfrak e}
\nc\ft{\mathfrak t}

\nc\wt{\widetilde}
\nc\inj{\mathit{inj}}
\nc\surj{\mathit{surj}}

\nc\Path{\mathit{Path}}
\nc\Set{\mathit{Set}}
\nc\Fin{\mathit{Fin}}

\nc\cyc{\mathit{cyc}}

\nc\per{\mathit{per}}

\nc\sym{\mathit{symp}}
\nc\con{\mathit{cont}}
\nc\gen{\mathit{gen}}
\nc\str{\mathit{str}}
\nc\rsdl{\mathit{res}}
\nc\impr{\mathit{impr}}
\nc\rel{\mathit{rel}}
\nc\pt{\mathit{pt}}
\nc\naive{\mathit{nv}}
\nc\forget{\mathit{For}}

\nc\sH{\mathsf H}
\nc\sW{\mathsf W}
\nc\sE{\mathsf E}
\nc\sP{\mathsf P}
\nc\sB{\mathsf B}
\nc\sS{\mathsf S}
\nc\fH{\mathfrak H}
\nc\fP{\mathfrak P}
\nc\fW{\mathfrak W}
\nc\fE{\mathfrak E}
\nc\sx{\mathsf x}
\nc\sy{\mathsf y}

\nc\ord{\mathit{ord}}

\nc\sm{\mathit{sm}}

\nc\rhu{\rightharpoonup}
\nc\dirT{\mathcal T}
\nc\dirF{\mathcal F}
\nc\link{\mathit{link}}
\nc\cT{\mathcal T}

\newcommand{\ssupp}{\mathit{ss}}
\newcommand{\cyl}{\mathit{Cyl}}
\newcommand{\ball}{\mathit{B(x)}}

 \nc\ssf{\mathsf f}
 \nc\ssg{\mathsf g}
\nc\sq{\mathsf q}
 \nc\sQ{\mathsf Q}
 \nc\sR{\mathsf R}

\nc\fa{\mathfrak a}
\nc\fA{\mathfrak A}

\nc\trunc{\mathit{tr}}
\nc\pre{\mathit{pre}}
\nc\expand{\mathit{exp}}

\nc\Sol{\mathit{Sol}}
\nc\direct{\mathit{dir}}

\nc\out{\mathit{out}}
\nc\Morse{\mathit{Morse}}
\nc\arb{\mathit{arb}}
\nc\prearb{\mathit{pre}}

\nc\BZ{\mathbb Z}
\nc\proper{\mathit{prop}}
\nc\torsion{\mathit{tors}}
\nc\Perv{\mathit{Perv}}
\nc\IC{\operatorname{IC}}

\nc\Conv{\operatorname{Conv}}
\nc\Span{\operatorname{Span}}

\nc\image{\operatorname{image}}
\nc\lin{\mathit{lin}}

\nc\inverse{\operatorname{inv}}

\nc\real{\operatorname{Re}}
\nc\imag{\operatorname{Im}}

\nc\fw{\mathfrak w}

\nc\sign{\on{sgn}}
\nc\conic{\mathit{con}}

\nc\even{\mathit{ev}}
\nc\odd{\mathit{odd}}
\nc\add{\mathit{add}}

\nc\orient{\mathit{or}}

\nc\op{\mathit{op}}

\title[Mirror symmetry for the Landau-Ginzburg $A$-model $M= \BC^n, W=z_1 \cdots z_n$]{Mirror symmetry for the Landau-Ginzburg 
$A$-model $M= \BC^n, W=z_1 \cdots z_n$}

\author{David Nadler}
\address{Department of Mathematics\\University of California, Berkeley\\Berkeley, CA  94720-3840}
\email{nadler@math.berkeley.edu}

\begin{abstract}
We calculate the category of branes in the Landau-Ginzburg $A$-model with background $M= \BC^n$ and superpotential $W=z_1 \cdots z_n$ in the form of microlocal sheaves along a natural Lagrangian skeleton. Our arguments employ
the framework of perverse schobers, and our results
 confirm expectations from mirror symmetry. \end{abstract}

\maketitle


\tableofcontents


\section{Introduction}

The aim of this paper is to establish a homological mirror symmetry equivalence for  
 the Landau-Ginzburg $A$-model with background $M= \BC^n$
and superpotential $W= z_1\cdots z_n$.
It presents new challenges due to the fact that the critical locus $\{dW = 0\} \subset M$ 
is not smooth or proper.
Its fundamental role is witnessed by the fact that its mirror variety is the $(n-2)$-dimensional pair of pants, the open complement of $n$ generic hyperplanes in $\BP^{n-2}$. 
The results of this paper strengthen and generalize to arbitrary dimensions the results of~\cite{N3d} for the case 
 $M= \BC^3, W= z_1 z_2 z_3$
 though the arguments
differ.
 Here we emphasize the role of symmetry in simplifying the calculation, while in~\cite{N3d}
we broke symmetry following the theory developed in~\cite{Narb, Nexp}. 
What results may appeal to audiences in several fields with distinct practices:

\smallskip

{\em (1) Constructible/Microlocal sheaves.} While our arguments employ universal paradigms that could apply in many settings, we have adopted the technical framework of microlocal sheaves~\cite{KS}. The calculation of categories of constructible sheaves forms  a longstanding central challenge
in Geometric Representation Theory  (notably stemming from Kazhdan-Lustzig theory~\cite{KL} and Lusztig's character sheaves~\cite{laumon, lusztig}), and prominently in the Geometric Langlands program
 (for example, in the Geometric Satake correspondence~\cite{BD, Gi, MV}). The rapidly growing industry of symplectic resolutions and their quantizations (see for example~\cite{BPW}) provides a broader setting where  microlocalization becomes a basic construction.   
Recent advances (\cite{STZ, STWZ}) have also broadened the impact of constructible sheaves and their microlocalizations on
symplectic and enumerative invariants. In particular,   our calculation in the case of $M= \BC^3, W= z_1 z_2 z_3$ 
established in \cite{N3d} appears prominently in work of Treumann-Zaslow~\cite{TZ} on Legendrian surfaces.

\smallskip

{\em (2) Homological mirror symmetry.} A natural motivation for our main result is homological mirror symmetry for Landau-Ginzburg models.
  For background on homological mirror symmetry, and specifically the
 Landau-Ginzburg model studied here, we refer the reader to the beautiful paper~\cite{AAEKO} and the references therein. 
  It establishes the ``opposite direction" of homological mirror symmetry 
  between the Landau-Ginzburg $B$-model of $M=\BC^3, W= z_1 z_2 z_3$,
 in the form of the derived category of singularities, and the $A$-model of  $\BP^1\setminus \{0, 1, \oo\}$, in the form of the wrapped Fukaya category.
 (For a brief discussion about the different guises of the $A$-model, see Remark~\ref{rem ansatz} below and the references therein.)
This can be viewed as a refinement of the results of Seidel~\cite{seidelgenustwo}, which in turn are generalized  by Sheridan~\cite{sheridan} 
 to a matching of
   the endomorphism algebras of the structure sheaf of the origin in the
   Landau-Ginzburg $B$-model of    $M=\BC^n, W= z_1 \cdots z_n$
   and
   of  a distinguished compact brane in the $A$-model of the $(n-2)$-dimensional pair of pants.
    For the direction of homological mirror symmetry considered here, there is also work in progress~\cite{AA} with results parallel to those of this paper.

\smallskip

{\em (3) Categorified sheaf theory.} A third setting for our results and arguments is the nascent subject of categorified sheaf theory. In traditional sheaf theory, a distinguished role is played by the nearby and vanishing cycles, which encode the Morse theory of sections. To formalize a similar structure  for sheaves of categories, Kapranov-Schechtman~\cite{KapS}
proposed the notion of perverse schober. In its most basic realization, the natural map from the
vanishing to nearby cycles is replaced by a spherical functor from a vanishing to nearby dg category. A motivating example is given by  the $A$-model of a Lefschetz fibration, where the vanishing dg category at each critical point is the local Landau-Ginzburg model. One expects the $A$-model of more general superpotentials to also provide perverse schobers, and our main technical work confirms this for  $M= \BC^n$, $W= z_1\cdots z_n$.


\subsection{Main result}
Set $M=\BC^n$, with coordinates $z_1 = r_1 e^{i\theta_1}, \ldots, z_n= r_n e^{i\theta_n}$, and superpotential $W=z_1 \cdots z_n$,
The origin $0\in \BC$ is the only critical value of $W$,
and we 
set 
$$
\xymatrix{
M_0 = W^{-1}(0) = \bigcup_{a=1}^n \{ z_a = 0\}
&
M_1 = W^{-1}(1) \simeq (\BC^\times)^{n-1}
}
$$
$$
\xymatrix{
M_{>0} = W^{-1}(\R_{>0}) \simeq (\BC^\times)^{n-1} \times \BR_{>0}
& 
M^\times = W^{-1}(\BC^\times) \simeq (\BC^\times)^{n}
}
$$ 

We also write
 $T=(S^1)^n$
for the standard $n$-torus, $\ft = \BR^n$ for its Lie algebra, $T^\circ  \simeq (S^1)^{n-1}\subset T$
for the kernel of the diagonal character, $\ft^\circ \subset \ft$ for its Lie algebra, 
and work with a natural  symplectic identification
$$
\xymatrix{
M_1\simeq (\BC^\times)^{n-1} \simeq T^*T^{\circ}\simeq T^\circ\times\ft^\circ
}
$$

 
%

Following the  paradigms of 
Landau-Ginzburg $A$-models, 
we will focus on the geometry of $M$ above a cut in the plane $\BC$, specifically the non-negative real ray $\BR_{\geq 0}\subset \BC$.
We introduce a natural   Lagrangian skeleton
$
L\subset M,
$
defined in polar coordinates by the equations
$$
\xymatrix{
\displaystyle
\sum_{a=1}^n \theta_a = 0 &\mbox{ and } & \theta_a = 0 \mbox{ when } r_a \not = r_\m,
\mbox{ for } a=1,\ldots,n 
}
$$
where we set $r_\m =\min\{ r_a\, |\, a=1,\ldots, n\}$.
It is a  closed Lagrangian subvariety,  conic with respect to positive real scalings,
and equal to the closure of its open subspace
$
L^\times =L \cap M^\times  = L \cap M_{>0}.
$ 
 Therefore  it is determined by its fiber 
$
L_1 = L \cap M_1,
$
which is itself a Lagrangian subvariety of $M_1$. 

Under the identification
$M_1\simeq  T^*T^{\circ}$,
the Lagrangian subvariety $L_1 \subset M_1$
transports to a conic Lagrangian subvariety $\Lambda_\Sigma \subset  T^* T^\circ$
of a simple combinatorial nature
$$
\xymatrix{
\Lambda_\Sigma = \bigcup_{\sigma \in \Sigma} \sigma T^\circ \times \sigma  \subset T^\circ \times (\ft^\circ)^*
}
$$
Here $\Sigma \subset (\ft^\circ)^*$ is the complete fan on the images $\ol e_1, \ldots, \ol e_n \in (\ft^\circ)^*$ of the coordinate vectors $ e_1, \ldots, e_n \in \ft^*$ under the restriction $\ft^*\to (\ft^\circ)^*$,
and given a positive cone $\sigma\in \Sigma$, we write  $\sigma T^\circ \subset T^\circ$ for the subtorus with Lie algebra the orthogonal
subspace $\sigma^\perp \subset \ft^\circ$.

Returning to the 
Landau-Ginzburg $A$-model, 
we would like to study $A$-branes within $M$ running along the Lagrangian skeleton $L$, as found in the infinitesimal Fukaya-Seidel category~\cite{Seidel}, or transverse to $L$, as found in the partially wrapped Fukaya category~\cite{AS, Aur}.
In some generality, these two variants are expected to be in duality (in parallel with $B$-model dualities as found
in~\cite{BNP}), and in the specific situation at hand, each should in fact be self-dual and  equivalent to microlocal sheaves on $M$ supported along $L$. 

\begin{ansatz}\label{ansatz msh}
The category of branes in the Landau-Ginzburg $A$-model of $M= \BC^n, W= z_1\cdots z_n$ with Lagrangian skeleton $L\subset M$ is given by the dg category of microlocal sheaves on $M$ supported along $L$.
\end{ansatz}

\begin{remark}\label{rem ansatz}
The ansatz is compatible with the broad expectation, realized in numerous situations,
 that given  $L\subset M$ a Lagrangian skeleton of an exact symplectic manifold, there are equivalent approaches to its ``quantum category" of $A$-branes: the Floer-Fukaya-Seidel theory of Lagrangian intersections and pseudo-holomorphic disks \cite{FOOO, Seidel} (analysis); the Kashiwara-Schapira theory of microlocal  sheaves~\cite{KS} (topology); the theory of holonomic modules over deformation quantizations, exemplified by $\cD$-modules~\cite{bernstein}
(algebra); and finitary models following expectations of Kontsevich~\cite{kont, Narb, Nexp} (combinatorics). 
In particular, since all of our constructions ultimately lie in cotangent bundles, 
one could translate our results into the traditional language of Fukaya categories
following~\cite{NZ, Nequiv}. Furthermore, there is work in progress~\cite{ENS, GPS} detailing such equivalences
more generally for Weinstein manifolds.  When the dust settles, the results of this paper, and perhaps more interestingly, its methods,   should hold  independently
of the specific language used to describe $A$-branes.

\end{remark}

\begin{remark}
One can argue that $L\subset M$ is the most fundamental Lagrangian skeleton for  the 
Landau-Ginzburg model $M= \BC^n, W= z_1\cdots z_n$, but it is by no means the only possibility.
For example, we discuss below the alternative ``singular thimble" $L_c \subset M$,
which is proper over  $\BR_{\geq 0}\subset \BC$ and can be thought of as the smallest nondegenerate Lagrangian skeleton.
Thanks to the inclusion $L_c \subset L$, our results for $L$ easily imply results for $L_c$, which we record in some corollaries below.
But there are other distinct possibilities associated to alternative Lagrangian skeleta of the fiber $M_1 \simeq T^* T^\circ$, for example conic Lagrangian subvarieties $\Lambda_{\Sigma'} \subset  T^* T^\circ$
 defined by alternative fans $\Sigma'\subset (\ft^\circ)^*$. We expect our techniques to extend easily to this level of generality,
 and more broadly to other Landau-Ginzburg models as well.

\end{remark}

Now we will state our main theorem. 
Fix a base field $k$ of characteristic zero.

Let $\mu\Sh_{L}(M)$ denote the dg category of microlocal sheaves of $k$-vector spaces supported along 
the Lagrangian skeleton $L \subset M$. In Sections~\ref{ss prelim} and~\ref{ss micro shs}, we explain how to  work with  such microlocal sheaves 
 building on the foundations of Kashiwara-Schapira~\cite{KS}. Roughly speaking, 
 we  identify the contactification $N = M\times \BR$ with the one-jet bundle $JX = T^*X \times \BR$
 of the base manifold $X= \BR^n$, and
  then observe that its symplectification  is equivalent to an 
 open conic subspace $\Omega_X \subset T^*( X\times \BR)$. 
 The symplectification and contactification come with natural maps
 $$
 \xymatrix{
 \Omega_X \ar[r]^-s &  JX \simeq N\ar[r]^-c &  M
 }
 $$
and we  lift $L\subset M$  along the projection $c$ to the  Legendrian subvariety $L\times\{0\} \subset N$, then
transport it to $JX$, and  take its inverse-image under $s$
 to arrive at a conic Lagrangian subvariety
$\Lambda \subset \Omega_X$.
The fact that $L\subset M$ is conic implies that $\Lambda\subset \Omega_X$ is in fact biconic, and in particular conic for a contracting action on $X \times \BR$ with fixed locus the origin.
In this setting, one can define microlocal sheaves as a localization of conic constructible sheaves on $X$ such that  the intersection
of their singular support
 with $\Omega_X$ lies within $\Lambda$.

Thanks to the comprehensive work~\cite{KS}, microlocal sheaves enjoy powerful functoriality
induced by similar functoriality for constructible sheaves.
Microlocal kernels induce microlocal transformations, and Hamiltonian reductions
induce natural functors. 
For example, 
the open inclusion $M^\times \subset M$ provides a restriction functor 
$$
\xymatrix{
J^*:\mu\Sh_{L}(M)\ar[r] &   \mu\Sh_{L^\times}(M^\times) 
}
$$
and the Lagrangian correspondence $M^\times \leftarrow M_{>0} \to M_1$
leads to an equivalence
 $$
\xymatrix{
 \mu\Sh_{L^\times}(M^\times) \ar[r]^-\sim & \mu\Sh_{L_1}(M_1)
}
$$

Going further, 
the identification
$M_1 \simeq T^* T^\circ$
 allows us to pass from microlocal sheaves to a more concrete dg category of constructible sheaves
 $$
\xymatrix{
 \mu\Sh_{L_1}(M_1) \ar[r]^-\sim & \Sh_{\Lambda_\Sigma}(T^\circ)
}
$$
Moreover, the conic Lagrangian subvariety $\Lambda_\Sigma \subset T^* T^\circ$ is the singular support condition  appearing in the most basic instance
$$
\xymatrix{
 \Sh_{\Lambda_\Sigma}( T^{\circ})  \ar[r]^-\sim & \Coh(\BP^{n-1})
}
$$
of the coherent-constructible correspondence~\cite{B, ccc, T} between dg categories of constructible and coherent sheaves.

Here the projective space $\BP^{n-1}$ arises as  the $\check T^\circ$-toric variety for the complete fan $\Sigma \subset (\ft^\circ)^*$ 
and algebraic torus $\check T^\circ \simeq (\G_m^\times)^{n-1}$ dual to the compact torus $T^\circ \simeq (S^1)^{n-1}$.
The conic Lagrangian subvariety $\Lambda_\Sigma \subset T^* T^\circ$ contains the zero-section $T^{\circ} \subset T^* T^{\circ}$, the singular support condition 
appearing in the usual Fourier equivalence
$$
\xymatrix{
\Loc(T^\circ) \simeq \Sh_{T^{\circ}}( T^{\circ})  \ar[r]^-\sim & \Coh_{\torsion}(\check T^\circ) 
}
$$
between finite-rank local systems and  torsion sheaves.


%
Now to state our main theorem, 
consider the section 
$$
\xymatrix{
s:\cO_{\BP^{n-1}} \ar[r] & \cO_{\BP^{n-1}}(1)
&
s([x_1, \ldots, x_n]) = x_1 + \cdots + x_n
}$$
and the inclusion of its zero-locus
$$
\xymatrix{
i:\BP^{n-2} \simeq \{ s= 0\} \ar@{^(->}[r] & \BP^{n-1} 
}$$
The specific coefficients of $s$ are not  important only the $\check T^\circ$-invariant fact that they are all non-zero.
Consider the corresponding pushforward
on bounded dg categories of coherent complexes
$$
\xymatrix{
i_*:\Coh(\BP^{n-2}) \ar[r] & \Coh(\BP^{n-1})
}
$$

Here is our main theorem (appearing as Theorem~\ref{thm main} below).

\begin{thm}\label{intro main thm}
There is a 
commutative diagram with horizontal equivalences 
$$
\xymatrix{
\ar[d]_{J^*} \mu\Sh_{L}(M) \ar[r]^-\sim & \Coh(\BP^{n-2})\ar[d]^-{i_*} \\
 \mu\Sh_{L^\times}(M^\times) \ar[r]^-\sim & \Coh(\BP^{n-1})
}
$$
\end{thm}

The theorem immediately implies a subsidiary mirror equivalence 
which some readers may find more expected.
Introduce the proper Lagrangian skeleton $L_c  \subset M$ defined in polar coordinates by the equations 
$$
\xymatrix{
\displaystyle
\sum_{a=1}^n \theta_a = 0 & \mbox{ and } & r_a = r_b, \mbox{ for } a, b = 1,\ldots, n
}$$
It is a  closed Lagrangian subvariety, conic with respect to positive real scalings, 
and proper over  $\BR_{\geq 0}\subset \BC$.
It can be viewed as a ``singular thimble"  in that it is the cone over the vanishing torus
$$
\xymatrix{
L_c = \Cone(T^\circ) \subset M
}$$  

Let $\mu\Sh_{L_c}(M) \subset \mu\Sh_{L}(M)$ denote the full dg subcategory of microlocal sheaves of $k$-vector spaces supported along 
 $L_c \subset M$. 

For each $a=1,\ldots, n$,
introduce the hyperplane $ \BP^{n-3}_a =\{x_a = 0\}\subset \BP^{n-2}$ cut out by 
the corresponding coordinate $x_a$ of the ambient $\BP^{n-1}$.
Introduce  the inclusion of the ``open simplex" given by the complement of these hyperplanes 
$$
\xymatrix{
j:\Delta^{n-2} = \BP^{n-2} \setminus \bigcup_{a=1}^n \BP^{n-3}_a \ar@{^(->}[r] &  \BP^{n-2}
}
$$

Pushforward along $j$ provides a full embedding $\Coh_\torsion(\Delta^{n-2}) \subset \Coh(\BP^{n-2})$ of torsion sheaves supported on 
$\Delta^{n-2} \subset \BP^{n-2}$.

The theorem immediately restricts to an equivalence on full dg subcategories.

\begin{corollary}
There is a 
canonical equivalence 
$$
\xymatrix{
\mu\Sh_{L_c}(M) \ar[r]^-\sim & \Coh_\torsion(\Delta^{n-2})
}
$$
\end{corollary}

To go beyond  torsion sheaves, we can adopt the formalism of wrapped microlocal sheaves introduced in~\cite{Nwms}.
We will not review this notion here but remark that our arguments naturally extend to it and we obtain the following equivalence.

\begin{corollary}
For wrapped microlocal sheaves, there is a 
canonical equivalence 
$$
\xymatrix{
\mu\Sh^w_{L_c}(M) \ar[r]^-\sim & \Coh(\Delta^{n-2})
}
$$
\end{corollary}

\begin{remark}
The two corollaries are related by duality in that the first results from the second by taking
exact functionals to $\Perf_k$ (see~\cite{BNP} for details for coherent sheaves). 
One can think of  
$ \mu\Sh_{L_c}(M)$ as the
infinitesimal Fukaya-Seidel category of the Landau-Ginzburg model,
with branes running along the singular thimble $L_c\subset M$,
and
 $ \mu\Sh^w_{L_c}(M) $ as the 
  partially wrapped Fukaya category  of the Landau-Ginzburg model,
  with branes transverse to $L_c\subset M$.
 \end{remark}

\begin{remark}
The theorem and its corollaries can be viewed as a distinguished instance of homological mirror symmetry
  for hypersurfaces in toric varieties~\cite{AAK}.
The other   Landau-Ginzburg $A$-models arising in the subject can be obtained from 
that of
 $M=\BC^n, W= z_1 \cdots z_n$ by Hamiltonian reduction. 
Thanks to the functoriality of microlocal sheaves,  the theorem and its corollaries
should imply analogous results for them as well.

\end{remark}

Before continuing on, let us mention one other straightforward application of our results.

In the course of our arguments, to any angle $\theta\in S^1$,
we introduce a Lagrangian skeleton $L(\theta) \subset M$ living over the ray
$e^{2\pi i \theta}\cdot \BR_{\geq 0} \subset \BC$. 
For $\theta = 0$,  this is the Lagrangian skeleton introduced above $L(0) = L$.
For $\theta \not = 0$,  we show that $L(\theta) \subset M$ has equivalent microlocal geometry to
$L(0) = L$, via natural monodromy equivalences,
though they are not even homeomorphic.

Now consider the Landau-Ginzburg model with background $M=\BC^n$ as before, but now with superpotential
$W = z_1^r \cdots z_n^r$. Thus its geometry above a cut in the plane $\BC$
is the same as the geometry  of the original superpotential
above $r$ cuts.
Fix a collection  of $r$ angles $\Theta \subset S^1$,
and introduce the
corresponding Lagrangian skeleton 
$$
\xymatrix{
L(\Theta) = \bigcup_{\theta \in \Theta} L(\theta)\subset M
}
$$
In accordance with Ansatz~\ref{ansatz msh},
let us take 
the category of branes in the Landau-Ginzburg $A$-model with background  $M= \BC^n$
and superpotential $W= z^r_1\cdots z^r_n$ 
to be the dg category of microlocal sheaves on $M$ 
supported along $L(\Theta)$.

Our results imply the following generalization of Theorem~\ref{intro main thm}.
To state it, 
let $\cM(r)$ be the dg category of diagrams 
of coherent sheaves
$$
\xymatrix{
i_*M_0 & \ar[l] M_1 & \ar[l] M_2 & \ar[l] \cdots & \ar[l] M_{r-1}
}
$$
where $M_0 \in \Coh(\BP^{n-2})$, $M_1,\ldots, M_{r-1} \in \Coh(\BP^{n-1})$, and
$
i:\BP^{n-2} \to  \BP^{n-1} 
$
is the inclusion of the generic linear hyperplane introduced above.

\begin{thm}
Suppose $r = |\Theta|$. Then there is a 
canonical equivalence
$$
\xymatrix{
\mu\Sh_{L(\Theta)}(M) \ar[r]^-\sim & \cM(r)
}
$$

\end{thm}

\begin{remark}
The theorem  fits naturally into the formalism of perverse schobers discussed
immediately below, in particular the semiorthogonal decompositions of spherical pairs and 
their higher analogues. 
It reflects what one 
expects to find 
by taking the $r$th power of the superpotential
of a Landau-Ginzburg $A$-model with a single critical value: its branes should consist of an $A_{r-1}$-quiver of objects from the nearby category augmented by an object of the vanishing category. 
In the most basic example,  
for the Landau-Ginzburg $A$-model with $M= \BC$ and $W=z^r$ (the case $n=1$ of the theorem),
the vanishing category is trivial, and the nearby category is $\Perf_k$.
Thus its 
branes  form
perfect modules over the $A_{r-1}$-quiver
(for more discussion, see for example~\cite{DK, Ncyc}).
\end{remark}


\subsection{Sketch of arguments}
We outline our arguments here, highlighting the two key notions of perverse schobers and monoidal symmetry. 
They formalize basic principles implicit  in Landau-Ginzburg models and more broadly homological mirror symmetry.
The first encodes the relation between the nearby and vanishing geometry
of branes; the second encodes the
convolution symmetry of branes corresponding  to tensor product  under $T$-duality. 

As outlined above, our starting point  is the restriction functor
$$
\xymatrix{
J^*:\mu\Sh_L(M) \ar[r] & \mu\Sh_{L^\times}(M^\times)
}
$$

Following  Kapranov-Schechtman~\cite{KapS}, we interpret this as part of the diagram defining a perverse schober
on the complex plane $\BC$, with one singular point $0\in \BC$, and a single cut $\BR_{\geq 0} \subset \BC$.
Recall that perverse sheaves on the complex plane $\BC$, with one singular point $0\in \BC$,  are equivalent to
 diagrams of vector spaces 
$$
\xymatrix{
\Phi\ar@<0.75ex>[r]^-p & \ar@<0.75ex>[l]^-q \Psi 
}
$$
such that the  endomorphisms
$$
\xymatrix{
m_\Phi = \id_\Phi - qp
&
m_\Psi = \id_\Psi - pq}
$$
are invertible. The equivalence is given by assigning to a perverse sheaf $\cF$ its nearby and vanishing cycles
$\Psi = \psi_0(\cF)$, $\Phi = \phi_0(\cF)$ at the origin $0\in \BC$ equipped with their canonical maps. In particular,  
 the composite endomorphisms $m_\Psi$, $m_\Phi$ are their monodromy transformations.

Perverse schobers are a categorical analogue of perverse sheaves inspired by the above description.
By definition, as recalled in Section~\ref{ss sph f}, a perverse schober 
on the complex plane $\BC$, with one singular point $0\in \BC$, and a single cut $\BR_{\geq 0} \subset \BC$,
 is simply a spherical functor $$
\xymatrix{
S:\cD_\Phi \ar[r] & \cD_\Psi
}
$$
from a ``vanishing  category" to a ``nearby category".
In one of several equivalent formulations, this means that $S$ fits into an adjoint triple $(S^\ell, S, S^r)$
such that the monodromy functors
$T_{\Phi, r}$, $ T_{\Psi, r}$, $T_{\Psi, \ell}$, $T_{\Phi, \ell}$
defined by the triangles of the units and counits of the adjunctions
$$
\xymatrix{
T_{\Phi, r} = \Cone(u_r)[-1] \ar[r] &  \id_\Phi \ar[r]^-{u_r} & S^r S 
&
SS^r \ar[r]^-{c_r} & \id_\Psi \ar[r] &  \Cone(u_r) = T_{\Psi, r} 
}
$$
$$
\xymatrix{
T_{\Psi, \ell} = \Cone(u_\ell)[-1] \ar[r] &  \id_\Psi \ar[r]^-{u_\ell} & S S^\ell 
&
S^\ell S \ar[r]^-{c_\ell} & \id_\Phi \ar[r] &  \Cone(u_\ell) = T_{\Phi, \ell} 
}
$$
are equivalences. 


Our main technical work is to show that $J^*$ extends to a spherical functor,
in particular that it fits into an adjoint triple $(J_!, J^*, J_*)$. We will highlight below the primary arguments proving this, but it is worth mentioning here
that we do not establish it directly. We rather introduce a larger Lagrangian skeleton living over the real line $\BR\subset \BC$,
 and construct a perverse schober 
 on the complex plane $\BC$, with one singular point $0\in \BC$,
 and the  double cut $\BR\subset \BC$.

By definition, as recalled in Section~\ref{ss sph p}, a perverse schober 
on the complex plane $\BC$, with one singular point $0\in \BC$,
and the double cut $\BR\subset \BC$ is a spherical pair.
Spherical functors and spherical pairs are the cases $n=1$ and $n=2$
of structures one can formulate for any number of cuts in the complex plane. They are categorical analogues of the quiver presentations of perverse sheaves resulting from such cuts. Given the structure for some number of cuts, 
one can naturally form the structure  for another number of cuts. In particular, a spherical pair always gives rise to a spherical functor,
but there is an advantage to a spherical pair: its axioms do not explicitly involve units and counits of adjunctions.
Roughly speaking, from the perspective of Fukaya categories, it  encodes  the pseudo-holomorphic ``teardrops" defining the units and counits without explicitly counting them.

Once we have  that $J^*$ extends to a spherical functor, we may proceed to
monadically calculate  the Landau-Ginzburg vanishing category $\mu\Sh_L(M)$  in terms of the  
known nearby category
$$
\xymatrix{
\mu\Sh_{L^\times}(M^\times)\ar[r]^-\sim & 
\Coh(\BP^{n-1})
}
$$ 

We first need that $J^*$ is conservative, i.e.~that its kernel is trivial, which is an immediate consequence
of dimension bounds for the support of microlocal sheaves. This special property is an analogue of a perverse sheaf
having no sections strictly supported at the origin $0\in \BC$. 

We next calculate that 
the nearby monodromy $T_{\Psi, \ell}$ corresponds to tensoring with the line bundle $\cO_{\BP^{n-1}}(-1)$, and hence
the monad 
$
J^* J_!,
$
presented as the cone of a morphism of functors 
$$
\xymatrix{
T_{\Psi, \ell} \ar[r] &    \id_{\Psi},
}
$$
corresponds to tensoring with the cone of a morphism of line bundles
$$
\xymatrix{
\cO_{\BP^{n-1}}(-1) \ar[r]^-s &  \cO_{\BP^{n-1}}
}
$$
To see that $s$ is indeed a generic morphism, equivalent to 
$s([x_1, \ldots, x_n]) = x_1 + \cdots + x_n$,
we observe that it must be nonzero at each coordinate point 
of $\BP^{n-1}$.
This is a manifestation of the fact that the superpotential $W$ is a submersion at the generic  point of 
each coordinate hyperplane of $M$, and hence the Landau-Ginzburg model vanishes there (see the discussion of the case $n=1$ in Section~\ref{intro low-dim} below).

Finally, we  verify the remaining technical hypotheses
of Lurie's Barr-Beck theorem~\cite{LurieHA}, appealing to the explicit form of the monad described above.

\medskip

Now let us return to the assertion that $J^*$ 
 is a spherical functor, and discuss the key role of symmetry in our arguments.

Recall that we set $T=(S^1)^n$. Let us focus 
on the Hamiltonian $T$-action on $M=\BC^n$ by coordinate rotation.

%

By the formalism of microlocal kernels and transforms developed in~\cite{KS}, one expects constructible sheaves on $T$ to give endofunctors of microlocal sheaves on $M$. To make this precise, we 
must take into account the well-known ``metaplectic anomaly"  appearing for example in identities for the Fourier-Sato transform as encoded by the Maslov index. At the most concrete level,
it reflects the fact that rotating a graded Lagrangian line $\ell$  in the plane $\BC$ by a full circle $2\pi$ will return the same line $\ell$ but with grading shifted by two.

Consider the $\BZ$-cover $T' \to T$ defined by the diagonal character $\delta:T\to S^1$.
There is a canonical lift $T^\circ \subset T'$, since by definition $T^\circ\subset T$
is the kernel of $\delta$, and for concreteness, one can choose an isomorphism $T'\simeq T^\circ \times \BR$ if one likes.
Following~\cite{KS}, the monoidal dg category $\Sh_c(T')$ of constructible sheaves on $T'$ with compact support does indeed act on  
microlocal sheaves on $M$.
But the action does not factor through constructible sheaves on $T$ since
if we translate an object $\cA\in \Sh_c(T')$ by an element  $m \in \BZ \simeq \ker(T'\to T)$,
its action  on microlocal sheaves will be shifted by $[2m]$.

With this in hand, we still must address that the endofunctors given by most objects of $\Sh_c(T')$ do not preserve the support condition given by the Lagrangian skeleton $L\subset M$.
To proceed, we recall that a governing property of the coherent-constructible correspondence is that the equivalence 
$$
\xymatrix{
 \Sh_{\Lambda_\Sigma}( T^{\circ})  \ar[r]^-\sim & \Coh(\BP^{n-1})
}
$$
is symmetric monoidal with respect to convolution and tensor product.
We show that convolution by objects of $\Sh_{\Lambda_\Sigma}( T^{\circ}) $,
regarded as objects of   $\Sh_c(T')$ via the lift $T^\circ\subset T'$,
provides  endfunctors of the nearby and vanishing categories
compatible with the  restriction $J^*$.
%

By construction, the 
 nearby category $ \mu\Sh_{L^\times}(M^\times)$ is a free rank one module over  $\Sh_{\Lambda_\Sigma}( T^{\circ})$.
Thus to define adjoints to $J^*$, it suffices to define their restrictions to a generator
for the monoidal action,
for example, to the microlocal sheaf $\cA\in  \mu\Sh_{L^\times}(M^\times)$
corresponding to the structure sheaf $\cO_{\BP^{n-1}} \in \Coh(\BP^{n-1})$. In our main technical step,
we construct explicit constructible sheaves representing
$J_!\cA, J_*\cA\in  \mu\Sh_{L}(M)$, and confirm the adjunction identities.

Finally, to verify the axioms of a spherical functor and to calculate the monodromy transformations
$T_{\Phi, r}$, $ T_{\Psi, r}$, $T_{\Psi, \ell}$, $T_{\Phi, \ell}$, we appeal to the further symmetry given by a natural multiplicative system of objects $\cA_\tau \in \Sh_c(T')$, indexed by $\tau\in \BR$.
They come equipped with canonical equivalences $\cA_{\tau_1} \star \cA_{\tau_2} \simeq \cA_{\tau_1+\tau_2}$, and $\cA_0$ is  the skyscraper  $k_e$ at the identity $e\in T'$.

Convolution by the multiplicative system  provides a parallel transport of microlocal sheaves on $M$ so that if the support of a microlocal sheaf $\cF$ lies over a cut $e^{i\theta} \BR_{\geq 0} \subset \BC$, then the support of $\cA_\tau\star \cF$ will lie 
over the cut $e^{i(\theta+\tau)} \BR_{\geq 0} \subset \BC$.
When $\tau\in 2\pi \BZ$, we can think of $\cA_\tau$  as the ``convex hull" of the the Dehn twists/Hecke operators for the individual coordinate directions.
It  preserves  
the support condition given by the Lagrangian skeleton $L\subset M$, and
specifically when $\tau = 2\pi$,
enables us to see the monodromy  transformation $T_{\Psi, \ell}$ corresponds to tensor product with $\cO_{\BP^{n-1}}(-1) \in \Coh(\BP^{n-1})$.


\subsection{Low-dimensional cases}\label{intro low-dim}

The one and two--dimensional cases of our results  are well-known and easy to deduce due to the fact that the critical locus $\{dW = 0\}$
is either empty and $W$ is a submersion ($n=1$), or an isolated point and $W$ is Morse ($n=2$).

Nevertheless, we include a brief discussion of these cases
to help guide the interested reader.
At minimum, our general arguments  appeal to the simple geometry appearing in the submersive case ($n=1$),
 and for completeness it is worth
highlighting it here.

\subsubsection{Submersive case: $n=1$} 
The Landau-Ginzburg $A$-model with $M= \BC$ and $W=z^r$, for any $r\geq 1$, is well understood: 
its dg category of branes is equivalent to perfect modules over the $A_{r-1}$-quiver
(for further discussion, see for example~\cite{DK, Ncyc}).
In particular, in our situation where $r=1$, its branes form the zero category,
reflecting the fact that a submersion should not have any nontrivial vanishing geometry. 

In the setting of microlocal sheaves, it is easy to see  that the branes form  the zero category.
Our Lagrangian skeleton is the closed non-negative real ray $L = \BR_{\geq 0} \subset \BC = M$,
or alternatively any closed ray emanating from the origin $0\in M$.
Thus we expect the vanishing category to be the zero category  $\mu\Sh_L(M)= 0$, since no nontrivial microlocal sheaves
have support a manifold with nonempty boundary. 

%
%
To verify this, let us say more precisely
what we mean by microlocal sheaves.
We understand microlocal sheaves on $M= \BC$ to be microlocal sheaves on the conic open ball
$$
\xymatrix{
\Omega = \{(x, t), (\xi, \eta) \, |\, \eta>0\} \subset T^*\BR^2 
}
$$
obtained by taking the symplectification of the contactification of $M$.
More specifically, we understand microlocal sheaves supported along a conic Lagrangian subvariety $R \subset M$,
so by necessity a finite union of closed rays emanating from the origin $0\in M$,
to be microlocal sheaves on  the associated conic Lagrangian surface
$$
\xymatrix{
\Lambda_R = \{(x, xy), (-\eta y, \eta) \, |\, x+ i y \in R, \, \eta >0\} \subset \Omega
}
$$
obtained by trivially lifting $R$ to a Legendrian in the contactification and then taking its inverse-image
in the symplectification.
Note that since $R \subset M$ is invariant under scaling,  $\Lambda_R \subset \Omega$ is invariant under the additional Hamiltonian  scaling
$$
\xymatrix{
r\cdot ((x, t), (\xi, \eta)) = ((rx, r^2 t), (r^{-1} \xi, r^{-2} \eta))
&
r\in \BR_{>0}
}
$$
Therefore all of the structure of such microlocal sheaves is captured in a small conic neighborhood of the 
central codirection
$\{((0, 0), (0, \eta)) \,|\, \eta>0\} \subset \Omega$.

Now starting with the Lagrangian skeleton given by the closed non-negative real ray $L = \BR_{\geq 0} \subset \BC = M$, we arrive at  the   conic Lagrangian surface
$$
\xymatrix{
\Lambda= \{(x, 0), (0, \eta) \, |\, x\geq 0, \eta >0\} \subset \Omega
}
$$
Note that $\Lambda$  is diffeomorphic to the manifold with nonempty boundary $\BR_{\geq 0} \times \BR_{>0}$, and so indeed the vanishing category $\mu\Sh_L(M)$,
realized in the form $\mu\Sh_\Lambda(\Omega)$, is the zero category.

Alternatively, starting
with the closed conic ray $L(\theta) = e^{i\theta}\cdot \BR_{\geq 0} \subset \BC = M$,  with $\theta\not = \pm \pi/2$, 
of slope $c =  \sin(\theta) /\cos(\theta)$ and
 horizontal direction $d = \cos(\theta)/|\cos(\theta)|$, 
   the associated  conic Lagrangian surface takes the form
$$
\xymatrix{
\Lambda (\theta)=  \{(x,  c x^2), (-\eta cx, \eta) \, |\, dx\geq 0, \, \eta>0\} \subset \Omega
}
$$
In the special case $\theta = \pm \pi/2$, when 
 $L(\pm\pi/2) = \pm i \cdot \BR_{\geq 0} \subset \BC = M$, 
   the associated  conic Lagrangian surface
takes the form
$$
\xymatrix{
\Lambda (\pm\pi/2)=  \{(0, 0), ( \mp \eta y, \eta) \, |\, y\geq 0,  \, \eta>0\} \subset \Omega
}
$$
Rotations of $M=\BC$ identify the original rays and corresponding rotations of $\Omega$ identify the associated 
conic Lagrangian surfaces. 
In any case, each  $\Lambda(\theta)$  is diffeomorphic to $\BR_{\geq 0} \times \BR_{>0}$, and so the vanishing category $\mu\Sh_L(M)$,
realized in the form $\mu\Sh_{\Lambda(\theta)}(\Omega)$, is the zero category.

\begin{remark}  
For the Landau-Ginzburg $A$-model with $M= \BC$ and $W=z^r$, for any $r\geq 1$, we could take as Lagrangian skeleton the union of $r$ closed rays.

For example, for $r=2$, we could take the union $L_\BR = L(0) \cup L(\pi)$ 
of the two real rays, and work with microlocal sheaves on the associated conic Lagrangian surface 
$$
\xymatrix{
\Lambda_\BR =\Lambda(0) \cup \Lambda(\pi)
= \{(x, 0), (0, \eta) \, |\,  \eta>0\} \subset \Omega
}
$$
Note that $\Lambda_\BR$  is diffeomorphic to the manifold $\BR \times \BR_{>0}$, and thus 
the vanishing category takes the expected form
$\mu\Sh_{L_\BR}(M) \simeq \Perf_k$.

Alternatively, we could take the union $L_{i\BR} = L(\pi/2) \cup L(-\pi/2)$
 of the two imaginary rays,
  and work with microlocal sheaves on the associated conic Lagrangian surface 
$$
\xymatrix{
\Lambda_{i\BR} =\Lambda(\pi/2) \cup \Lambda(-\pi/2)
= \{(0, 0), (y, \eta) \, |\,  \eta>0\} \subset \Omega
}
$$
Rotation of $M= \BC$ by $\pi/2$ takes $L_\BR$ to $L_{i\BR}$
and
a corresponding rotation of $\Omega$ takes  $\Lambda_\BR$ to $\Lambda_{i\BR}$.
This leads to a natural  Fourier-Sato type equivalence  of the vanishing categories
$$
\xymatrix{
\mu\Sh_{L_\BR}(M) \ar[r]^-\sim & \mu\Sh_{L_{i\BR}}(M)
}
$$

Going further, rotation by $\pi$   leads to iterating the above equivalence twice, and results in the auto-equivalence of
$\mu\Sh_{L_\BR}(M) \simeq \Perf_k$ given by tensoring with the invertible shifted orientation line 
$\orient_{L_\BR}[1]$.
Rotation by $2\pi$ leads to iterating it four times, and thus results in the auto-equivalence given by the 
shift~$[2]$ alone. This is the most basic instance of the ``metaplectic anomaly" found in the monodromy
of the vanishing category.
%
%
%
%
%
%

\end{remark}


\subsubsection{Morse case: $n=2$}

When $M=\BC^2$ and $W = z_1 z_2$, the dg category of the Landau-Ginzburg $A$-model
 will be equivalent to perfect modules $\Perf_k$.
This reflects the fact that a single Morse critical point has a smooth vanishing thimble
and otherwise is a submersion.

Following our general constructions, we work with a natural symplectic identification of the nearby fiber $M_1 = W^{-1}(1) \simeq \BC^\times$  with the cotangent bundle $T^* S^1$ of the vanishing circle.
We start with the Lagrangian skeleton $L_1 \subset M_1$ given by the union $T^*_\cS S^1 \subset  T^*S^1$ of the conormal bundles  of the stratification $\cS$ by the point $0\in S^1$ and its complement $S^1 \setminus \{0\}$.
We then take the Lagrangian skeleton  $L \subset   M$ to be the closure of the  positive real scalings of $L_1 \subset M_1$. 

Away from the vanishing thimble $L_c \subset L$, given by the cone over the vanishing circle $S^1 \subset L_1$, the Lagrangian skeleton $L \subset M$ is diffeomorphic to the manifold with boundary $\BR_{\geq 0 } \times (L_1 \setminus S^1)$.
Thus any microlocal sheaf supported along $L\subset M$ must be trivial away from $L_c \subset L$.
In fact, if we start with an arbitrary Lagrangian skeleton $L_1\subset  M_1$, 
and similarly form the Lagrangian skeleton  $L \subset   M$, the same argument will apply:
since the superpotential is
  a submersion away from $L_c\subset M$, we will find that $L\setminus L_c$ is diffeomorphic to
  $\BR_{\geq 0} \times (L_1 \setminus S^1)$.
    Thus any microlocal sheaf supported along $L\subset M$ must be trivial away from $L_c \subset L$,
  and we can assume that 
 $L_1\subset  M_1$ reduces to the vanishing circle alone,
 so that $L\subset M$ is simply the vanishing thimble. 

Finally, the  vanishing thimble itself $L_c \subset L$ is   diffeomorphic to $\BR^2$,
and
so indeed the vanishing category admits the expected description  $\mu\Sh_L(M)\simeq \Perf_k$.
Let us place this within the mirror equivalence for the nearby category
$$
\xymatrix{
\mu\Sh_{L_1}(M_1) \simeq\Sh_{\cS}(S^1)\ar[r]^-\sim &  \Coh(\BP^1)
}
$$
More specifically, let us discuss how the natural restriction  
$$
\xymatrix{
\mu\Sh_L(M)\ar[r]^-\sim & \mu\Sh_{L_1}(M_1) 
}
$$
corresponds to the pushforward
$$
\xymatrix{
i_*:\Perf_k \simeq \Coh(pt) \ar[r] & \Coh(\BP^1)
}
$$
along the inclusion $i:pt \to \BP^1$ of a point not equal to $0, \oo\in \BP^1$.

First, let us take a direct approach available in this dimension. Under the mirror equivalence for the nearby category,
skyscraper sheaves at points $\lambda\in \G_m = \BP^1\setminus\{0, \infty\}$ correspond to rank $1$,
monodromy $\lambda$ local systems on the vanishing circle $S^1 \subset M_1$. 
And rank $1$ local systems on the vanishing thimble $L_c \subset M$ restrict to trivial rank $1$ local systems on the vanishing circle $S^1 \subset T^*S^1 \simeq M_1$. Thus under suitable
conventions for choosing the  equivalence
$\mu\Sh_{L_1}(M_1) \simeq\Sh_{\cS}(S^1),
$
the inclusion $i:pt \to \BP^1$ will be of the point $1\in\G_m = \BP^1\setminus\{0, \infty\}$

In higher dimensions, we will invoke a generalization of the following argument.
Under the coherent-constructible equivalence for the nearby category,
the restriction of a microlocal sheaf 
 to the non-zero locus of the conormal line $T^*_{0} S^1 \subset T^* S^1 \simeq  M_1$ corresponds to the restriction
of a coherent sheaf to the points $0, \oo\in \BP^1$. Since any object of the vanishing category
must be trivial away from the vanishing thimble $L_c \subset L$,  in particular it must be trivial  along the  non-zero locus of the conormal line $T^*_0 S^1 \subset T^* S^1 \simeq  M_1$. Thus the inclusion $i:pt \to \BP^1$ must be
 of a point not equal to $0, \oo\in \BP^1$.


\subsection{Acknowledgements}
I  thank  D. Auroux, D. Ben-Zvi, M. Kontsevich, J. Lurie, N. Rozenblyum, V. Shende, N. Sheridan, D. Treumann, H. Williams,
and E. Zaslow for their interest, encouragement, and valuable comments.
Finally, I am grateful to the NSF for the support of grant DMS-1502178.


\section{Perverse Schobers on a disk}\label{s ps}

This section is a synopsis of some of the theory proposed by Kapranov-Schechtman~\cite{KapS}.
In particular, we recall the notion of a perverse schober in its appearances as a spherical functor and  spherical pair.


\subsection{Single cut: spherical functors}\label{ss sph f}

Let $\cD_\Phi, \cD_\Psi$ be pre-triangulated dg categories. 

Suppose given a dg functor
$$
\xymatrix{
S:\cD_\Phi \ar[r] & \cD_\Psi
}
$$ 
that admits both a left and right adjoint so that we have adjunctions 
$
(S^\ell, S)
$
and
$
(S, S^r)
$
with units and counits 
$$
\xymatrix{
u_r:\id_\Phi\ar[r] & S^r S & c_r: S S^r \ar[r] & \id_\Psi
}
$$
$$
\xymatrix{
u_\ell:\id_\Psi\ar[r] & S S^\ell & c_\ell: S^\ell S \ar[r] & \id_\Phi 
}
$$

Form the natural triangles of functors
$$
\xymatrix{
T_{\Phi, r} = \Cone(u_r)[-1] \ar[r] &  \id_\Phi \ar[r]^-{u_r} & S^r S 
&
SS^r \ar[r]^-{c_r} & \id_\Psi \ar[r] &  \Cone(u_r) = T_{\Psi, r} 
}
$$
$$
\xymatrix{
T_{\Psi, \ell} = \Cone(u_\ell)[-1] \ar[r] &  \id_\Psi \ar[r]^-{u_\ell} & S S^\ell 
&
S^\ell S \ar[r]^-{c_\ell} & \id_\Phi \ar[r] &  \Cone(u_\ell) = T_{\Phi, \ell} 
}
$$
\begin{defn}
We call $S:\cD_\Phi \to \cD_\Psi$ a {\em spherical functor} if it satisfies:

(SF1) $T_{\Psi, r}$ is an equivalence.

(SF2) The natural composition
$$
\xymatrix{
S^r \ar[r] & S^r S S^\ell \ar[r] & T_{\Phi, r} S^\ell[1]
}
$$
is an equivalence.
\end{defn}

\begin{remark} Consider the additional conditions: 

 (SF3) $T_{\Phi, r}$ is an equivalence.

(SF4) The natural composition
$$
\xymatrix{
S^\ell  T_{\Psi, r}[-1]  \ar[r] & S^\ell S S^r \ar[r] & S^r
}
$$
is an equivalence.

A theorem of Anno-Logvinenko~\cite{AL} establishes that any two of the conditions  (SF1) -- (SF4) imply the other two.
\end{remark}

\begin{remark}
For a spherical functor, 
 $T_{\Phi, \ell}$, $T_{\Psi, \ell}$ are 
 respective inverses of 
  $T_{\Phi, r}$, $T_{\Psi, r}$.
\end{remark}

\begin{example}[Smooth hypersurfaces]\label{ex hyper}
Let $X$ be a smooth variety. Let $\cL_X \to  X$ be a line bundle and $\sigma:X\to \cL_X$ a section transverse to the zero section. Let $Y = \{\sigma = 0\}$ be the resulting smooth hypersurface and $i:Y \to X$ its  inclusion.

Let $\Coh(Y)$, $\Coh(X)$  denote the respective dg categories of coherent sheaves.
We will check that the pushforward $i_*:\Coh(Y) \to \Coh(X)$ is a spherical functor.

Regard the line bundle $\cL_X$ as an object of $\Coh(X)$,
and its restriction $\cL_Y = i^*\cL_X$ as an object of $\Coh(Y)$.
Regard the section $\sigma$ as a morphism $\sigma:\cO_X\to \cL_X$,
which by duality gives a morphism $\sigma^\vee:\cL_X^\vee\to \cO_X$.

Consider the natural adjunctions 
$$
\xymatrix{
i_*:\Coh(Y) \ar@{<->}[r] &\Coh(X):i^!
&
i^*:\Coh(X) \ar@{<->}[r] & \Coh(Y):i_*
}
$$
Note  the functorial identities
$$
\xymatrix{
i^*(-) \simeq \cO_Y \otimes_{\cO_X} (-)  
&
i^!(-) \simeq  \cL_Y[-1]\otimes_{\cO_X} (-) 
}
$$

The natural triangles of functors associated to the units and counits of the adjunctions
are given by tensoring with the respective triangles of objects
$$
\xymatrix{
  \cL_Y[-2] \ar[r]^-0 &  \cO_Y  \ar[r] &
  \cO_Y \oplus \cL_Y[-1]
&
 \cL_Y[-1] \ar[r] &  \cO_X  \ar[r]^-\sigma &
 \cL_X
}
$$
$$
\xymatrix{
\cL_X^\vee \ar[r]^-{\sigma^\vee} &  \cO_X  \ar[r] & \cO_Y
&
 \cO_Y \oplus \cL_Y[1] \ar[r]  &  \cO_Y 
 \ar[r]^-0 &   \cL_Y[2]
}
$$

Thus if we set $\cD_\Phi = \Coh(Y)$,  $\cD_\Psi= \Coh(X)$ and $S= i_*$,  we find that
$$
\xymatrix{
T_{\Psi, r}(-)  \simeq  \cL_X \otimes_{\cO_X} (-)  &  
T_{\Phi, r}(-)  \simeq  \cL_Y[-2] \otimes_{\cO_Y}(-)
}
$$
are both equivalences. Thus  (SF1) and  (SF3) hold so that $S=i_*$ is a spherical functor.
\end{example}


\subsection{Double cut: spherical pairs}\label{ss sph p}

\subsubsection{Semi-orthogonal decompositions}

Let $\cA$ be a pre-triangulated dg category, and $\cB \subset \cA$ a full pre-triangulated dg subcategory.

Let us denote by $J:\cB\to\cA$ the embedding. 
Introduce the full dg subcategories of left and right orthogonals
$$
\xymatrix{
{}^\perp \cB = \{A\in \cA \, |\, \Hom_\cA(A, B) \simeq 0, \mbox{ for all } B\in \cB\}
}
$$
$$
\xymatrix{
 \cB^\perp = \{A\in \cA \, |\, \Hom_\cA(B, A) \simeq 0, \mbox{ for all } B\in \cB\}
}
$$

One says that $\cB$ is left admissible, respectively right admissible, if $J$ admits a left adjoint $J^\ell:\cA\to \cB$,
respectively a right adjoint $J^r:\cA\to \cB$. If either holds, then we have the corresponding identity  ${}^\perp \cB = \ker(J^\ell)$,
respectively ${}^\perp \cB = \ker(J^r)$.
Moreover, 
we have a corresponding semi-orthogonal decomposition in the sense of a functorial triangle
$$
\xymatrix{
C = \Cone(u)[-1] \ar[r] & A \ar[r]^-u & J^\ell J A = B & B\in \cB, C\in {}^\perp \cB
}
$$
$$
\xymatrix{
B' = J J^rA \ar[r]^-c & A \ar[r] &   \Cone(c) = C' &   B'\in \cB,  C'\in  \cB^\perp 
}
$$

Note that if $\cB$ is left admissible, then ${}^\perp \cB$ is right admissible and $({}^\perp \cB)^\perp = \cB$.
Similarly, $\cB$ is right admissible, then $ \cB^\perp$ is left admissible and ${}^\perp (\cB^\perp) = \cB$. 


\subsubsection{Spherical pairs}

Suppose we have a diagram of pre-triangulated dg categories
$$
\xymatrix{
\cD^\circ_- & \ar[l]_-{J_{-}^*} \cD \ar[r]^-{J_+^*}&  \cD^\circ_+ 
}
$$

Suppose further that $J_-^*, J_+^*$ admit fully faithful left and right adjoints
so that we have adjoint triples
$$
\xymatrix{
 (J_{-!}, J_-^*, J_{-*})
&
(J_{+!}, J_+^*, J_{+*})
}
$$
Thus we have the right admissible  dg subcategories
$$
\xymatrix{
\cD^\circ_{-!} = J_{-!}(\cD_-^\circ)
&
\cD^\circ_{+!} = J_{+!}(\cD_+^\circ)
}
$$
and the left admissible dg subcategories
$$
\xymatrix{
\cD^\circ_{-*} = J_{-*}(\cD_-^\circ)
&
\cD^\circ_{+*} = J_{+*}(\cD_+^\circ)
}
$$

Introduce the dg subcategories
$$
\xymatrix{
\cD_- = \ker(J_+^*) = {}^\perp(\cD^\circ_{+*})
= (\cD^\circ_{+!})^\perp
&
\cD_+ = \ker(J_-^*)  = {}^\perp(\cD^\circ_{-*})
= (\cD^\circ_{-!})^\perp
}
$$
with embeddings denoted by 
$$
\xymatrix{
\cD_-  \ar[r]^-{I_{-!}} &  \cD & \ar[l]_-{I_{+*}}  \cD_+ 
}
$$
Note that $\cD_-, \cD_+$ are left and right admissible so that  we have adjoint triples
$$
\xymatrix{
 (I_{-}^*, I_{-!}, I_{-}^!)
&
(I^{*}_+, I_{+!}, I_{+}^!)
}
$$
and further that
$$
\xymatrix{
\cD_{-*}^\circ = (\cD_+)^\perp 
&
\cD_{-!}^\circ = {}^\perp(\cD_+)
&
\cD_{+*}^\circ = (\cD_-)^\perp 
&
\cD_{+!}^\circ = {}^\perp(\cD_-)
}
$$

\begin{defn}
 A {\em spherical pair} is a diagram
$$
\xymatrix{
\cD^\circ_- & \ar[l]_-{J_{-}^*} \cD \ar[r]^-{J_+^*}&  \cD^\circ_+ 
}
$$
of functors admitting fully faithful left and right adjoints so that:

 (SP1)
The compositions 
$$
\xymatrix{
J_+^* J_{-*}:  \cD^\circ_-\ar[r] &  \cD^\circ_+
&
J^*_- J_{+*}:  \cD^\circ_+\ar[r] &  \cD^\circ_-
}
$$
are equivalences.

 (SP2) The compositions 
$$
\xymatrix{
I^!_+ I_{-!}: \cD_-\ar[r] & \cD_+
&
I_-^! I_{+!}: \cD_+\ar[r] & \cD_-
}
$$
are equivalences.

\end{defn}

\begin{remark}
If the compositions of (SP1) are equivalences, their respective  inverses are given by the adjoint compositions
$$
\xymatrix{
J_-^* J_{+!}:  \cD^\circ_+\ar[r] &  \cD^\circ_-
&
J^*_+ J_{-!}:  \cD^\circ_-\ar[r] &  \cD^\circ_+
}
$$
and similarly 
if the compositions of (SP2) are equivalences, their respective  inverses are given by the adjoint compositions
$$
\xymatrix{
I_-^* I_{+!}:  \cD_+\ar[r] &  \cD_-
&
I^*_+ I_{-!}:  \cD_-\ar[r] &  \cD_+
}
$$\end{remark}

\begin{lemma}\label{lemma cons sph pair}
Suppose  the compositions
$$
\xymatrix{
J_{-}^*I_{-!}:\cD_-\ar[r] & \cD_-^\circ
&
J_{+}^*I_{+!}:\cD_+\ar[r] & \cD_+^\circ
}
$$ 
are conservative. Then (SP1) implies (SP2).
\end{lemma}

\begin{proof}
Let $\cG \in \cD_+$. We will construct a functorial equivalence
$$
I_+^! I_{-!}I^*_-I_{+!}\cG  
\simeq  \cG  
$$
and leave the other parallel equivalences to the reader.

Let $\cF \in \cD$. By assumption, we have a triangle
$$
\xymatrix{
J_{+!}J^*_+\cF\ar[r] & \cF\ar[r] & I_{-!} I^*_-\cF
}
$$
and so can view $I_{-!}I^*_-\cF$ as the complex 
$$
\xymatrix{
J_{+!}J^*_+\cF[1] \ar[r] & \cF
}
$$
Again by assumption, we have a triangle
$$
\xymatrix{
  J_{-*}J_-^*J_{+!}J^*_+\cF[1] \ar[r] & J_{-*}J_-^* \cF \\ 
 \ar[u] J_{+!}J^*_+\cF[1] \ar[r] & \cF\ar[u]  \\
\ar[u] I_{+!}I_+^! J_{+!}J^!_+\cF[1] \ar[r] & I_{+!}I_+^! \cF\ar[u]
}
$$
and so can view $I_{+!}I_+^! I_{-!}I^*_-\cF$ as the total  complex 
$$
\xymatrix{
  J_{-*}J_-^*J_{+!}J^*_+\cF \ar[r] & J_{-*}J_-^* \cF[-1] \\ 
 \ar[u] J_{+!}J^*_+\cF[1] \ar[r] & \cF \ar[u]  
}
$$

Now set $\cF = I_{+!}\cG \in \cD_+$.
Then $J^*_-\cF \simeq J^*_-I_{+!}\cG \simeq 0$, so we can view 
$I_{+!}I_+^! I_{-!}I^*_-\cF$ as the  total complex 
$$
\xymatrix{
  J_{-*}J_-^*J_{+!}J^*_+\cF \ar[r] & 0 \\ 
 \ar[u] J_{+!}J^*_+\cF[1] \ar[r] & \cF \ar[u]  
}
$$

Since the total complex $I_{+!}I_+^! I_{-!}I^*_-\cF$ and the right vertical complex $\cF =  I_{+!}\cG$ both result from applying $I_{+!}$ to an object of $\cD_{+}$, the left vertical complex does as well.
Since applying $J_+^*$ to the left vertical arrow produces an equivalence, the left vertical arrow must already be an equivalence since $J_+^* I_{+!}$ is conservative and $I_{+!}$ is fully faithful. Thus the total complex collapses to $\cF$ itself, and we arrive at the sought-after equivalence
$$
I_{+!}I_+^! I_{-!}I^*_-I_{+!}\cG  =I_{+!} I_+^! I_{-!}I^*_-\cF  \simeq \cF =  I_{+!}\cG
$$
\end{proof}

\begin{remark}
We call a spherical pair conservative
if the compositions of the above lemma are conservative.
A conservative spherical pair is an analogue of a perverse sheaf with no sections strictly supported at the origin.
\end{remark}


\subsubsection{From spherical pairs to spherical functors}

Given a spherical pair, introduce the diagram of pre-triangulated dg categories
$$
\xymatrix{
S = J^*_+|_{\cD_+}:\cD_\Phi = \cD_+ \ar[r] & \cD^\circ_+ =\cD_\Psi
}
$$

Kapranov-Schechtman~\cite[Proposition~3.8]{KapS} prove the following.

\begin{prop}\label{prop KS}
$S$ is a spherical functor with 
$$
\xymatrix{
T_{\Psi,\ell} \simeq J_+^* J_{-!}J_-^*J_{+!}&
T_{\Psi,r} \simeq J_+^* J_{-*}J_-^*J_{+*}
}$$
$$
\xymatrix{
T_{\Phi,\ell} \simeq I_+^! I_{-!}I_-^!I_{+!}&
T_{\Phi,r} \simeq I_+^* I_{-!}I_-^*I_{+!}
}$$\end{prop}

\begin{remark}
Note if we start with a conservative spherical pair, then the resulting spherical functor is conservative.
 \end{remark}


\section{Geometry of $M=\BC^n$, $W= z_1\cdots z_n$}\label{s geom}


\subsection{Preliminaries}\label{ss prelim}

Let $M=\BC^n$ with coordinates $z_a = x_a + i y_a = r_a e^{i \theta_a}$,
for  $a= 1, \ldots, n$.

Equip $M$ with the exact  symplectic form 
$$
\xymatrix{
\displaystyle
\omega_M = \sum_{a=1}^n dx_a dy_a  =  \sum_{a=1}^n r_a dr_a d\theta_a   
}
$$ with primitive 
$$
\xymatrix{
\displaystyle
\alpha_M = \frac{1}{2} \sum_{a=1}^n (x_a dy_a - y_a dx_a) = \frac{1}{2} \sum_{a=1}^n r_a^2 d\theta_a 
}
$$
and  Liouville vector field 
$$
\xymatrix{
\displaystyle
v_M = \frac{1}{2}\sum_{a=1}^n (x_a \partial_{x_a} + y_a \partial_{y_a}) = \frac{1}{2}\sum_{a=1}^n r_a \partial_{r_a}
}
$$ 
characterized by $i_{v_M}\omega_M = \alpha_M$.

We will refer to the above as the {\em conic} exact symplectic structure on $M$.
Note that the Liouville vector field $v_M$ generates the positive real scalings of $M$ as a vector space.

\begin{remark}
It is not particularly significant whether  we work with the above symplectic structure $\omega_M$ 
or its  opposite 
$
-\omega_M = \sum_{a=1}^n dy_a dx_a  = - \sum_{a=1}^n r_a dr_a d\theta_a
$
since they are exchanged by complex conjugation. There is a modest inconvenience that
$\omega_M$ is compatible with the natural  identification  of $M^\times = (\BC^\times)^n$  
with (an open subspace of) $T^*(S^1)^n$,
while $-\omega_M$ is compatible with the natural identification of $M =\BC^n$  
with $T^*\BR^n$.
%
%
\end{remark}

By a {\em Lagrangian subvariety} $L\subset M$, we will mean a real analytic subvariety  
of pure dimension $n$
such that
the restriction of $\omega_M$ to any submanifold contained within $L$ vanishes.
By an {\em exact Lagrangian subvariety} $L\subset M$, we will mean a Lagrangian subvariety 
that admits a continuous function
$f:L\to \BR$ such that
the restriction of $f$ to any submanifold contained within $L$ is differentiable and a primitive for the restriction of $\alpha_M$.
By a {\em conic Lagrangian subvariety} $L\subset M$, we will mean a Lagrangian subvariety 
invariant under positive real scalings.
Note that any conic Lagrangian subvariety is exact with primitive any constant function. 

\subsubsection{Summary}

In what follows, we record some standard constructions tuned to  our current setting.  
Our aim is to place $M = \BC^n$, with its  given exact symplectic structure, within the microlocal geometry of $X=\BR^n$.

We  first introduce the contactification $N = M\times \BR$, and then identify  it with the one-jet bundle $JX = T^*X \times \BR$,
compatibly with the natural projections to $X\times \BR$.
We  then observe that the symplectification of $JX = T^*X \times \BR$ is equivalent to an 
 open conic subspace $\Omega_X \subset T^*( X\times \BR)$, compatibly with the natural projections to $X\times \BR$.
 
 The symplectification and contactification come with natural maps
 $$
 \xymatrix{
 \Omega_X \ar[r]^-s &  JX \simeq N\ar[r]^-c &  M
 }
 $$
Given an exact Lagrangian subvariety $L\subset M$ with primitive $f:L\to \BR$, we can  lift it along $c$ to a Legendrian graph $\Gamma_{L, -f} \subset N$, then
transport it to $JX$, and finally take its inverse-image under $s$
 to arrive at a conic Lagrangian subvariety
$\Lambda \subset \Omega_X$.
In this way, we will be able to apply the tools of microlocal geometry to study the given exact  symplectic geometry.

\begin{remark}\label{rem symp of cont}
In what follows, we set  conventions so that taking  the symplectification of the contactification of  a conic open subspace $\Omega_Z \subset T^*Z$  produces again such a conic open subspace 
$$
\xymatrix{
\Omega_{Z}'  =\{((z, t), (\zeta, \eta)\, |\, (z, \zeta) \in \Omega_Z, \, \eta>0\} \subset T^*(Z\times \BR)
}
$$
Therefore given a conic Lagrangian subvariety $\Lambda\subset \Omega_Z$, the associated conic Lagrangian subvariety $\Lambda'\subset \Omega_{Z}'$ will have 
equivalent microlocal geometry.
 \end{remark}


\subsubsection{Contactification}
Given an exact symplectic manifold $M$, with symplectic form $\omega_M$, and primitive $d\alpha_M = \omega_M$, we will take its contactification
to be the contact manifold $N = M\times \BR$, with contact form $\lambda_N = dt + \alpha_M$,
and contact structure $\xi_N = \ker (\lambda_N)$. Here and in what follows, we often write $t$ for a coordinate on $\BR$. (The choice of $\lambda_N = dt + \alpha_M$ rather than $dt - \alpha_M$ is in the name of the consistency
mentioned in Remark~\ref{rem symp of cont}.)

\medskip

Let us return to specifically $M= \BC^n$ with its conic exact symplectic structure.

Consider the contactification $N = M\times \BR = \BC^n \times \BR$, 
 with the contact form
$$\xymatrix{
\displaystyle
\lambda_N = dt + \alpha_M  = 
dt +  \frac{1}{2} \sum_{a=1}^n (x_a dy_a - y_a dx_a)  =  dt + \frac{1}{2} \sum_{a=1}^n r_a^2 d\theta_a 
}
$$
 and  cooriented contact structure 
 $$
 \xymatrix{
\xi_N = \ker(\lambda_N) \subset TN
}
$$

By a {\em Legendrian subvariety} $\cL\subset N$, we will mean a real analytic subvariety  
of pure dimension $n$
such that
any submanifold contained within $\cL$ is tangent to the contact structure $\xi_N$.

Note that an exact Lagrangian subvariety $L\subset M$ equipped with a primitive $f:L\to \BR$ lifts to a Legendrian
graph  
$$
\xymatrix{
\Gamma_{L, -f} =\{(x, -f(x)) \, | \, x\in L\}  \subset M \times \BR = N
}
$$
In particular, 
a conic Lagrangian submanifold $L\subset M$ lifts  
to the trivial
graph 
$$
\xymatrix{
 \Gamma_{L, 0} = L \times \{0\} \subset M \times \BR = N
}
$$


\subsubsection{Identification with one-jets}
Let $X$ be an $n$-dimensional smooth manifold. 

Let $\pi_X:T^* X \to X$ be the cotangent bundle, with points denoted by pairs $(x, \xi) \in T^* X$ with $x\in X$ a point, and $\xi \subset T^*_x X$ a covector.
We will equip $T^*X$
 with its canonical one-form 
 $$\xymatrix{
\displaystyle
\alpha_X =    \sum_{a=1}^n \xi_a d x_a
}
$$
and symplectic form
$$
\xymatrix{
\displaystyle
\omega_X = d\alpha_X =     \sum_{a=1}^n d\xi_a d x_a
}
$$
 Recall that the graph $\Gamma_{df} \subset T^* X$ of the differential of a  function $f:X\to \BR$ is an
 exact Lagrangian submanifold with canonical primitive $f \circ\pi_X |_{\Gamma_{df}}:\Gamma_{df}\to \BR$.

Let $J X = T^* X \times \BR \to X$ be the one-jet bundle,
with points denoted by triples $(x, \xi, t) \in J X$ with $(x, \xi) \in T^* X$ a point and covector, and $t\in \BR$ a number.
 We will equip $J X$
 with its canonical contact form
 $$\xymatrix{
\displaystyle
\lambda_X =   dt- \alpha_X = dt -  \sum_{a=1}^n \xi_a x_a
}
$$
 and   cooriented contact structure 
$$
\xymatrix{
\xi_X = \ker(\lambda_X) \subset  TJ X
}
$$
Recall that the one-jet $J_f \subset J X$ of a function $f:X\to \BR$ is a Legendrian submanifold.

Note that by our conventions, the diffeomorphism
$$
\xymatrix{
J X \ar[r]^-\sim & JX
&
(x, \xi, t) \ar@{|->}[r] & (x, -\xi, t) 
}
$$
intertwines the canonical contact form $\lambda_X= dt- \alpha_X$ with
the contact form $ dt+ \alpha_X$ arisingon $JX$ as the contactification of $T^*X$
 following our conventions.

 \medskip

Now  set $X=\BR^n$ with coordinates $x_a$, for $a=1, \ldots, n$.

Consider the linear Lagrangian fibration given by taking real parts
$$
\xymatrix{
p: M = \BC^n \ar[r] & \BR^n = X & p(z_1, \ldots, z_n) = (x_1, \ldots, x_n)
}
$$
Note that $p$ is equivariant for real scalings and invariant under conjugation.
There is a unique lift  to a Legendrian fibration 
$$
\xymatrix{
\displaystyle
q:  N = \BC^n \times \BR  \ar[r] & \BR^n \times \BR  = X\times \BR & q(z_1, \ldots, z_n, t) = (x_1, \ldots, x_n, t
+\frac{1}{2} \sum_{a=1}^n x_a y_a)
}
$$
such that the last component of $q$ vanishes on Legendrian lift $\BR^n \times\{0\} \subset \BC^n \times \BR = N$
of the real subspace $\BR^n\subset \BC^n = M$ regarded as a section of $p$.
Note that $q$ is equivariant for simultaneous real scalings of the $x, z$ components and 
squared real scalings of the last components. It is also equivariant for simultaneous  conjugation 
of the $z$ components and 
negation of the last component. 

There is a cooriented contactomorphism
 $$
 \xymatrix{
 N \ar[r]^-\sim & J X
&
(z_1, \ldots, z_n, t) \ar@{|->}[r] & ((x_1, \ldots, x_n), (y_1, \ldots, y_n), t+\frac{1}{2} \sum_{a=1}^n x_a y_a)
}
 $$
 intertwining the Legendrian projection $q:N\to X\times \BR$ and front projection $J X \to X\times \BR$.
 Note that it is equivariant for simultaneous real scalings of the $x, y, z$ components and 
squared real scalings of the last components. It is also equivariant for  simultaneous conjugation of the $z$ components and negation
of the $y$ components and last components.

 \begin{remark}
 The above  contactomorphism is an 
 instance of the  general observation: given two primitives $d\alpha_M = \omega_M$, $d\alpha_M'= \omega_M$ for a symplectic form 
 on a manifold $M$,
 if the difference $\alpha_M - \alpha'_M$ is exact, then any primitive $df = \alpha_M - \alpha'_M$ provides a diffeomorphism 
 $$
 \xymatrix{
 F:M\times \BR \ar[r]^-\sim &  M \times \BR
 &
 F(m, t) = (m, t + f(m))
 }
 $$ 
 intertwining the respective contact forms $F^*(dt+\alpha'_M) = dt + \alpha_M$. 
 \end{remark}


\subsubsection{Symplectification}

 Let $Z$ be an $(n+1)$-dimensional smooth manifold.

 Let $\pi^\oo_Z:S^\oo Z \to Z$ be the 
 spherically projectivized cotangent  bundle,
with points denoted by pairs $(z, [\xi]) \in S^\oo Z$ with $z\in Z$  a point and $[\xi] = \BR_{>0} \cdot\xi \subset T^*_z Z \setminus \{(z, 0)\}$ a  nontrivial ray.
Consider the canonical line bundle $\cL_Z \to S^\oo Z$ 
with fiber at  $(z, [\xi]) \in S^\oo Z$ the line $\BR\cdot \xi \subset T^*_z Z$.
The canonical one-form $\alpha_Z$ on $T^*Z$  descends to a $\cL_Z^\vee$-valued 
 one-form $\lambda^\oo_Z$ on $S^\oo Z$ whose kernel defines a cooriented contact structure $\xi^\oo_Z \subset T S^\oo Z$.

A choice of Riemannian metric on $Z$ provides an identification of $S^\oo Z$ with the 
resulting unit cosphere bundle  $U^*Z \subset T^*Z$, 
and equivalently, a trivialization of the canonical line bundle $\cL_Z \to S^\oo Z$.
In this case,  the then untwisted one-form $\lambda^\oo_Z$  on $S^\oo Z$ corresponds to the  restriction 
of the canonical one-form $\alpha_Z$ to $U^* Z$

Next,  suppose $Z=X\times \BR$, for an   $n$-dimensional smooth manifold $X$.

 Introduce the open subspace  
$$
\xymatrix{
\Upsilon_X = \{(x, t), [\xi, \eta]) \, |\, \eta>0\}   \subset S^\oo (X \times \BR)
}
$$ 
and fix the diffeomorphism
 $$
 \xymatrix{
J X \ar[r]^-\sim  & \Upsilon 
&
(x, \xi, t) \ar@{|->}[r] &  ((x, t), [-\xi, 1]) 
 }
 $$
 respecting the natural projections to $X\times \BR$. 
  The canonical line bundle $\cL_{X\times \BR}\to S^\oo (X\times \BR)$ is canonically trivialized over the image,
 and the pullback of the thus untwisted  one-form $\lambda^\oo_{X\times \BR}$ on $S^\oo (X\times \BR)$ is equal to the canonical contact form $\lambda_X$
 on $J X$. 
Thus the above diffeomorphism furnishes a cooriented contactomorphism.

 \medskip

Now  set $Z=X\times \BR = \BR^n \times \BR$. 

The composition of our previous two cooriented contactomorphisms provides a cooriented contactomorphism
$$
\xymatrix{
\psi:N\ar[r]^-\sim & JX  \ar[r]^-\sim &  \Upsilon_X
}
$$
$$
\xymatrix{
\psi(z_1, \ldots, z_n, t) = ((x_1, \ldots, x_n), t+\frac{1}{2} \sum_{a=1}^n x_a y_a), [-y_1, \ldots, -y_n, 1])
}
 $$
 intertwining the Legendrian projection $q:N\to X\times \BR$ and the natural projection $\Upsilon_X \to X\times \BR$.
  Note that it is equivariant for simultaneous real scalings of the $x, y, z$ components and 
squared real scalings of the additional components. It is also equivariant for  simultaneous conjugation of the $z$ components and negation
of the $y, t$ components and last base component.

\medskip

For compatibility with standard refererences, which often adopt the  setting of exact symplectic  rather than contact geometry, it is useful to go one step further. 

Let us regard
$T^*(X\times \BR) \setminus (X\times \BR)$
 as the symplectification of $ S^\oo (X\times \BR)$. 
Introduce the symplectification of the open subspace  $\Upsilon_X \subset S^\oo( X \times \BR)$
in the form of the  conic open subspace 
$$
\xymatrix{
\Omega_X  = \{((x, t), (\xi, \eta))  \, |\, \eta>0\} \subset  T^*(X\times \BR) \setminus (X\times \BR)
}
$$
Here and in what follows, we say a  subvariety of  $T^*(X \times \BR)$ is conic if it is invariant under  positive real scalings of the cotangent fibers.

Note that taking the inverse-image under the natural map $\Omega_X \to \Upsilon_X$ induces a bijection from subvarieties of $\Upsilon_X$ to conic subvarieties of $\Omega_X$.

\begin{defn}\label{def assoc lag}
To an exact Lagrangian subvariety $L \subset M$ with primitive $f:L\to \BR$, we define the {\em associated Lagrangian subvariety} $\Lambda \subset \Omega_X$ as follows. 

First, we lift $L\subset M$ to the Legendrian graph
$\Gamma_{L, -f} \subset N$ in the contactification, then transport $\Gamma_{L, -f} \subset N$ to the Legendrian subvariety $\Lambda^\oo = \psi(\Gamma_{L, -f}) \subset \Upsilon_X$, and finally take  $\Lambda \subset \Omega_X$ to be the inverse image
of  $\Lambda^\oo \subset \Upsilon_X$ under   
the natural map $\Omega_X \to \Upsilon_X$.

To a conic Lagrangian subvariety $L \subset M$, we always take the zero function as primitive,
and  then define the {\em associated Lagrangian subvariety} $\Lambda \subset \Omega_X$ as above. 
\end{defn}

\begin{remark}\label{rem biconic}
By construction,  the  associated Lagrangian subvariety $\Lambda \subset \Omega_X$ 
of an exact Lagrangian subvariety $L \subset M$ with primitive $f:L\to \BR$ is always conic with 
 respect to positive real scalings of the cotangent fibers.

For a conic Lagrangian subvariety $L \subset M$, 
the  associated Lagrangian subvariety $\Lambda \subset \Omega_X$ 
is additionally conic 
with respect to the commuting Hamiltonian scaling action 
$$
\xymatrix{
r\cdot ((x, t), (\xi, \eta))   = ((rx, r^2 t), (r^{-1}\xi, r^{-2} \eta)) 
&
r\in \BR_{>0}}
$$
induced by the scaling action
$
r\cdot (x, t) = (rx, r^2 t)
$
on the base.
To see this, note that 
 $\Lambda \subset \Omega_X$ 
is  conic with respect to the  scaling action 
$$
\xymatrix{
r\cdot ((x, t), (\xi, \eta))   = ((rx, r^2 t), (r\xi, \eta)) 
&
r\in \BR_{>0}
}
$$
and this simply differs from the asserted action by the corresponding squared
scalings of the cotangent fibers under which $\Lambda \subset \Omega_X$ is already invariant.

Furthermore, the above Hamiltonian scaling action contracts the pair
$\Lambda \subset \Omega_X$ to a neighborhood
of the positive codirection 
$$
\xymatrix{
\{((0, 0), (0, \eta)) \, |\, \eta>0\}  \subset \Omega_X
}
$$

We will  use the term {\em biconic} to summarize the above structure of 
the pair $\Lambda \subset \Omega_X$.
\end{remark}

%
%
%
%


\subsubsection{Symmetries}
Let $S^1 = \BR/2\pi\BZ$, and $T=(S^1)^n$, with Lie algebra $\ft = \BR^n$.
There is a  Hamiltonian $T$-action on $M=\BC^n$ by coordinate rotation 
$$
\xymatrix{
(\theta_1, \ldots, \theta_n) \cdot (z_1, \ldots, z_n) = (e^{i \theta_1} z_1, \ldots, e^{i\theta_n} z_n)
}
$$
with moment map
$$
\xymatrix{
\mu:M \ar[r] & \ft^* &
\mu(z_1, \ldots, z_n) = (r_1^2/2, \ldots, r_n^2/2)
}$$
It preserves the conic exact symplectic structure,
and its fixed locus is the origin $0\in M$.

There is an induced $T$-action on the contactification $N = \BC^n\times \BR$ which is trivial on the additional factor
$$
\xymatrix{
(\theta_1, \ldots, \theta_n) \cdot (z_1, \ldots, z_n, t) = (e^{i \theta_1} z_1, \ldots, e^{i\theta_n} z_n, t)
}
$$
It preserves the contact structure, 
and its fixed locus is the transverse curve $\{0\} \times \BR \subset N$. 

By transport along the contactomorphism 
$$
\xymatrix{
\psi:N\ar[r]^-\sim &  \Upsilon_X \subset S^\oo(X \times \BR)
}
$$
there is an induced $T$-action on the open subspace $\Upsilon_X\subset S^\oo(X \times \BR)$.
It preserves the contact structure, 
and its fixed locus is the transverse curve $\{((0, t), [0, 1])\} \subset \Upsilon_X$.

There is an induced Hamiltonian $T$-action on the symplectification $\Omega_X\subset T^*(X\times \BR)$
 with moment map
$$
\xymatrix{
\nu:\Omega_X \ar[r]^-s & \Upsilon_X \ar[r]^-{\psi^{-1}} & N \ar[r]^-c & M \ar[r]^-\mu &  \ft^*
}$$
where $s:\Omega_X\to \Upsilon_X$ is the  projection of the symplectification, 
and  $c: N \to M$ is the  projection of the contactification.
It preserves the conic exact symplectic structure,
and its fixed locus is the conic symplectic surface
$\{(0, t), (0, \eta)) \, |\, \eta>0 \}\subset \Omega_X$.

\begin{remark}
Tracing back through the constructions,
the Hamiltonian $T$-action on $\Omega_X$ originates by viewing $T$ as a maximal torus in the symplectic group of the contact plane at the point
$
((0, 0), [0, 1]) \in \Omega_X^\oo.
$

\end{remark}

%
%
%
%

Finally, it is useful to recast the Hamiltonian $T$-action on $\Omega_X$ in the form of the action Lagrangian correspondence
$$
\xymatrix{
\cL_{T, \Omega_X} \subset \Omega_X \times \Omega_X^a\times T^*T 
}
$$
$$
\xymatrix{
\cL_{T, \Omega_X} = \{ (\omega_1, -\omega_2, (g, \zeta)) \in \Omega_X \times \Omega_X^a \times T^*T \, |\,\omega_1 =  g\cdot \omega_2,\,  \nu(\omega_1) = \zeta\}  
}
$$
where $\Omega_X^a \subset  T^* (X\times \BR)$ denotes the antipodal subspace with respect to the negation of covectors.
Note the diffeomorphism
$$
\xymatrix{
\cL_{T, \Omega_X} \ar[r]^-\sim &   \Omega_X \times T
&
( \omega_1, -\omega_2, (g, \zeta)) \ar@{|->}[r] & (\omega_1, g)
}
$$

In particular, for $g\in T$, there is the action Lagrangian correspondence
$$
\xymatrix{
\cL_{g, \Omega_X} \subset \Omega_X \times \Omega_X^a
}
$$
$$
\xymatrix{
\cL_{g, \Omega_X} = \{ (\omega_1, -\omega_2) \in \Omega_X \times \Omega_X^a \, |\,\omega_1 =  g\cdot \omega_2\}
}
$$

\begin{remark}\label{rem action lag fiber}
Fix $g = (\theta_1, \dots, \theta_n)\in T$.

Let $Y_g \subset (X\times \BR) \times (X\times \BR)$ be the front projection of 
$\cL_{g, \Omega_X} \subset \Omega_X \times \Omega_X^a$.
To describe it, let $(x_1, \ldots, x_n, t)$, $(x_1', \ldots, x_n', t')$ be coordinates
on the two factors of $(X\times \BR) \times (X\times \BR)$.

First, points of $Y_g$ always satisfy $t' = t$.
If $\theta_a = 0$, they satisfy $x'_a = x_a$,
and if $\theta_a = \pi$, they satisfy $x'_a = - x_a$.
Otherwise, the projection $(x_a, x_a'):Y_g \to \BR^2$ is a fibration.

Thus $Y_g \subset (X\times \BR) \times (X\times \BR)$  is a smooth submanifold with
$\on{codim} Y_g = 1+\#\{a \, | \, \theta_a = 0 \mbox{ or } \pi \}$,
and $\cL_{g, \Omega_X} \subset \Omega_X \times \Omega_X^a$
is the intersection of $\Omega_X$ with its conormal bundle. 
\end{remark}


\subsection{Lagrangian skeleta}
We continue with $M=\BC^n$ and the above setup.

Introduce the superpotential  
$$
\xymatrix{
W:M\ar[r] &  \BC &  W(z_1, \ldots, z_n) = z_1\cdots z_n
}
$$
 


Set $M_0 = W^{-1}(0)$, $M^\times = W^{-1}(\BC^\times) = (\BC^\times)^n$.

 For $\theta\in S^1$, let
 $\BC^\times(\theta) \subset \BC^\times$  be the open ray 
 $$
 \xymatrix{
\BC^\times(\theta) =  \{z = re^{i\theta} \, |\, r\in \BR_{>0}\}
}
$$
For $\Theta\subset S^1$ a nonempty finite subset, let $C(\Theta) \subset \BC$ be the closed union of rays
$$
\xymatrix{
C(\Theta) = \{0\} \cup \coprod_{\theta\in \Theta} \BC^\times(\theta)
}
$$
Set $M(\Theta) = W^{-1}(C(\Theta))$ and $M^\times(\theta) = W^{-1}(\BC^\times(\theta))$ so that  
$$
\xymatrix{
M(\Theta) = M_0 \cup \coprod_{\theta\in \Theta} M^\times(\theta)
}
$$
When $\Theta = \{\theta\}$ is a single element, we write $C(\theta)$ in place of $C(\Theta)$, 
and $M(\theta)$ in place of $M(\Theta)$.

Fix a point $z = (z_1, \ldots, z_n)\in \BC^n$ with polar coordinates $z_a = r_a e^{ i \theta_a}$, for $a = 1, \ldots, n$.
Let $\ell  = \{r_1 , \ldots, r_n\}\subset \BR_{\geq 0}$ be the set of lengths of the coordinates,
and $\ell_0 \in \ell$ the minimum length. Let $I_\m\subset \{1, \ldots, n\}$ comprise those indices $a\in \{1, \ldots, n\}$ whose coordinate is of minimal length 
$r_a = \ell_0$.
Note $a\not \in I_\m$ implies in particular $r_a>0$.

Introduce the  subspace $L(\Theta) \subset  M(\Theta)$
cut out by the  equations
$$
\xymatrix{
\theta_a = 0, \mbox{ for } a\not \in I_\m
}
$$
Note that $L(\Theta) \subset  M$ is closed since $M(\Theta) \subset  M$ is closed, and $L(\Theta) \subset M(\Theta)$
results from imposing the above additional equations that become weaker as $I_\m$ increases in size.
When $\Theta = \{\theta\}$ is a single element, we write $L(\theta)$ in place of $L(\Theta)$.

There is a natural decomposition of
$L(\Theta)$ into conic isotropic locally closed submanifolds. 
We have the initial decomposition
 $$
 \xymatrix{
 L(\Theta) = L_0 \cup \coprod_{\theta\in \Theta} L^\times(\theta)
 }
 $$
 $$
 \xymatrix{
 L_0 = L(\Theta) \cap M_0 & L^\times(\theta) = L(\Theta)\cap M^\times(\theta)
 }
 $$

For each nonempty subset $\fI\subset \{1, \ldots, n\}$,
introduce the subspace $\fI L_{0} \subset L_0$ of points with $I_\m = \fI$. This is the 
locally closed submanifold cut out by the equations
$$
\xymatrix{
 r_a = 0, \mbox{ for } a\in \fI &  r_a >  0, \mbox{ for } a\not\in \fI  & \theta_a = 0, \mbox{ for } a\not\in\fI
}
$$
Its codimension is $n+|\fI|$ and it is clearly isotropic.

For each nonempty subset $\fI\subset \{1, \ldots, n\}$,
introduce the subspace $\fI L^\times(\theta) \subset L^\times(\theta)$ of points with $I_\m = \fI$. This is the 
locally closed submanifold cut out by the equations
$$
\xymatrix{
r_a >0, \mbox{ for all } a & r_a = r_b, \mbox{ for } a, b\in \fI & r_a < r_b, \mbox{ for } a\in \fI, b\not \in \fI 
}
$$
$$
\xymatrix{
\theta_a = 0, \mbox{ for } a\not\in\fI
& \sum_a \theta_a = \theta
}
$$
Its codimension is $n$ and it is clearly isotropic hence Lagrangian.

Finally, for each nonempty subset $\fI\subset \{1, \ldots, n\}$,
 note the natural identification
$$
\xymatrix{
 \fI L_0\cup \fI L^\times(\theta) \simeq \Cone((S^1)^{|\fI|-1}) \times \BR_{>0}^{n-|\fI|}
}
$$
 
 \begin{example}
 When $\fI= \{1, \ldots, n\}$, we have $\fI L_{0} = \{0\}$ and also
 $$
\xymatrix{
\fI L^\times(\theta) = \{(re^{i\theta_1}, \ldots, re^{i\theta_n})\, |\, r>0, \sum_a \theta_a = \theta\} \simeq (S^1)^{n-1} \times \BR_{>0}
}
$$
so that their union is the closed Lagrangian cone
$$
\xymatrix{
 \fI L_0\cup \fI L^\times(\theta) \simeq \Cone((S^1)^{n-1}) 
}
$$
\end{example}

\begin{lemma}
$L(\Theta)\subset M$  is a closed conic Lagrangian. 
\end{lemma}

\begin{proof}
We have noted that $L(\Theta)\subset M$ is closed. Each piece $\fI L_0$, $\fI L(\theta) \subset M$ is conic and isotropic. Moreover, we have seen that   $\fI L_0$ is in the closure of $ \fI L^\times(\theta)$ and the latter is of dimension $n$. Thus $L(\Theta)\subset M$ is Lagrangian.
\end{proof}

\begin{defn}(Lagrangian skeleton)
By a {\em Lagrangian skeleton} for $M=\BC^n$, $W= z_1\cdots z_n$, we will mean the
closed conic Lagrangian subvariety $L(\theta)\subset M$, for some $\theta\in S^1$.
\end{defn}


%


\subsection{Microlocal interpretation}

Fix the standard identification $T^*S^1 \simeq S^1 \times \BR$ with canonical coordinates
$(\theta, \xi)$. We have the canonical one-form $\alpha = \xi d\theta$, symplectic form $\omega = d\alpha = d\xi d\theta$, and Liouville vector field $v = \xi \partial_{\xi}$. 

Introduce the product torus $T = (S^1)^n$. 
Fix the standard identification $T^*T \simeq  T \times \BR^n$ with canonical coordinates
$(\theta_1, \ldots, \theta_n, \xi_1, \ldots, \xi_n)$. We have the canonical one-form $\alpha = \sum_{a=1}^n \xi_a d\theta_a$, symplectic form $\sum_{a= 1}^n \omega = d\alpha_a = d\xi_a d\theta_a$, and Liouville vector field $v = \sum_{a= 1}^n \xi_a \partial_{\xi_a}$. 

Introduce the open subspaces 
$$
\xymatrix{
T^{>0} S^1 = \{\xi>0\} \subset T^* S^1
&
T^{>0} T = (T^{>0} S^1)^n \subset T^* T
}$$
%
%
and the exact symplectic identification
$$
\xymatrix{
\varphi: M^\times  = (\BC^\times)^n \ar[r]^-\sim &  T^{>0} T
}
$$
$$
\xymatrix{ 
\varphi(r_1e^{i\theta_1}, \ldots, r_n e^{i\theta_n}) = (\theta_1, \ldots, \theta_n,  r_1^2/2,\ldots, r_n^2/2)
}
$$
Note that $\varphi$ is equivariant for the natural $T$-actions, and the codirection component of $\varphi$ is simply the restriction of the moment map $\mu$.

Fix $\theta\in S^1$. Recall the Lagrangian skeleton $L(\theta)\subset M$, and specifically
 its open subspace $ L^\times(\theta)\subset M^\times$,
 with locally closed submanifolds 
 $
   \fI L^\times(\theta)\subset L^\times(\theta),
   $
%
for  nonempty $\fI\subset \{1, \ldots, n\}$.

Transporting them along the above identification,
we obtain a corresponding conic Lagrangian with locally closed submanifolds 
$$
\xymatrix{
L^{>0}(\theta)  =\varphi(L^{\times}(\theta) ) 
&
\fI L^{>0}(\theta) = \varphi(\fI L^{\times}(\theta) ) 
}$$

Our aim in this section  is to describe them in microlocal terms.


\subsubsection{Lagrangians via cones}

Continue with $T=  (S^1)^n$, so that 
$$
\xymatrix{
\chi_*(T)= \Hom(S^1, T) \simeq \BZ^{n} & \chi^*(T) = \Hom(T, S^1) \simeq \BZ^{n}
}
$$
$$
\xymatrix{
\ft = \chi_*(T)\otimes_\BZ \BR \simeq \BR^n & \ft^*= \chi^*(T)\otimes_\BZ \BR \simeq \BR^n 
}
$$

Similarly, set $T^+ = S^1 \times T$, with  
$$
\xymatrix{
\chi_*(T^+) = \Hom(S^1, T^+) \simeq \BZ^{1+n} & \chi^*(T^+)= \Hom(T^+, S^1) \simeq \BZ^{1+n}
}
$$
$$
\xymatrix{
\ft^+ = \chi_*(T^+) \otimes_\BZ \BR \simeq \BR^{1+n} & (\ft^+)^*= \chi^*(T^+) \otimes_\BZ \BR \simeq \BR^{1+n} 
}
$$
 
Let $e_0, e_1, \ldots, e_n \in  \chi^*(T^+) $ be the coordinate vectors, 
so
that  $\ft = \{e_0 = 0\} \subset \ft^+$.

Let  $\tau \in [0, 2\pi)$ be the lift of $\theta \in S^1 = \BR/2\pi \BZ$,
and define
$$
\xymatrix{
\delta = e_1 + \cdots + e_n \in \chi^*(T)  & \delta^+ = -(\tau/2\pi) e_0 + e_1 + \cdots + e_n \in \chi^*(T^+) 
}
$$
and note that $\delta^+|_\ft = \delta$. 

For each nonempty subset $\fI\subset \{1, \ldots, n\}$, 
introduce  the linear span 
$$
\xymatrix{
\fI \sigma_{\lin} = \Span(\{\delta\}\cup\{ e_a \, |\, a\not \in \fI\}) \subset \ft^*
}
$$
Introduce also the relatively open cones
$$
\xymatrix{
\fI\sigma = \Span_{>0}(\{\delta\}\cup\{ e_a \, |\, a\not \in \fI\}) \subset \ft^*
&
\fI\sigma^+ = \Span_{>0}(\{\delta^+\}\cup \{ e_a \, |\, a\not \in \fI\}) \subset (\ft^+)^*
}
$$
where by the positive span we require that all of the listed vectors have positive coefficients.
Note that $\fI\sigma^+|_\ft = \fI\sigma$, and also that $\fI_1 \subset \fI_2$ implies $ \fI_2\sigma \subset \fI_1\sigma$,
$ \fI_2\sigma^+ \subset \fI_1\sigma^+$.

Introduce the affine subspace $\ft^+_\aff = \{e_0 = 1\} \subset \ft^+$ and the canonical identification $\ft^+_\aff\simeq \ft$
preserving the coordinates $e_1, \ldots, e_n$.
Introduce the orthogonal subspace 
$$
\xymatrix{
(\fI\sigma^+)^\perp =\{ v\in \ft^+ \, |\, \langle v, \lambda\rangle = 0, \mbox{ for all } \lambda\in \fI\sigma^+\}   \subset \ft^+
}
$$
and the affine subspace
$$
\xymatrix{
(\fI\sigma^+)^\perp_\aff =(\fI\sigma^+)^\perp\cap  \ft^+_\aff  \subset \ft^+_\aff \simeq \ft.
}
$$
Note that $\fI_1 \subset \fI_2$ implies $ (\fI\sigma_1^+)^\perp \subset(\fI\sigma_2^+)^\perp$,
 $ (\fI\sigma_1^+)_\aff^\perp \subset(\fI\sigma_2^+)_\aff^\perp$.

Consider the natural projection $q:\ft\to \ft/\chi_*(T) \simeq T$, and
form the image
$$
\xymatrix{
\fI S = q((\fI\sigma^+)^\perp_\aff) \subset T
}$$
Note that $\fI_1 \subset \fI_2$ implies $\fI_1 S \subset\fI_2 S$. 

Note also 
when $\fI = \{1, \ldots, n\}$, we have  that $\fI S \subset T$
is cut out by the equation $\sum_a \theta_a = \theta$, since $(\fI\sigma^+)^\perp_\aff \subset \ft$ 
is cut out by the equation $\sum_a v_a = \tau/2\pi$.
More generally, for any nonempty subset $\fI \subset \{1, \ldots, n\}$,
 we have  that $\fI S \subset T$
is cut out by the further equations  $\theta_a = 0$, for $a\not\in\fI$, since $(\fI\sigma^+)^\perp_\aff \subset \ft$ 
is cut out by the further equations  $v_a = 0$, for $a\not\in\fI$.

Let $T^* T$ be the cotangent bundle of $T$ with its natural identification $T^* T\simeq T\times \ft^*$.
For each nonempty subset $\fI\subset \{1, \ldots, n\}$, introduce the conic Lagrangian subspaces
$$
\xymatrix{
\fI S \times \fI\sigma_\lin \subset T\times \ft^*
&
\fI S \times \fI\sigma \subset T\times \ft^*
}$$
and note the identification
$$
\xymatrix{
T^*_{\fI S} T = \fI S  \times \fI\sigma_\lin
}
$$
Recall the conic locally closed Lagrangian submanifolds 
$$
\xymatrix{
\fI L^{>0}(\theta)   \subset (T^{>0}(S^1))^n   \subset T\times \ft^*
}$$

\begin{lemma}\label{lemma micro interpretation}
For a nonempty subset $\fI\subset \{1, \ldots, n\}$, inside of $T^* T \simeq T \times \ft^*$, 
we have 
$$
\xymatrix{
&
\fI L^{>0}(\theta) = \fI S  \times \fI\sigma 
}
$$
\end{lemma}

\begin{proof}
Recall that $\fI L^\times(\theta) \subset M^\times$ is cut out by the equations
$$
\xymatrix{
r_a >0, \mbox{ for all } a & r_a = r_b, \mbox{ for } a, b\in \fI & r_a < r_b, \mbox{ for } a\in \fI, b\not \in \fI 
}
$$
$$
\xymatrix{
\theta_a = 0, \mbox{ for } a\not\in\fI
& \sum_a \theta_a = \theta
}
$$
Recall that $\fI L^{>0}(\theta) = \varphi(\fI L^\times(\theta))$ for the exact symplectic identification 
$$
\xymatrix{
\varphi: M^\times  = (\BC^\times)^n \ar[r]^-\sim &  (T^{>0}(S^1))^n  \subset T^*((S^1)^n)
}
$$
$$
\xymatrix{ 
\varphi(r_1e^{i\theta_1}, \ldots, r_n e^{i\theta_n}) = (\theta_1, r_1^2/2, \ldots, \theta_n, r_n^2/2)
}
$$
Therefore $\fI L^{>0}(\theta) \subset T^*T$ is cut out by the similar equations
$$
\xymatrix{
\xi_a >0, \mbox{ for all } a & \xi_a = \xi_b, \mbox{ for } a, b\in \fI & \xi_a < \xi_b, \mbox{ for } a\in \fI, b\not \in \fI 
}
$$
$$
\xymatrix{
\theta_a = 0, \mbox{ for } a\not\in\fI
& \sum_a \theta_a = \theta
}
$$

Now we can simply match formulas. We have seen that the last two of the above collections of equations  
together cut out  $\fI S  \subset T$.
The 
first three describe precisely what it means to be in the positive cone 
$\fI\sigma = \Span_{>0}(\{\delta\}\cup\{ e_a \, |\, a\not \in \fI\}) \subset \ft$.
\end{proof}


\subsubsection{Structure when $\theta= 0$}
Let us  focus further on the case $\theta = 0\in S^1$.

Consider the diagonal character
$$
\xymatrix{
\delta:T \ar[r] &  S^1 &
\delta(\theta_1\ldots, \theta_n) = \theta_1 + \cdots + \theta_n
}
$$
Introduce the subtorus  $T^\circ = \ker(\delta) \subset T$, 
with Lie algebra  $\ft^\circ \subset \ft$, and note 
$$
\xymatrix{
\chi_*(T^\circ) =\Hom(S^1, T^\circ) \simeq   \{\delta\}^\perp \subset \chi_*(\ft)
&
\chi^*(T^\circ) = \Hom(T^\circ, S^1) \simeq \chi^*(T)/ \Span(\{\delta\})
}$$

Let $\Sigma \subset \ft^*$ be the complete real fan with rays $\ol e_1, \ldots, \ol e_n \in \chi^*(T^\circ)$ the images of
the coordinate vectors $e_1, \ldots, e_n \in \chi^*(T)$
under the quotient map $\chi^*(T)\to \chi^*(T^\circ)$.
Note that nonempty subsets $\fI \subset \{1, \ldots, n\}$ 
 index the positive cones $\sigma = \Span_{>0}(\{ \ol e_a \, |\, a\not \in \fI\}) \subset  \Sigma$, and
in particular, the subset $\fI = \{1, \ldots, n\}$ indexes  the origin $\sigma = \{0\} \subset \Sigma$.

\begin{remark}
Let $\check T^\circ =\Spec \BC[ \chi_*(T^\circ)]$ denote the complex torus dual to $T^\circ$. Then the complete fan 
$\Sigma \subset \chi^*(T^\circ)$ corresponds to the
$\check T^\circ $-toric variety  $\BP^{n-1}$.
\end{remark}

For each positive cone $\sigma\subset \Sigma$, 
introduce the orthogonal subspace 
$$
\xymatrix{
\sigma^\perp =\{ v\in \ft^\circ \, |\, \langle v, \lambda\rangle = 0, \mbox{ for all } \lambda\in \sigma\}   \subset \ft^\circ
}
$$
Consider the natural projection $q:\ft^\circ \to \ft^\circ/\chi_*(T^\circ) \simeq T^\circ $, and form the image
$$
\xymatrix{
\sigma T^\circ = q(\sigma^\perp) \subset T^\circ
}
$$ 

Define $\Lambda_\Sigma \subset T^* T^\circ \simeq T^\circ \times (\ft^\circ)^*$ to be the conic Lagrangian
$$
\Lambda_\Sigma = \bigcup_{\sigma \subset \Sigma} \sigma T^\circ \times \sigma \subset T^\circ \times (\ft^\circ)^*
$$

\begin{remark}
As we will discuss later, the conic Lagrangian
$
\Lambda_\Sigma \subset T^* T^\circ
$
is the mirror skeleton to the  $\check T^\circ$-toric variety  $\BP^{n-1}$.
\end{remark}

The inclusion $T^\circ \subset T$ induces a natural Lagrangian correspondence
$$
\xymatrix{
\ar[d]^\sim T^* T^\circ & \ar@{->>}[l]_-p \ar[d]^\sim  T^* T \times_T T^\circ \ar@{^(->}[r]^-i &\ar[d]^-\sim  T^* T \\
T^\circ \times (\ft^\circ)^* & \ar@{->>}[l]  T^\circ\times\ft^* \ar@{^(->}[r] & T\times \ft^* 
}
$$
compatible with the natural projection  $\ft^* \to  \ft^*/\Span(\{\delta\}) \simeq (\ft^\circ)^*$.

Recall the conic Lagrangian $L^{>0}(0) \subset T^* T$, and
its locally closed submanifolds $ \fI L^{>0}(0) \subset T^* T$,
for nonempty subsets $\fI\subset \{1, \ldots, n\}$.

Introduce the corresponding conic Lagrangian and  locally closed submanifolds
$$
\xymatrix{
\Lambda^\circ  = p(i^{-1}(L^{>0}(0))) \subset T^*T^\circ
&
\fI\Lambda^\circ  = p(i^{-1}(\fI L^{>0}(0))) \subset T^*T^\circ
}
$$

Note that $L^{>0}(0)$ in fact already lies in $T^* T \times_T T^\circ$, 
since its points satisfy $\sum_{a=1}^n \theta_a = 0$,
so the inverse image $i^{-1}$ is unnecessary in the above formulas.

Note also that
 the fibers of $p$ are the cosets of the line  $\Span(\{\delta\}) \simeq \BR$, and their intersections
 with $L^{>0}(0)$ are cosets of the positive ray $\Span_{>0}(\{\delta\}) \simeq \BR_{>0}$.
 Thus the projection $L^{>0}(0)\to \Lambda^\circ$ is simply an $\BR_{>0}$-bundle.


%

\begin{lemma}\label{lemma ham red}
Inside of $T^* T^\circ$,  we have 
$$
\xymatrix{
\Lambda^\circ  = \Lambda_{\Sigma}
 &
 \fI \Lambda^\circ = \sigma T^\circ \times \sigma  
}
$$
where a nonempty subset $\fI \subset \{1, \ldots, n\}$ indexes the positive cone  
$\sigma = \Span_{>0}(\{ \ol e_a \, |\, a\not \in \fI\}) \subset  \Sigma$.
\end{lemma}

\begin{proof}
The second assertion refines the first. For the second, by Lemma~\ref{lemma micro interpretation},
we have
$$
\xymatrix{
\fI L^{>0}(0) \simeq \fI S \times \fI\sigma
}
$$
where 
$\fI S = q(\sigma^\perp) \subset T^\circ$ since $\tau=0$,
and
$\fI \sigma= \Span_{>0}(\{\delta\} \cup \{ \ol e_a \, |\, a\not \in \fI\}) \subset  \ft^*$. Hence $\fI S =  \sigma T^\circ $, and
$\fI \sigma$ projects to $\sigma$.
\end{proof}


\subsection{Canonical section}

Recall for any $\theta\in S^1$, the Lagrangian skeleton $ L(\theta) \subset M$ admits a decomposition
 $$
 \xymatrix{
 L(\theta) = L_0 \cup \coprod_{\theta\in \Theta} L^\times(\theta)
 }
 $$
 $$
 \xymatrix{
 L_0 = L(\theta) \cap M_0 & L^\times(\theta) = L(\theta)\cap M^\times(\theta)
 }
 $$ 

Recall the decomposition of $L_0$ into  the locally closed submanifolds  $\fI L_{0} \subset L_0$,
for nonempty subsets $\fI\subset \{1, \ldots, n\}$, cut out by the equations
$$
\xymatrix{
 r_a = 0, \mbox{ for } a\in \fI & \theta_a = 0, \mbox{ for } a\not\in\fI
}
$$
Note that points of $L_0$ are completely described by their radial coordinates and the angular coordinates are either
not well-defined or set equal to zero.  

Recall the complete fan $\Sigma\subset (\ft^\circ)^*$
 with rays  $\ol e_1, \ldots, \ol e_n \in \chi^*(T^\circ)$ the images of $e_1, \ldots, e_n \in \chi^*(T)$
under the quotient map $\chi^*(T)\to \chi^*(T^\circ)$.
Recall that nonempty subsets $\fI \subset \{1, \ldots, n\}$ 
 index the positive cones $\sigma = \Span_{>0}(\{ \ol e_a \, |\, a\not \in \fI\}) \subset  \Sigma$, and
in particular, the subset $\fI = \{1, \ldots, n\}$ indexes  the origin $\sigma = \{0\} \subset \Sigma$. 

%


\begin{lemma}\label{lemma homeo fiber}
We have a piecewise-linear homeomorphism 
$$
\xymatrix{
h_0:L_0 \ar[r]^-\sim & (\ft^\circ)^* \simeq \BR^{n-1}
&
h_0(r_1, \ldots, r_n)  = - r_1 \ol e_1 - \cdots - r_n \ol e_n
}
$$
that takes the locally closed submanifold $\fI L_{0} \subset L_0$ homeomorphically to the corresponding opposite  cone $ -\sigma \subset -\Sigma$.
\end{lemma}

\begin{proof}
Note that $L_0 \subset M$ consists  of $n$-tuples  $(r_1, \ldots, r_n) \in M$ of 
real non-negative radii
with
at least one radius
equal to zero. The corresponding submanifolds and cones are cut out by the  vanishing and positivity of 
the respective radii and coordinate coefficients.
\end{proof}

\begin{remark}
Motivation for the negative signs in the definition of $h_0$ can be found in natural extensions of it immediately below.
\end{remark}

Now fix a representative $\tau\in (-2\pi, 2\pi)$ projecting to $\theta\in S^1$.


We will construct a closed conic Lagrangian subvariety $P(\tau) \subset L(\theta)$ and a homeomorphism
$$
\xymatrix{
h = g \times w :P(\tau) \ar[r]^-\sim & (\ft^\circ)^*\times C(\theta) \simeq \BR^{n-1} \times \BR_{\geq 0}
}
$$
where the second factor $w:P(\tau) \to C(\theta)$ is simply the restriction of the superpotential $W:M\to \BC$.
Furthermore, above the origin $0\in C(\theta)$, the homeomorphism will be that of the previous lemma
$$
\xymatrix{
h|_0 = h_0:P(\tau)|_0 = L_0   \ar[r]^-\sim & (\ft^\circ)^* \simeq \BR^{n-1}
}
$$
%

For $a=1, \ldots, n$, fix the representative $\tau_a\in [0, 2\pi)$  projecting to $\theta_a\in S^1$, if $\tau\geq 0$,
or alternatively $\tau_a\in (-2\pi, 0]$, if $\tau< 0$.

\begin{defn}[Canonical section]
Define $P^\times(\tau) \subset L^\times(\theta)$ to be the closed conic Lagrangian cut out by the
single additional  equation
$$
\xymatrix{
 \sum_{a=1}^n \tau_a = \tau
}
$$

Define $P(\tau) \subset L(\theta)$ to be the closed conic Lagrangian $P(\tau) = L_0 \cup P^\times (\tau)$.

\end{defn}

\begin{remark}
Note that $P(\tau)$ is equivalently the closure of  $P^\times (\tau)$ regarded as a subspace of $ L(\theta)$
or as a subspace of $M$. 
\end{remark}

Recall  the decomposition of $L^\times(\theta)$ into the 
locally closed submanifolds $\fI L^\times(\theta)$,
for nonempty subsets $\fI\subset \{1, \ldots, n\}$.
Taking intersections, we obtain a decomposition of $P^\times(\tau)$ into  
locally closed submanifolds 
$$
\xymatrix{
\fI P^\times(\tau)  = P^\times(\tau) \cap \fI L^\times(\theta)
}
$$
cut out by the equations
$$
\xymatrix{
r_a >0, \mbox{ for all } a & r_a = r_b, \mbox{ for } a, b\in \fI & r_a < r_b, \mbox{ for } a\in \fI, b\not \in \fI 
}
$$
$$
\xymatrix{
\tau_a = 0, \mbox{ for } a\not\in\fI
& \sum_a \tau_a = \tau
}
$$

Transporting them along the identification $\varphi$,
we obtain a conic Lagrangian with locally closed submanifolds 
$$
\xymatrix{
P^{>0}(\tau)  =\varphi(P^{\times}(\tau) ) 
&
\fI P^{>0}(\tau) = \varphi(\fI P^{\times}(\tau) ) 
}$$
for nonempty subsets $\fI\subset \{1, \ldots, n\}$.

Let $\Delta(\tau) \subset T$ be the simplex with $ \sum_{a=1}^n \tau_a = \tau$.
Note that we have  
$$
\xymatrix{
P^{>0}(\tau)  = L^{>0}(\theta) \times_T \Delta(\tau)
}
$$

For a nonempty subset $\fI \subset \{1, \ldots, n\}$, introduce the relatively open subsimplex
 $\fI \Delta(\tau) \subset \Delta(\tau)$  defined by the equations $\tau_a \not = 0$,  for  $a\in\fI$,
 and  $\tau_a = 0$,  for  $a\not\in\fI$.

As an immediate consequence of Lemma~\ref{lemma micro interpretation}, we have the following description.

\begin{lemma}\label{lemma micro interpretation slice}
For a nonempty subset $\fI\subset \{1, \ldots, n\}$, inside of $T^* T \simeq T \times \ft^*$, 
we have 
$$
\xymatrix{
&
\fI P^{>0}(\tau) = \fI \Delta(\tau)  \times \fI\sigma 
}
$$
\end{lemma}

Now set $r = r_1 \cdots r_n$, and define the first factor of the sought-after homeomorphism to be
$$
\xymatrix{
g: P(\tau) \ar[r] & (\ft^\circ)^* \simeq \BR^{n-1}
&
g(r_1 e^{i\theta_1}, \ldots, r_n e^{i\theta_n}) = (r\tau_1 - r_1)\ol e_1 + \cdots + (r\tau_n - r_n)\ol e_n 
}
$$
Observe that when $r=0$, this clearly restricts to the homeomorphism $h_0$.


\begin{prop}\label{prop homeo slice}
The map $g:P(\tau) \to (\ft^\circ)^*$ provides the first factor of a homeomorphism 
$$
\xymatrix{
h = g \times w:P(\tau) \ar[r]^-\sim & (\ft^\circ)^* \times C(\theta) \simeq \BR^{n-1} \times \BR_{\geq 0}
}
$$
with second factor $w:P(\tau) \to C(\theta)$ the restriction of the superpotential $W:M\to \BC$.
\end{prop}

\begin{proof}
By Lemma~\ref{lemma homeo fiber}, it suffices to study the restriction to $P^\times(\tau)$.

When $\tau = 0$,  observe that all angles vanish, $P^\times(0)\subset M^\times$
 consists  of $n$-tuples  $(r_1, \ldots, r_n) \in M^\times$ of positive
radii, and the map reduces to the homeomorphism
$$
\xymatrix{
h(r_1, \ldots, r_n) =  ( - r_1\ol e_1 - \cdots  - r_n\ol e_n, r)
}
$$

When $\tau >  0$, observe that $P^\times(\tau)\subset M^\times$
 consists  of $n$-tuples  $(r_1e^{i\tau_1}, \ldots, r_ne^{i\tau_n}) \in M^\times$ satisfying the following.

First,
the   angles $(\tau_1, \ldots, \tau_n)\in \BR_{\geq 0}^n $ form the simplex $\Delta(\tau)$. 

Second,
by Lemma~\ref{lemma micro interpretation slice}, 
 above $\fI \Delta(\tau)  \subset \Delta(\tau)$,  we have 
$$
\xymatrix{
P^\times(\tau)|_{\fI\Delta} =  \fI \Delta(\tau) \times \fI \sigma
&
\fI \sigma = \Span_{ >0}(\{e\} \cup \{e_a \, |\, a\not \in \fI\})
}
$$
Thus  above $\fI \Delta(\tau)  \subset \Delta(\tau)$, the  map  takes the form 
$$
\xymatrix{
h(r_1, \ldots, r_n, \tau_1, \ldots, \tau_n) =  (r \sum_{a\in\fI} \tau_a \ol e_a - \sum_{b\not \in \fI} r_b \ol e_b, r)
}
$$
and  hence provides an inclusion 
$$
\xymatrix{
P^\times(\tau)|_{\fI\Delta(\tau)} \ar@{^(->}[r] &   \BR^{n-1} \times \BR_{> 0}
}
$$

The images of the above inclusions decompose  $\BR^{n-1} \times \BR_{> 0}$
into disjoint subspaces indexed by nonempty subsets $\fI \subset \{1, \ldots, n\}$.
 Thus $h$ provides a bijection 
 $$
\xymatrix{
P^\times(\tau) \ar[r] &   \BR^{n-1} \times \BR_{> 0}
}
$$
 and by the description of Lemma~\ref{lemma micro interpretation slice},  it is a homeomorphism.
 
 When $\tau <  0$, a similar analysis holds.
\end{proof}

\begin{corollary}
For $\tau_1, \tau_2\in (-2\pi,2\pi)$
representing $\theta_1 \not = \theta_2\in S^1$, 
the union $P(\tau_1) \cup P(\tau_2)\subset L(\theta_1) \cup L(\theta_2)$ admits a homeomorphism 
$$
\xymatrix{
H:P(\tau_1) \cup P(\tau_2) \ar[r]^-\sim & (\ft^\circ)^* \times (C(\theta_1) \cup C(\theta_2) \simeq \BR^{n-1} \times \BR 
}
$$
\end{corollary}

\begin{proof}
Take the homeomorphisms constructed above on each piece of the union $P(\tau_1) \cup P(\tau_2)$ and note that they agree on the intersection $L_0 =P(\tau_1) \cap P(\tau_2)$.
\end{proof}


\section{Landau-Ginzburg $A$-model}\label{s lg model}


\subsection{Microlocal sheaves}\label{ss micro shs}

This section collects mostly standard material from~\cite{KS} tailored to our setting.

%
%
%

\subsubsection{Setup}

Let $Z$ be a real analytic manifold. 

Consider the cotangent bundle  and 
its spherical projectivization 
$$
\xymatrix{
\pi:T^*Z\ar[r] & Z
&
\pi^\oo:S^\oo Z = (T^*Z \setminus Z)/\BR_{>0}\ar[r] & Z
}
$$
with their respective standard exact symplectic and contact structures.

For convenience, fix a Riemannian metric on $Z$, so that in particular we have an identification
with the unit cosphere bundle
$$
\xymatrix{
S^\oo Z\simeq U^*Z \subset T^*Z
}
$$

Consider a closed conic Lagrangian subvariety and its Legendrian spherical projectivization
$$
\xymatrix{
\Lambda \subset T^*Z
&
\Lambda^\oo = (\Lambda\cap  (T^*Z \setminus Z))/\BR_{>0} \subset S^\oo Z
}
$$
Introduce the front projection 
$$
\xymatrix{
Y = \pi^\oo(\Lambda^\oo)\subset Z
}
$$ 
In the generic situation, the restriction 
$$
\xymatrix{
\pi^\oo|_{\Lambda^\oo}:\Lambda^\oo\ar[r] & Y 
}
$$ is finite so that $Y\subset Z$
is a hypersurface. 

Fix  $\cS=\{Z_\alpha\}_{\alpha\in A}$  a Whitney stratification  of $Z$ such that $Y\subset Z$ is a union of strata. Hence
we have  inclusions
$$
\xymatrix{
\Lambda \subset T^*_\cS Z = \coprod_{\alpha\in A} T^*_{Z_\alpha} Z
&
\Lambda^\oo \subset S^\oo_\cS Z = \coprod_{\alpha\in A} S^\oo_{Z_\alpha} Z
}
$$ 
where we take the union of conormal bundles to strata 
and their spherical projectivizations.


\subsubsection{Sheaves}

Fix a field $k$ of characteristic zero. 

Let $\Sh(Z)$ denote the dg category of  constructible complexes of sheaves of $k$-vector spaces on $Z$.
Let $\Sh_\cS(Z) \subset \Sh(Z)$ denote the full dg subcategory of   $\cS$-constructible complexes.
We will abuse terminology and refer to objects of $\Sh(Z)$ as constructible sheaves.
 All functors between dg categories of constructible sheaves will be derived in the dg sense, though the notation may not explicitly reflect it.  
 
 Recall to any $\cF \in\Sh(Z) $, we can assign its
singular support 
$$
\xymatrix{
\ssupp(\cF) \subset T^* Z
}
$$ which is a closed conic Lagrangian subvariety,
and also its spherical projectivization 
$$\xymatrix{
\ssupp^\oo(\cF) = (\ssupp(\cF)\setminus (T^*Z \setminus Z))/\BR_{>0} \subset S^\oo Z
}
$$
which is a closed Legendrian subvariety.

\begin{example}\label{ex conv}
To fix conventions, suppose $i:U\to Z$ is the inclusion of an open subspace whose closure is a submanifold with boundary modeled on a Euclidean halfspace. Then the singular support  $ \ssupp(i_* k_U) \subset T^*Z$ of the standard extension $i_* k_U\in \Sh(Z)$ 
consists 
of the union of $U \subset Z$ and the inward conormal codirection along the boundary $\partial U \subset Z$.
More precisely, if near a point $z\in \partial U$, we have $U = \{f > 0\}$, for 
a local coordinate $f$, then $ \ssupp(i_* k_U)|_z$ is the closed ray $\R_{\geq 0} \langle df|_z\rangle $. 

More generally, suppose $i:U\to Z$ is the inclusion of an open subspace whose closure is a submanifold with corners
modeled  on a Euclidean quadrant. 
Then the singular support  $ \ssupp(i_* k_U) \subset T^*Z$ consists of the inward conormal cone along the boundary $\partial U \subset Z$. More precisely, if near a point $z\in \partial U$, we have $U = \{f_1, \ldots, f_k > 0\}$, for 
local coordinates $f_1, \ldots, f_k$, then $  \ssupp(i_* k_U)|_z$
 is the closed cone
$\R_{\geq 0} \langle df_1|_z, \ldots, df_k|_z\rangle$. 

\end{example}

For a conic Lagrangian subvariety $\Lambda\subset T^*Z$,
we write $\Sh_{\Lambda}(Z) \subset \Sh(Z)$ for the full dg category of objects $\cF \in \Sh(Z)$ with singular support satisfying 
$
\ssupp(\cF) \subset \Lambda.
$

The inclusion $\Lambda\subset T^*_\cS Z$ implies the full inclusion 
$\Sh_{\Lambda}(Z)\subset \Sh_\cS(Z)$, and more generally, an inclusion $\Lambda\subset \Lambda'$
implies the full inclusion 
$\Sh_{\Lambda}(Z)\subset \Sh_{\Lambda'}(Z)$. 

For the zero-section $\Lambda = Z$, there is a canonical equivalence
$\Sh_{\Lambda}(Z)\simeq \Loc(Z)$ with the full dg subcategory $\Loc(Z) \subset \Sh(Z)$ of local systems.
For the antipodal  conic Lagrangian subvariety  $-\Lambda \subset T^*Z$, Verdier duality provides a canonical equivalence
$$
\xymatrix{
\BD_Z:\Sh_{\Lambda}(Z)^{\op} \ar[r]^-\sim & \Sh_{-\Lambda}(Z)
}
$$

When $U\subset Z$ is an open subset, we will abuse notation and write 
 $\Sh_\cS(U) \subset \Sh(U)$ for complexes constructible with respect to $\cS \cap U$,
and $\Sh_\Lambda(U) \subset \Sh(U)$ 
for complexes with singular support lying in 
$
\Lambda \cap \pi^{-1}(U).
$

\begin{example}\label{ex torus tensor cat}
Let $T \simeq (S^1)^n$ be a torus.

%
%
%

Let $m:T\times T \to T$ be the multiplication map, and $\iota:T\to T$ the inverse map.
Then $\Sh(T)$ is a tensor category with respect to convolution
$$
\xymatrix{
\cF_1\star \cF_2 = m_!(\cF_1\boxtimes \cF_2) &
\cF_1, \cF_2\in \Sh(T)
}
$$
 with  unit 
 $k_{e}\in \Sh(T)$ the skyscraper at the identity $e\in T$, and
 duals given by  
 $$
 \xymatrix{
 \cF^\vee = \iota_!\mathbb D_T(\cF)
 &
 \cF\in \Sh(T)
  }
  $$

The full dg subcategory $\Loc(T) \simeq \Sh_T(T)$ of local systems 
is a monoidal ideal, and admits the non-unital monoidal Fourier description
$$
\xymatrix{
\Loc(T) \simeq \Sh_T(T) \ar[r]^-\sim & \Coh_\torsion(\check T)
}
$$
where $\check T \simeq (\G_m)^n$ is the dual torus, and $\Coh_\torsion(\check T)$ its dg category of torsion
sheaves.

Let  $i:S  \to T$  be the inclusion of a subtorus.
Then $\Sh(S)$ is similarly a tensor category, and pushforward along $i$
induces a fully faithful  tensor functor
$$
\xymatrix{
i_*: \Sh(S)\ar@{^(->}[r]  & \Sh(T)
}
$$

Let  $p:T'  \to T$  be a covering group, possibly with infinite but discrete kernel.
Then the full dg subcategory $\Sh_c(T') \subset \Sh(T')$ of objects with compact support is similarly a tensor category, and pushforward along $p$
induces a fully faithful  tensor functor
$$
\xymatrix{
p_!\simeq p_*: \Sh(T')\ar[r]  & \Sh(T)
}
$$

\end{example}

\begin{example}\label{ex ccc}
Recall the torus $T^\circ$ and the conic Lagrangian $\Lambda_\Sigma \subset T^*T^\circ$ associated to the complete fan
$\Sigma\subset (\ft^\circ)^*$.
Recall the dual torus $\check T^\circ$ and that the complete fan 
$\Sigma \subset  (\ft^\circ)^*$ corresponds to the
$\check T^\circ$-toric variety  $\BP^{n-1}$.

The full dg subcategory 
$\Sh_{\Lambda_\Sigma} (T^\circ)\subset \Sh(T^\circ)$ is a tensor subcategory,
and
a basic instance of the coherent-constructible correspondence of~\cite{B, ccc, T} is a canonical 
tensor equivalence
$$
\xymatrix{
\Sh_{\Lambda_\Sigma} (T^\circ) \ar[r]^-\sim & \Coh(\BP^{n-1})
}
$$
 where $\Coh(\BP^{n-1})$ is equipped with its usual tensor product.

Alternatively, we could work with 
the antipodal conic Lagrangian subvariety
$-\Lambda_\Sigma \subset T^*T^\circ$.
The choice is largely a matter of conventions thanks to the auxiliary equivalences
provided by the inverse map and Verdier duality
$$
\xymatrix{
\iota_*:\Sh_{\Lambda_\Sigma} (T^\circ) \ar[r]^-\sim &
\Sh_{-\Lambda_\Sigma} (T^\circ)
&
\BD_{T^\circ}:\Sh_{-\Lambda_\Sigma} (T^\circ) \ar[r]^-\sim &
\Sh_{\Lambda_\Sigma} (T^\circ)^\op
}
$$
The full dg subcategory 
$\Sh_{-\Lambda_\Sigma} (T^\circ)\subset \Sh(T^\circ)$ is also a tensor subcategory,
and the  inverse map provides a tensor equivalence.

Let us mention two further compatibilities
among many the coherent-constructible equivalence  enjoys:

%
 
 i) For $a = 1, \ldots, n$, introduce variables $\tau_a \in (0, 2\pi)$, and consider  the open simplex 
 $$
 \xymatrix{
 d:\Delta = \{(\tau_1, \ldots, \tau_n) \, |\,  \sum_{a=1}^n \tau_a = 2\pi\}
 \ar@{^(->}[r] &  T^\circ
}
$$
Then $\ssupp(d_*k_{\Delta}) \subset \Lambda_\Sigma$,
and the equivalence takes   $d_*k_{\Delta}\in \Sh_{\Lambda_\Sigma} (T^\circ)$
to $\cO_{\BP^{n-1}}(-1) \in \Coh(\BP^{n-1})$.

ii) On the one hand, recall that 
over the identity $e\in T^\circ$, the fiber of $\Lambda_\Sigma$ is the complete fan
$ \Sigma$. Moreover, recall that
 the smooth locus of $\Lambda_\Sigma|_e\simeq \Sigma$ is the union of the open cones
$$
\xymatrix{
\sigma_\alpha = \Span_{>0}(\{ \ol e_a \, |\, a\not = \alpha \}) \subset  \Sigma
&
\alpha\in \{1, \ldots, n\}
}
$$

Given a covector $(e, \xi_\alpha) \in \sigma_\alpha$ in such an open cone, we can form the vanishing cycles
$$
\xymatrix{
\phi_\alpha:\Sh_{\Lambda_\Sigma}(T^\circ)\ar[r] & \Perf_k
&
\phi_\alpha(\cF) = \Gamma_{\{f_\alpha\geq 0\}}(U; \cF)
}
$$
where $f_\alpha:T^\circ\to\BR$ is any smooth function with $f_\alpha(e) = 0, df_\alpha|_e = \xi_\alpha$, and  $U\subset T^\circ$ is a sufficiently small open ball
around $e$.

On the other hand,  for $\alpha\in \{1, \ldots, n\}$, introduce  the inclusion of the $\alpha$-coordinate line
$$
 \xymatrix{
 i_\alpha:\pt = \{[e_\alpha]\}\ar@{^(->}[r] &  \BP^{n-1}
 }
$$
and the induced pullback functor
$$
\xymatrix{
i_\alpha^*:\Coh(\BP^{n-1}) \ar[r] & \Coh(pt) \simeq \Perf_k
}
$$
 
 Then the  equivalence extends to a commutative diagram
$$
\xymatrix{
\ar[dr]_-{\phi_\alpha}\Sh_{\Lambda_\Sigma} (T^\circ) \ar[rr]^-\sim && \Coh(\BP^{n-1})\ar[dl]^-{i^*_\alpha}\\
& \Perf_k &
}
$$

\end{example}
%


\subsubsection{Microlocal sheaves}

Let 
$\Omega_Z \subset T^*Z$ be a conic open subspace,  and let $\Lambda \subset T^*Z$  be a closed conic Lagrangian subvariety.
Only the intersection $\Lambda \cap \Omega_Z$ will play a role, and we will often not specify $\Lambda$ outside of $\Omega_Z$.

%

Let $\mu\Sh_{\Lambda}(\Omega_Z)$ denote the dg category of microlocal sheaves on $\Omega_Z$ supported along $\Lambda$. 
It is useful to view $\mu\Sh_{\Lambda}(\Omega_Z)$  as the sections over $\Lambda$ of a natural sheaf of dg categories with local sections admitting the following concrete descriptions.
Note for $(x, \xi) \in \Lambda$ there are two local cases: either 1) $\xi = 0$ so that locally
$\Omega_Z$ is  the cotangent bundle $T^*B$ of a small open ball $B\subset Z$, or  2) $\xi \not = 0$ so that locally $\Omega_Z$ is the symplectification of a small 
open ball  $\Omega_Z^\oo \subset S^\oo Z$.  

Case 1) For $B= \pi(\Omega_Z)$, there is always a canonical functor $\Sh_\Lambda(B) \to \mu\Sh_\Lambda(\Omega_Z)$,
and when $\Omega_Z = T^*B$, this functor is in fact an equivalence
$$
\xymatrix{
\Sh_\Lambda(B) \ar[r]^-\sim & \mu\Sh_\Lambda(T^*B)
}
$$

Case 2) Suppose $\Omega_Z \subset T^*Z$ is the symplectification of
a small open ball $\Omega_Z^\oo\subset S^\oo Z$. 
By applying a contact transformation, we may arrange to be in the generic situation where  the
front projection 
$$
\xymatrix{
\pi^\oo|_{\Lambda^\oo}:\Lambda^\oo\ar[r] & Y 
}
$$ 
is finite so that $Y= \pi^\oo(\Lambda^\oo )\subset Z$ is a hypersurface.
%
%
%

For $B = \pi(\Omega_Z)$, the natural functor
$\Sh_{\Lambda}(B)\to \mu \Sh_{\Lambda}(\Omega_Z)$
 induces  an
equivalence on the quotient dg category
$$
\xymatrix{
\Sh_{\Lambda}(B)/\Loc(B)\ar[r]^-\sim  & \mu \Sh_{\Lambda}(\Omega_Z)
}
$$
where  $\Loc(B) \subset \Sh(B)$ denotes the full dg subcategory of local systems, or in other words complexes with  singular support lying in the zero-section $B\subset T^*B$.

Alternatively, in this case, 
introduce the respective full dg subcategories 
$$
\xymatrix{
\Sh_{\Lambda}(B)^0_* \subset  \Sh_{\Lambda}(B) &
\Sh_{\Lambda}(B)^0_! \subset  \Sh_{\Lambda}(B) }
$$
of complexes $\cF \in \Sh_{\Lambda}(B) $
with no sections and no  compactly-supported sections 
$$
\xymatrix{
\Gamma(B, \cF) \simeq 0
&
\Gamma_c(B, \cF) \simeq 0
}$$
Then the natural functor $\Sh_{\Lambda}(B)\to \mu \Sh_{\Lambda}(\Omega_Z)$ restricts to equivalences
$$
\xymatrix{
\Sh_{\Lambda}(B)^0_*\ar[r]^-\sim  & \mu \Sh_{\Lambda}((\Omega_Z)
&
\Sh_{\Lambda}(B)^0_!\ar[r]^-\sim  & \mu \Sh_{\Lambda}((\Omega_Z)
}
$$

More generally, if we happen not to be in the generic situation, let $\Sh_{\Lambda}(B, \Omega_Z)\subset \Sh(B)$
denote the full dg subcategory of objects $\cF\in \Sh(B)$ with singular support satisfying $\ssupp(\cF) \cap \Omega_Z \subset \Lambda$.
Then there is a natural equivalence
$$
\xymatrix{
\Sh_{\Lambda}(B, \Omega_Z)/K(B, \Omega_Z)\ar[r]^-\sim  & \mu \Sh_{\Lambda}(\Omega_Z)
}
$$
where  $K(B, \Omega_Z)\subset\Sh_{\Lambda}(B, \Omega_Z)$ denotes the full dg subcategory of 
 objects $\cF\in \Sh(B)$ with singular support satisfying $\ssupp(\cF) \cap \Omega_Z = \emptyset$.

\begin{remark}\label{rem micro when conic} 
We will not encounter complicated gluing for microlocal sheaves.

When not in  Case 1), we will have a contracting action $\alpha:\BR_{>0}\times Z\to Z$ 
with a unique fixed point, the pair $\Lambda\subset \Omega_Z$ will be biconic for the additional induced Hamiltonian action and
contracted by it to a neighborhood of a single codirection
based at the fixed point.
Thus the situation will be equivalent to Case 2),
and
we will have an equivalence 
$$
\xymatrix{
\Sh^\conic_{\Lambda}(Z, \Omega_Z)/K^\conic(Z, \Omega_Z)\ar[r]^-\sim  & \mu \Sh_{\Lambda}(\Omega_Z)
}
$$
where $\Sh^\conic_{\Lambda}(Z, \Omega_Z)\subset \Sh_{\Lambda}(Z, \Omega_Z)$, $K^\conic(Z, \Omega_Z)\subset
K(Z, \Omega_Z)$
denote  the respective full dg subcategories of $\alpha$-conic objects.
In this way, we will be able to work with $\mu \Sh_{\Lambda}(\Omega_Z)$ concretely as a localization of 
$\Sh^\conic_{\Lambda}(Z, \Omega_Z)$ all at once, and in particular be in the local setting studied in detail in ~\cite[Ch.~VI]{KS}.  See Remark~\ref{rem biconic} for the precise situation we will encounter.
\end{remark}

\begin{remark}
We will primarily work with microlocal sheaves supported along a fixed
  closed conic Lagrangian subvariety $\Lambda \subset \Omega_Z$. 
 An inclusion $\Lambda\subset \Lambda'$ of such  induces a full embedding 
  $\mu\Sh_\Lambda(\Omega_Z) \subset \mu\Sh_{\Lambda'}(\Omega_Z)$.
It is sometimes convenient to not specify the support, for example if we have a collection
of $\Lambda \subset \Omega_Z$ in mind, and then we will write 
 $\mu\Sh(\Omega_Z)$ for the union of the dg categories $\mu\Sh_{\Lambda}(\Omega_Z)$ over all
such $\Lambda \subset \Omega_Z$ under consideration.
\end{remark}

\begin{example}\label{ex cod 1}
Suppose $Z = \BR$.
Inside of $T^*\BR \simeq \BR\times \BR$, introduce the conic Lagrangian subvariety and  conic open subspace 
$$
\xymatrix{
 \Lambda =   \BR \cup \{(0, \eta) \, |\, \eta > 0\}    &
  \Omega_Z   =    \{(t, \eta) \, |\, \eta>0\} 
 }
 $$

Then there are canonical equivalences
$$
\xymatrix{
\Perf_k \ar[r]^-\sim &  \Sh_\Lambda ( Z)^0_! \ar[r]^-\sim & \mu \Sh_\Lambda(\Omega_Z)
&
V \ar@{|->}[r] & j_{+*}p^*V
}
$$
induced by the correspondence
$$
\xymatrix{
pt & \ar[l]_-{p_+}  \BR_{<0} \ar@{^(->}[r]^-{j_+}  & \BR
}
$$

Similarly,  there are canonical equivalences
$$
\xymatrix{
\Perf_k \ar[r]^-\sim &  \Sh_\Lambda ( Z)^0_* \ar[r]^-\sim & \mu \Sh_\Lambda(\Omega_Z)
&
V \ar@{|->}[r] & j_{-!}p_-^! V
}
$$
induced by the correspondence
$$
\xymatrix{
pt & \ar[l]_-{p_-}  \BR_{<0} \ar@{^(->}[r]^-{j_-}  & \BR
}
$$

Furthemore, the composite functors are naturally equivalent
$$
\xymatrix{
j_{-!}p_-^! \simeq j_{+*}p_+^* :\Perf_k \ar[r]^-\sim &  \mu\Sh_\Lambda(\Omega_Z)
}
$$
An inverse equivalence is induced by the hyperbolic localization
$$
\xymatrix{
\phi: \Sh_\Lambda( Z) \ar[r] &\Perf_k  
&
\phi(\cF) = i_0^* i_+^! \cF
}
$$
with respect to the inclusions
$$
\xymatrix{
 i_0: X = \{0\}  \ar@{^(->}[r] &  \BR_{\geq 0} 
 &
 i_+:\BR_{\geq 0} \ar@{^(->}[r] & \BR
}
$$

The constructions $j_{+*}p_+^*$ and
$j_{-!}p_-^!$ provide  respective left and right  adjoints to the natural
 microlocalization functor 
 $$
\xymatrix{
 \Sh_\Lambda( Z) \ar[r] &  \mu\Sh_{\Lambda}(\Omega_Z) \simeq \Perf_k
 }
 $$ 
 realized by functorial equivalences
 $$
 \xymatrix{
 \Hom(j_{+*}p_+^*\cL, \cF) \simeq \Hom(\cL, \phi(\cF))
&
\Hom( \phi(\cF), V) \simeq \Hom(\cF, j_{-!}p_-^!V) 
 }
 $$
 
\end{example}

\begin{example}\label{ex cod 2}
Suppose $Z = \BR^2$. 

Suppose $g_\pm:\BR \to \BR$ are smooth functions
with $g_+(0) = 0$, $g_+(s) >0$, for $s>0$, and $g_-(s)= - g_+(s)$. We will only use their restrictions to $\BR_{\geq 0} \subset \BR$.

Inside of $T^*\BR^2 \simeq \BR^2 \times \BR^2$, introduce
 the conic open subspaces
 $$
\xymatrix{
\Omega_{Z, \pm} = \{(s, t), (\xi, \eta))\, |\, s>0, \eta> 0\} 
 }
 $$
and conic Lagrangian subvarieties
$$
\xymatrix{
\Lambda_\pm = \{(s, f_\pm(s)), (-\eta dg_\pm(s), \eta))\, |\, s>0, \eta> 0\} 
 }
 $$
Inside of $T_{(0,0)}^*\BR^2 \simeq \BR^2$, introduce the  cone
$$
\xymatrix{
\Lambda_0 = \Span_{\geq 0}( (-dg_+(0), 1), (-dg_-(0), 1)) 
 }
 $$
 Form the total conic Lagrangian subvariety and  conic open subspace 
$$
\xymatrix{
 \Lambda =  \BR^2 \cup \Lambda_+ \cup \Lambda_0 \cup \Lambda_-
 & 
  \Omega_Z   =    \{((s, t), (\xi, \eta) \, |\, \eta>0\} 
 }
 $$

Consider the  iterated inclusions 
$$
\xymatrix{
U \ar@{^(->}[r]^-u & V \ar@{^(->}[r]^-v & \BR^2
}
$$
$$
\xymatrix{
U = \{(s, t) \in \BR_{>0} \times \BR  \, | \, g_-(s) < t < g _+(s)\}
&
V = \{(s, t) \in \BR_{>0} \times \BR  \, | \, g_-(s) \leq t < g _+(s)\}
}
$$

Then there is a canonical equivalence
$$
\xymatrix{
\Perf_k  \ar[r]^-\sim & \mu \Sh_\Lambda(\Omega_X)
&
V \ar@{|->}[r] & v_!u_*V_U
}
$$
factoring through the coincident full dg subcategories
$\Sh_\Lambda ( \BR^2 )^0_!  = \Sh_\Lambda (  \BR^2 )^0_! \subset \Sh(\BR^2)$.

Finally, the open restrictions provide further equivalences
 $$
\xymatrix{
\mu \Sh_\Lambda(\Omega_X)\ar[r]^-\sim & \mu \Sh_{\Lambda_\pm}(\Omega_{X, \pm})
}
$$
Note that  each pair $\Lambda_\pm \subset \Omega_{Z, \pm}$ is locally modeled
on the pair of Example~\ref{ex cod 1}.
When we compare each composite equivalence
 $$
\xymatrix{
c_\pm:\Perf_k\ar[r]^-\sim &  \mu \Sh_\Lambda(\Omega_X)\ar[r]^-\sim & \mu \Sh_{\Lambda_\pm}(\Omega_{X, \pm})
}
$$
with the equivalence $c = j_{-!}p_-^! \simeq j_{+*}p_+^*$ of Example~\ref{ex cod 1},
we see that $c_-$ agrees with $c$, 
but $c_+$ agrees with  $c\otimes\orient_\BR[-1]$,
where we shift by $-1$ and twist by the orientation line $\orient_\BR$ of the second factor of $\BR^2$.
\end{example}


\subsubsection{Twisted symmetry}

Let us focus here on the setting of Remark~\ref{rem micro when conic},
and specifically the setting of Remark~\ref{rem biconic}.

Set $Z = X \times \BR = \BR^n\times\BR$,
and consider 
 the  conic open subspace 
$$
\xymatrix{
\Omega_X  = \{((x, t), (\xi, \eta))  \, |\, \eta>0\} \subset  T^*(X\times \BR) 
}
$$

Let $\Lambda\subset \Omega_X$ be a closed biconic Lagrangian subvariety in the sense of 
Remark~\ref{rem biconic}, so conic with respect to the positive scaling
of covectors, and also conic  
with respect to the commuting Hamiltonian scaling action 
induced by the scaling action on the base
$$
\xymatrix{
\alpha:\BR_{>0} \times X\times \BR \ar[r] & X\times \BR
&
\alpha(r, (x, t)) = (rx, r^2 t)
}
$$
Recall that the Hamiltonian scaling action contracts the pair
$\Lambda \subset \Omega_X$ to a neighborhood
of the positive codirection 
$$
\xymatrix{
\{((0, 0), (0, \eta)) \, |\, \eta>0\} \subset \Lambda\subset   \Omega_X
}
$$
Thus microlocal sheaves on $\Omega_X$ supported along $\Lambda$ can be represented by $\alpha$-conic
constructible sheaves on $X\times \BR$, or alternatively by their restrictions to any small open ball around the origin.

Next, recall the Hamiltonian $T$-action on $\Omega_X$ with moment map $\nu:\Omega_X \to \ft^*$ and action Lagrangian correspondence
$$
\xymatrix{
\cL_{T, \Omega_X} \subset \Omega_X \times \Omega_X^a\times T^*T 
}
$$
$$
\xymatrix{
\cL_{T, \Omega_X} = \{ (\omega_1, -\omega_2, (g, \zeta)) \in \Omega_X \times \Omega_X^a \times T^*T \, |\,\omega_1 =  g\cdot \omega_2,\,  \nu(\omega_1) = \zeta\}  
}
$$
and in particular, for $g\in T$, the action Lagrangian correspondence
$$
\xymatrix{
\cL_{g, \Omega_X} \subset \Omega_X \times \Omega_X^a
}
$$
$$
\xymatrix{
\cL_{g, \Omega_X} = \{ (\omega_1, -\omega_2) \in \Omega_X \times \Omega_X^a \, |\,\omega_1 =  g\cdot \omega_2\}
}
$$


The theory of microlocal kernels and transformations~\cite[Ch. VII]{KS}
provides, for each $g\in T$, an integral transform equivalence
$$
\xymatrix{
\Phi_g:\mu \Sh_{\Lambda}(\Omega_X) \ar[r]^-\sim & \mu \Sh_{g(\Lambda)}(\Omega_{X})
&
\Phi_g(\cF) = \cK_g\circ \cF
}
$$ 
following the notation of~\cite[Definition 7.1.3]{KS},
where the microlocal kernel $\cK_g$,  to be specified momentarily, is rank one along the smooth 
action Lagrangian correspondence $\cL_{g, \Omega_X}$.

We would like to highlight the twisted nature of compositions of the above equivalences.
First, for the identity $e\in T$, let us normalize $\cK_e$ so that $\Phi_e$ is the identity.
Next, for any $g\in T$, let us attempt to specify $\cK_g$ by continuity: for a path $\gamma_s:[0,1]\to T$,
with $\gamma_0 = e, \gamma_1 = g$, there is a unique $\cK_g(\gamma)$ 
given by parallel transporting $\cK_e$.
But for a loop $\gamma_s:[0,1]\to T$,
with $\gamma_0 = \gamma_1 = e$, we find that $\cK_e(\gamma)$  is not necessarily equivalent to $\cK_e$.

\begin{prop}\label{prop twist calc}
There is a canonical equivalence
$$
\xymatrix{
\cK_{e}(\gamma)  \simeq \cK_e[2\langle \delta, \ol \gamma\rangle]
}
$$
where $\delta\in \chi^*(T)$ is the diagonal character, $\ol \gamma \in \pi_1(T) \simeq \chi_*(T)$
is the class of $\gamma_s$, and we shift by twice their natural pairing.
\end{prop}



\begin{proof}
For simplicity, we will focus on the one-dimensional calculation, as arises for each 
coordinate circle $S^1\subset T$, and not carry along the additional fixed coordinate
directions.

Thus we set $X = \BR$ and consider $T= S^1 = \BR/2\pi\BZ$ acting on $\Omega_X \subset T^*(\BR^2)$.

 
Consider the continuous family of integral transform equivalences
$$
\xymatrix{
\Phi_{r}:\mu \Sh(\Omega_X) \ar[r]^-\sim & \mu \Sh(\Omega_{X})
&
r\in \BR
}
$$ 
normalized so that $\Phi_0$ is the identity, for $0 \in \BR$. 
We seek to show $\Phi_{2\pi}\simeq [2]$, where only the specific twist of the identity functor is in question.

It suffices to act upon  the $\Omega_X$-microlocalization $\cF$  of the constant sheaf $k_{X}$
along the first coordinate direction $X = \BR \times\{0\}$ and calculate what results. The singular support of
 $k_{X}$ is the conormal bundle $T^*_{X} \BR^2$, and we will denote by $\Lambda = T^*_{X} \BR^2 \cap \Omega_X$ its relevant part.
%

 Rotation by $\theta \not =\pm \pi/2 \in S^1$ takes $\Lambda \subset \Omega_X$ to the smooth conic Lagrangian surface
$$
\xymatrix{
\Lambda (\theta)=  \{(x, c x^2), (-\eta cx, \eta) \, |\, \, \eta>0\} \subset \Omega_X
}
$$
where we set
$c =  \sin(\theta) /\cos(\theta)$ from here on,
and rotation by $\theta = \pm \pi/2$ takes it to
the smooth conic Lagrangian surface
$$
\xymatrix{
\Lambda (\pm\pi/2)=  \{(0, 0), (  \eta y, \eta) \, |\,   \, \eta>0\} \subset \Omega_X
}
$$
Note that $\Lambda(\theta) =\Lambda(-\theta)$ since we happen to have chosen $\Lambda = \Lambda(0)$ to be invariant under rotation
by $\theta = \pi$.

Let $X(\theta) = \pi(\Lambda(\theta)) \subset \BR^2$ be the front projection.
For $\theta \not = \pm \pi/2$, it
is the parabola  
$$
\xymatrix{
X(\theta) = \{( x,  c x^2) \} \subset \BR^2
}
$$
and for $\theta = \pm \pi/2$,  it
is the origin $X( \pm \pi/2) = \{(0, 0)\}$.

Now for $r\in \BR$, with image $\theta\in S^1$,
 let us calculate the microlocal sheaf $\Phi_r(\cF)$. 
It  will be rank one along $\Lambda(\theta) \subset \Omega_X$,
with its particular twist what we seek. 

To start, recall that $\cF$ is represented by the constant sheaf $k_X$  along the first coordinate 
direction $X = \BR \times\{0\}$. Alternatively, following Example~\ref{ex cod 1},
it is also represented by the  extensions $j_{+*}k_{X_+}$ and $j_{-!}(k_{X_-}\otimes\orient_Y)[1]$
along the open inclusions
$$
\xymatrix{
j_+:X_+ = \BR \times\BR_{>0} \ar@{^(->}[r] & \BR^2
&
j_-:X_- = \BR \times\BR_{<0} \ar@{^(->}[r] & \BR^2}
$$
where $\orient_Y$ is  the line of orientations of the second coordinate direction $Y = \{0\} \times \BR$.

For $r\in (-\pi/2, \pi/2)$, with image $\theta\in S^1$, by continuity, $\Phi_r(\cF)$ is represented by the constant sheaf 
$k_{X(\theta)}$ on the parabola $X(\theta)$.
Alternatively,
it is also represented by the  extensions $j(\theta)_{+*}k_{X(\theta)_+}$ and $j(\theta)_{-!}(k_{X(\theta)_-}
\otimes\orient_Y)[1]$  
along the open inclusions
$$
\xymatrix{
j(\theta)_+:X(\theta)_+ = \{(x, t) \, |\, t> cx^2\} \ar@{^(->}[r] & \BR^2
&
j_-:X(\theta)_- =\{(x, t) \, |\, t< cx^2\} \ar@{^(->}[r] & \BR^2}
$$

When $r \to \pi/2$,   the representative $j(\theta)_{+*}k_{X(\theta)_+}$ 
limits to the extension
$i_{+*} k_{W_+}$
  along the inclusion of the ray
$$
\xymatrix{
i_+: W_+  =\{(0, t) \, |\, t>0\} \ar@{^(->}[r] & \BR^2
}
$$
To keep track of twists, it is worth noting the relation via the inverse Fourier-Sato transform~\cite[Definition 3.7.8]{KS}
in the first coordinate direction
$$
\xymatrix{
i_{+*} k_{W_+} \simeq (j_{+*}k_{X_+})^{\vee_X}
}
$$
as appears in \cite[Lemma 3.7.10]{KS}.
Observe as well that the  $\Omega_X$-microlocalization of $i_{+*} k_{W_+}$ is alternatively represented by the
 skyscraper $k_{(0,0)}$ at the origin.
  Thus we conclude that $\Phi_{\pi/2}(\cF)$ is represented by $k_{(0,0)}$.

Similarly, when $r \to -\pi/2$,   the representative $j(\theta)_{-!}(k_{X(\theta)_-}
\otimes\orient_Y)[1]$ limits to the extension
$i_{-!} (k_{W_-}\otimes\orient_{X\times Y})$
  along the inclusion of the ray
$$
\xymatrix{
i_-: W_-  =\{(0, t) \, |\, t<0\} \ar@{^(->}[r] & \BR^2
}
$$
where $\orient_{X\times Y}$ is  the line of orientations of $\BR^2 = X\times Y$.
Again there is the relation via the Fourier-Sato transform
in the first coordinate direction
$$
\xymatrix{
i_{-!}(k_{W_-}\otimes \orient_{X\times Y}) \simeq (j_{-!}(k_{X_-}\otimes\orient_Y)[1])^{\wedge_X}
}
$$
Observe as well that the  $\Omega_X$-microlocalization of $i_{-!} (k_{W_-}\otimes \orient_{X\times Y})$ 
is alternatively represented by the
 skyscraper sheaf $k_{(0,0)} \otimes \orient_X[-1]$ at the origin.
 Thus we conclude that $\Phi_{-\pi/2}(\cF)$ is represented by $k_{(0,0)} \otimes \orient_X[-1]$.

Therefore starting with  $\Phi_{-\pi/2}(\cF)$, and applying  $\Phi_{\pi}$, we obtain the
identity 
$$
\xymatrix{
\Phi_{\pi/2}(\cF) \simeq \Phi_{-\pi/2}(\cF)\otimes\orient_X[1]
}
$$
This can be viewed as a reflection of the standard identity~\cite[Proposition 3.7.12]{KS}
 for the square of the inverse Fourier-Sato transform in the first coordinate
direction
$$
\xymatrix{
k_{(0,0)}\otimes \orient_X[1] \simeq (k_{(0,0)})^{\vee_X\vee_X}
}
$$

Iterating this, we obtain the canonical equivalence $\Phi_{2\pi} \simeq [2]$ as asserted. 
This concludes the one-dimensional calculation, and higher-dimensional generalizations follow
by the same argument run independently on the relevant coordinate directions.
  \end{proof}

It is convenient to encode the above twist in the following form.
 Introduce  the $\BZ$-cover 
 $$
 \xymatrix{
 T' = T\times_{S^1} \BR \ar[r] & T
 }
 $$
defined by the diagonal character $\delta:T\to S^1$, and the universal cover $\BR\to S^1$.
  
 Then for $ g'\in  T'$, with image $g\in T$,  
we have unambiguous   integral transform equivalences
$$
\xymatrix{
\Phi_{g'}:\mu \Sh_{\Lambda}(\Omega_{X}) \ar[r]^-\sim & \mu \Sh_{g(\Lambda)}(\Omega_{X})
}
$$ 
obeying evident composition laws.
Furthermore, elements $m\in \BZ \simeq \ker(T'\to T) $ of the kernel  act by the invertible functor
$$
\xymatrix{
\Phi_{m}(\cF) \simeq \cF[2m]
&
\cF\in \mu \Sh_{\Lambda}(\Omega_{X})
}
$$

\begin{remark}
Following the literature on gradings in Fukaya categories, and specifically graded Lagrangians~\cite{SeidelGraded}, here is an intuitive way to think about the 
above twist.

Let $\kappa_{\Omega_X}$ be the complex canonical bundle of $\Omega_X$,
with respect to a compatible almost complex structure,
and  let $\kappa^{\otimes 2}_{\Omega_X}$ be its bicanonical bundle.
The embedding ${\Omega_X} \subset T^*(X\times\BR)$ provides a canonical trivialization 
$$
\xymatrix{
\tau_X:\kappa^{\otimes 2}_{\Omega_X} \ar[r] &  \BC
}
$$ 
by the top-exterior power of the tangent bundle of the zero-section $X\times\BR\subset T^*(X\times\BR)$.

Let $\Lambda_\sm \subset \Lambda$ denote the smooth locus,
so that we have the tangent bundle $T\Lambda_\sm \subset T \Omega_X$.
Taking the argument of $\tau_X$ applied to the top-exterior power
of $T\Lambda_\sm \subset T \Omega_X$ produces a phase 
$$
\xymatrix{
\varphi_X:\Lambda_\sm \to S^1
}
$$
Define the  grading $\BZ$-torsor  to be the base change under the phase
$$
\xymatrix{
\Lambda'_\sm = \Lambda_\sm \times_{S^1} \BR \ar[r] & \Lambda_\sm
}
$$ 

For a path $\gamma_s:[0,1]\to T$,
with $\gamma_0 = e, \gamma_1 = g$, there is an isomorphism
of grading $\BZ$-torsors 
$$
\xymatrix{
\Lambda'_\sm \ar[r]^-\sim &  g(\Lambda'_\sm)
}
$$
 And for a loop $\gamma_s:[0,1]\to T$,
with $\gamma_0 = \gamma_1 = e$, 
the resulting automorphism 
of the grading $\BZ$-torsor $\Lambda'_\sm$
is
 equal to translation by $2\langle \delta, \ol\gamma_s\rangle \in \BZ$.

This completely captures the twists on microlocal sheaves on $\Omega_X$ supported along $\Lambda$,
since the twists
are determined along the smooth locus $\Lambda_\sm$.
\end{remark}

Finally, it is useful to expand the scope of the above symmetry beyond individual group elements.
Note that the $T'$-action on $\Omega_Z$, via the cover $T'\to T$,  is encoded 
by the action Lagrangian correspondence
$$
\xymatrix{
\cA_{T', \Omega_X}  = \cA_{T, \Omega_X}  \times_T T' \subset \Omega_X \times \Omega_X^a\times T^*T'
}
$$

Let $\Sh_c(T') \subset \Sh(T')$ be the  full  dg subcategory of objects with compact support.
Then we have a monoidal convolution action
$$
\xymatrix{
\star:\Sh_c(T') \otimes \mu\Sh(\Omega_X) \ar[r] & \mu\Sh(\Omega_X)
}
$$ 
$$
\xymatrix{
\cA\star\cF = \Phi_{\cK}(\cA \boxtimes \cF ) = \cK\circ (\cA \boxtimes \cF ) 
&
\cA\in \Sh_c(T'),
\cF \in   \mu\Sh(\Omega_X)
}
$$  
where the microlocal kernel $\cK$ is a rank one local system along the smooth action Lagrangian
correspondence $\cA_{T', \Omega_X}$,
normalized so that  the monoidal unit $\cA_0 = k_e \in \Sh_c(T')$ acts by the identity functor.
One recovers the prior symmetries for group elements by convolving with skyscraper sheaves at points.

Let 
$\add:\BZ \times T'\to T'$ denote the translation action by the kernel $Z \simeq \ker(T'\to T)$.
Returning to the twists discussed above, for $m\in \BZ$, there is an equivalence of monoidal actions
$$
\xymatrix{
(\add(m)_*\cA) \star \cF \simeq (\cA[2m])\star \cF
&
\cA\in \Sh_c(T'),
\cF \in   \mu\Sh(\Omega_X)
}
$$  
To see this concretely, one can express objects of $\Sh_c(T')$ in terms of objects defined on fundamental domains for the $\BZ$-cover $T'\to T$, and then translate by elements of the kernel for which we have already calculated the twist.

\begin{remark}

To concisely formalize the above structure, one could introduce the dg category $\tau\Sh(T)$ of twisted constructible sheaves as the 
$\BZ$-coinvariants
of $\Sh_c(T')$ where for $m\in \BZ$, the translation $\add(m)_*$ is identified with the cohomological shift  $[2m]$.
Then the above $\Sh_c(T')$-action on $\mu\Sh(\Omega_X)$ factors through the natural map 
$\Sh_c(T') \to \tau\Sh(T)$. 

Informally speaking, objects of $\tau\Sh(T)$ are constructible sheaves with grading defined with respect to the nonconstant background bicanonical trivialization given by the diagonal character $\delta:\pi_1(T)\to S^1$.
This is the  bicanonical trivialization arising by restricting the constant bicanonical trivialization from $M = \BC^n$ to the unit torus $T = (S^1)^n \subset \BC^n = M$.
\end{remark}

%
%
%
%
Finally, note that the kernel $T^\circ \subset T$ of the diagonal character $\delta\in \chi^*(T)$
admits a canonical lift $T^\circ \subset T'$ since the cover $T'\to T$ is defined by $\delta\in \chi^*(T)$.
Pushforward along the canonical lift $T^\circ \subset T'$ provides a monoidal embedding $\Sh(T^\circ) \subset \Sh_c(T')$,
and we may restrict the  above monoidal convolution  functor to a monoidal  action
$$
\xymatrix{
\star:\Sh(T^\circ) \otimes \mu\Sh(\Omega_X) \ar[r] & \mu\Sh(\Omega_X)
}
$$

\subsection{Nearby and vanishing categories}\label{ss cat def}

Now we will form the dg category of  microlocal sheaves on the the exact symplectic manifold $M= \BC^n$ supported  
 along the Lagrangian skeleton 
$$
\xymatrix{
L(\theta) \subset M & \theta\in S^1
}
$$
or more generally, along the finite union of skeleta 
$$
\xymatrix{
L(\Theta) = \bigcup_{\theta\in \Theta} L(\theta) \subset M
&
\Theta\subset S^1
}$$

Set $X= \BR^n$, and $Z = X\times \BR = \BR^n \times \BR$.

Recall that in Section~\ref{ss prelim}, we constructed a conic open subspace $\Omega_X \subset T^*(X \times \BR)$.
Furthermore, recall that  in Definition~\ref{def assoc lag}, 
to a conic Lagrangian subvariety $L \subset M$, we 
associated a biconic Lagrangian subvariety $\Lambda \subset \Omega_X$.
Applying this to the Lagrangian skeleton $L(\theta) \subset M$, we obtain a biconic Lagrangian subvariety denoted by
$$
\xymatrix{
\Lambda(\theta) \subset \Omega_X & \theta\in S^1
}
$$
or more generally, applying this to the finite union $L(\Theta) \subset M$, we obtain a  biconic Lagrangian subvariety denoted by
$$
\xymatrix{
\Lambda(\Theta) = \bigcup_{\theta\in \Theta} \Lambda(\theta) \subset \Omega_X
&
\Theta\subset S^1
}
$$

\begin{defn}[Vanishing category]\label{def vanishing cat}
For $\theta\in S^1$, define the vanishing category
$$ 
\xymatrix{
\mu \Sh_{L(\theta)}(M)
}$$ 
to be the dg category $\mu \Sh_{\Lambda(\theta)}(\Omega_X)$ 
 of microlocal sheaves on $\Omega_X$
supported along $\Lambda(\theta)$.

More generally, for  finite nonempty $\Theta\subset S^1$, define the multi-vanishing category
$$
\xymatrix{
\mu \Sh_{L(\Theta)}(M)
}$$ 
to be the dg category $\mu \Sh_{\Lambda(\Theta)}(\Omega_X)$ 
 of microlocal sheaves on $\Omega_X$
supported along $\Lambda(\Theta)$.

\end{defn}

We will similarly form 
 the dg category of  microlocal sheaves on the exact symplectic manifold
$M^\times = (\BC^\times)^n$ supported along
the Lagrangian skeleton
$$
\xymatrix{
L^\times(\theta) \subset M^\times & \theta\in S^1
}
$$
via the corresponding conic open subspace $\Omega_X^\times \subset T^*(X\times \BR) \setminus (X\times \BR)$,
and biconic Lagrangian subvariety 
$$
\xymatrix{
\Lambda^\times (\theta) = \Lambda(\theta) \cap \Omega_X^\times  & \theta\in S^1
}
$$

%
%
%

\begin{defn}[Nearby category]
For $\theta\in S^1$, define the nearby category
$$
\xymatrix{
 \mu \Sh_{L^\times(\theta)}(M^\times)
}
$$ 
to be the dg category $\mu \Sh_{\Lambda^\times(\theta)}(\Omega_X^\times)$ 
 of microlocal sheaves on $\Omega_X^\times$
supported along $\Lambda^\times(\theta)$.
\end{defn}

The main goal of this paper is to calculate the vanishing category $\mu \Sh_{L(\theta)}(M)$
in terms of the much simpler nearby category $\mu \Sh_{L^\times(\theta)}(M^\times)$.
There is an evident restriction functor
$$
\xymatrix{
\mu \Sh_{L(\theta)}(M)\ar[r] & \mu \Sh_{L^\times(\theta)}(M^\times)
}
$$
and our main technical results will construct and characterize its adjoints.

To tackle the nearby category, let us first establish the following lemma 
which treats the case $\theta = 0$
and then appeal to monodromy equivalences in general.

\begin{lemma}\label{lem open equiv}
There is an equivalence
$$
\xymatrix{
 \Sh_{\Lambda_\Sigma} (T^\circ) \ar[r]^-\sim &
\mu \Sh_{L^\times(0)}(M^\times)
}
$$
\end{lemma}

\begin{proof}
Recall the exact symplectic identification
$$
\xymatrix{
\varphi: M^\times  = (\BC^\times)^n \ar[r]^-\sim &  T^{>0}((S^1)^n)  = T^{>0} T
}
$$
and the transported conic Lagrangian
$$
\xymatrix{
L^{>0}(0)  =\varphi(L^{\times}(0) ) 
}$$
We seek an equivalence
$$
\xymatrix{
\Sh_{\Lambda_\Sigma} (T^\circ) \ar[r]^-\sim & 
 \mu \Sh_{L^{>0}(0)}(T^{>0} T)
}
$$

Recall the inclusion $i:T^\circ \to T$ induces a natural Lagrangian correspondence
$$
\xymatrix{
\ar[d]^\sim T^* T^\circ & \ar@{->>}[l]_-p \ar[d]^\sim  T^* T \times_T T^\circ \ar@{^(->}[r]^-i &\ar[d]^-\sim  T^* T \\
T^\circ \times (\ft^\circ)^* & \ar@{->>}[l]  T^\circ\times\ft^* \ar@{^(->}[r] & T\times \ft^* 
}
$$
compatible with the natural projection  $\ft^* \to \ft^*/\Span(\{\delta\}) \simeq (\ft^\circ)^*$.

Recall  that $L^{>0}(0)\subset T^* T \times_T T^\circ$, 
and the projection $L^{>0}(0)\to \Lambda^\circ$ is simply a $\Span_{>0}(\{\delta\})$-bundle.
Therefore by Lemma~\ref{lemma ham red}, pushforward along the inclusion 
$$
\xymatrix{
i_*: \Sh(T^\circ) \ar[r] &  \Sh(T)
}
$$
induces the desired functor, with inverse  induced by the hyperbolic localization
$$
\xymatrix{
\phi_\delta: \Sh(T) \ar[r] &\Sh(T^\circ)  
&
\phi_\delta(\cF) = i_0^* i_+^! \cF
}
$$
with respect to the inclusions
$$
\xymatrix{
 i_0:T^\circ \ar@{^(->}[r] & T[0, \epsilon) 
 &
 i_+:T[0, \epsilon) \ar@{^(->}[r] & T
}
$$
where 
$T[0, \epsilon) = f^{-1}([0, \epsilon)
$
 for any function smooth function $f:T\to\BR$ with $T^\circ = f^{-1}(0)$, $df|_{T^\circ} = e$, and
 sufficiently  small $\epsilon>0$.  
\end{proof}

%

Coupling the lemma with the case of the coherent-constructible correspondence recalled in Example~\ref{ex ccc} gives the following.

\begin{corollary}\label{cor open equiv}
The nearby category for $\theta = 0$ admits a mirror equivalence
$$
\xymatrix{
\mu \Sh_{L^\times(0)}(M^\times) \ar[r]^-\sim &   \Coh(\BP^{n-1})
}
$$
\end{corollary}


\subsection{Symmetry and monodromy}

Recall the torus $T \simeq (S^1)^n$ and subtorus $i:T^\circ \to T$.

%
%
%

Following Example~\ref{ex torus tensor cat}, recall  that $\Sh(T)$ is a tensor dg category with respect to convolution,
and pushforward 
induces a tensor embedding
$
\Sh(T^\circ) \subset \Sh(T).
$

We will study here how appropriate objects of $\Sh(T)$ act on the nearby category,
vanishing category, and more generally, on the  multi-vanishing category. 
%
 Recall by the constructions of Section~\ref{ss prelim}
 and definitions of Section~\ref{ss cat def}, we set $X= \BR^n$, and take these categories to comprise suitable 
microlocal sheaves
 on the biconic open subspace $\Omega_X \subset T^*(X\times \BR)$.
%


Following the discussion of Section~\ref{ss micro shs},
introduce the $\BZ$-cover $T'\to T$ defined by the diagonal character $\delta\in \chi^*(T)$,
 the  tensor dg category $\Sh_c(T')$ of constructible sheaves on $T'$ with compact support, 
and the monoidal action
 $$
\xymatrix{
\star: \Sh_c(T') \otimes \mu\Sh(\Omega_X) \ar[r] & \mu\Sh(\Omega_X)
}
$$ 
Recall  there is an equivalence of monoidal actions
$$
\xymatrix{
(\add(m)_*\cA) \star \cF \simeq (\cA[2m])\star \cF
&
\cA\in \Sh_c(T'),
\cF \in   \mu\Sh(\Omega_Z)
}
$$  
where
$m\in \BZ \simeq \ker(T'\to T)$, and
$\add:\BZ \times T'\to T'$ is the translation action.
Recall as well the natural lift $T^\circ \subset T'$ provides a tensor embedding
$\Sh(T^\circ) \subset \Sh_c(T')$ 
allowing us to restrict the above monoidal action.

%
%


\subsubsection{Symmetry}

To apply the above symmetries to specific objects of $\Sh_c(T')$, we need to know they respect supports.
Here we will focus on the tensor subcategory
$\Sh(T^\circ) \subset \Sh_c(T')$,
and
following Example~\ref{ex ccc}, the further tensor subcategory $\Sh_{\Lambda_\Sigma}(T^\circ) \subset \Sh(T^\circ)$.

\begin{lemma}\label{lem symmetry}
For $\theta\in S^1$,
the convolution action of $\Sh_{\Lambda_\Sigma}(T^\circ)$ preserves the nearby category
$\Sh_{L^\times(\theta)}(M^\times)$ and vanishing category
$ \mu \Sh_{L(\theta)}(M)$,
and is compatible with restriction
$$
\xymatrix{
\mu \Sh_{L(\theta)}(M)\ar[r] & \mu \Sh_{L^\times(\theta)}(M^\times)
}
$$

More generally, for finite nonempty $\Theta\subset S^1$, 
the convolution action of $\Sh_{\Lambda_\Sigma}(T^\circ)$ preserves the 
multi-vanishing category
$ \mu \Sh_{L(\Theta)}(M)$,
and is compatible with restriction
$$
\xymatrix{
\mu \Sh_{L(\theta)}(M)\ar[r] & \bigoplus_{\theta\in \Theta} \mu \Sh_{L^\times(\theta)}(M^\times)
}
$$
\end{lemma}

\begin{proof}
 The $T$-action preserves $M^\times = W^{-1}(\BC^\times)$, so convolution is compatible with restriction. 

Recall that $M_0 = W^{-1}(0)$ is the union of the coordinate hyperplanes in $M=\BC^n$,
and hence does not contain any $n$-dimensional isotropic submanifolds. Thus for $\theta\in S^1$, if 
convolution by objects of $\Sh_{\Lambda_\Sigma}(T^\circ)$ preserves the nearby category $\mu\Sh_{L^\times(\theta)}(M^\times)$, 
then it will also preserve the vanishing category $\Sh_{L(\theta)}(M)$, since $L(\theta)$ is the closure of  $L^\times(\theta)$. 
Moreover, if 
convolution by objects of $\Sh_{\Lambda_\Sigma}(T^\circ)$ preserves the vanishing category $\mu\Sh_{L^\times(\theta)}(M^\times)$, for any $\theta\in S^1$,
then it will also preserve the multi-vanishing category $\mu\Sh_{L(\Theta)}(M)$, for  finite nonempty
$\Theta\subset S^1$, since  $L(\Theta)$ is the union of $L(\theta)$, for $\theta\in \Theta$. 
Thus it suffices to show that convolution by objects of $\Sh_{\Lambda_\Sigma}(T^\circ)$  
preserves the nearby category $\mu\Sh_{L^\times(\theta)}(M^\times)$, for fixed $\theta \in S^1$.

Recall the exact symplectic identification
$$
\xymatrix{
\varphi: M^\times  = (\BC^\times)^n \ar[r]^-\sim &  T^{>0}((S^1)^n)  = T^{>0} T
}
$$
and the transported conic Lagrangian
$$
\xymatrix{
L^{>0}(\theta)  =\varphi(L^{\times}(\theta) ) 
}$$

Since $\varphi$ is $T$-equivariant, it suffices to show that convolution by objects of $\Sh_{\Lambda_\Sigma}(T^\circ)$  
preserves the category $\mu \Sh_{L^{>0}(\theta)}(T^{>0}T)$.
More concretely, it suffices to see that  the correspondence induced by multiplication 
takes  $\Lambda_\Sigma \times L^{>0}(\theta) \subset T^*T^\circ \times T^{>0} T$ back into $L^{>0}(\theta) \subset T^{>0} T$.

To confirm this, recall   the decompositions
$$
\xymatrix{
L^{>0}(\theta) = \bigcup_{\fI} \fI S \times \fI \sigma \subset T \times \ft^*
&
\Lambda_\Sigma = \bigcup_{\sigma}  \sigma T^\circ \times \sigma \subset T^\circ \times (\ft^\circ)^*
}
$$
and that the index sets are matched by a nonempty subset  $\fI = \{ 1, \ldots, n\}$ determining the cone  
$\sigma = \Span_{>0}(\{e_a \, |\, a\not \in \fI\})  \subset \Sigma$. Furthermore,  the cones in the second factors are compatible under the projection
$\ft^* \to \ft^*/\Span(\delta) \simeq (\ft^\circ)^*$ in the sense that
$
\sigma = \fI \sigma /\Span_{>0}(\delta).
$

Since the positive cones of $\Sigma$ are disjoint, it remains to check  for fixed $\sigma \subset \Sigma$,
and corresponding $\fI$, the multiplication of $\fI S$ by elements of $\sigma T^\circ$ lies back within $\fI S$.
But recall that $\fI S$ is cut out by $\theta_a = 0$, for $a\not\in \fI$, and 
$\sum_{a=1}^n \theta_a = \theta$,
and $\sigma T^\circ$ 
is cut out by $\theta_a = 0$, for $a\not\in \fI$, and $\sum_{a=1}^n \theta_a = 0$.
\end{proof}

\begin{remark}\label{rem tensor comp}
In the special case $\theta = 0$,
note that the canonical equivalence
$$
\xymatrix{
 \Sh_{\Lambda_\Sigma} (T^\circ) \ar[r]^-\sim &
\mu \Sh_{L^\times(0)}(M^\times)
}
$$
of Lemma~\ref{lem open equiv}  is naturally compatible with
the convolution action of $\Sh_{\Lambda_\Sigma}(T^\circ)$
since by construction,
it is induced by pushforward along the inclusion $i:T^\circ \to T$.
\end{remark}


\subsubsection{Monodromy}

Recall the $\BZ$-cover $T'\to T$ defined by the diagonal character $\delta\in \chi^*(T)$.
Note the canonical identification of Lie algebras $\ft' = \ft \simeq \BR^n$, and let us write $q':\ft' \to \ft'/\ker(\delta) \simeq T'$ for the natural map.
%

For $\tau\in \BR^\times$, let $\sign(\tau)\in \{\pm 1\}$ be its sign. Consider the
inclusion  $i_\tau:\Delta(\tau)\to \ft'$
of the relatively open simplex 
$$
\xymatrix{
\Delta(\tau) = \{(\tau_1, \ldots, \tau_n) \in \ft' \, |\,  \sign(\tau) \tau_a >0, \mbox{ for all } a = 1, \ldots, n, \mbox{ and }\sum_{a=1}^n \tau_a = \tau\}
}
$$

Let $\cA_0 = k_{e}\in \Sh_c(T')$ be the skyscraper at the identity $e\in T'$.
For $\tau >0$, let $\cA_\tau\in \Sh_c(T')$  be the pushforward of the $*$-extension of the constant sheaf 
$$
\xymatrix{
\cA_\tau = q'_* i_{\tau*} k_{\Delta(\tau)}
}
$$
For $\tau < 0$, let $\cA_\tau\in \Sh_c(T')$ be the pushforward of the  $!$-extension of the Verdier dualizing complex
$$
\xymatrix{
\cA_\tau = q'_* i_{!} \omega_{\Delta(\tau)}
}
$$

Note the canonical convolution equivalences
$$
\xymatrix{
\cA_{\tau_1} \star \cA_{\tau_2} \simeq\cA_{\tau_1 +\tau_2} 
}
$$
and in particular that $\cA_\tau$ is invertible with inverse
$$
\xymatrix{
\cA_{\tau}^\vee \simeq \iota_*\mathbb D_{T'}\cA_\tau \simeq \cA_{-\tau}
}
$$

\begin{lemma}\label{lem par trans}
For $\tau \in \BR$, and $\theta \in S^1= \BR/2\pi \BZ$,
convolution with $\cA_\tau\in \Sh_c(T')$ provides monodromy equivalences
$$
\xymatrix{
\cA_\tau\star:\mu \Sh_{L(\theta)}(M) \ar[r]^-\sim& \mu \Sh_{L(\theta+\tau)}(M)
&
\cA_\tau\star:\mu \Sh_{L^\times(\theta)}(M^\times) \ar[r]^-\sim& \mu \Sh_{L^\times(\theta+\tau)}(M^\times)
}
$$
fitting into a commutative diagram with restriction
$$
\xymatrix{
\ar[d]\mu \Sh_{L(\theta)}(M) \ar[r]^-\sim & \mu \Sh_{L(\theta+\tau)}(M)\ar[d]\\
\mu \Sh_{L^\times(\theta)}(M^\times) \ar[r]^-\sim & 
\mu \Sh_{L^\times(\theta+\tau)}(M^\times)
}
$$


More generally, for $\tau \in \BR$, and finite nonempty  $\Theta \subset S^1= \BR/2\pi \BZ$,
convolution with $\cA_\tau\in \Sh_c(T')$ provides a monodromy equivalence
 $$
\xymatrix{
{\cA_\tau\star}:\mu \Sh_{L(\Theta)}(M) \ar[r]^-\sim& \mu \Sh_{L(\Theta+\tau)}(M)\\
}
$$
fitting into a commutative diagram with restriction
$$
\xymatrix{
\ar[d]\mu \Sh_{L(\Theta)}(M) \ar[r]^-\sim & \mu \Sh_{L(\Theta+\tau)}(M)\ar[d]\\
\bigoplus_{\theta\in \Theta}
\mu \Sh_{L^\times(\theta)}(M^\times) \ar[r]^-\sim & 
\bigoplus_{\theta\in \Theta}
\mu \Sh_{L^\times(\theta+\tau)}(M^\times)
}
$$
%
%
\end{lemma}

\begin{proof}
Convolution by  $\cA_\tau \in \Sh_c(T')$ is invertible with inverse given by convolution by
 the dual $\cA_{-\tau}  \simeq \cA_\tau^\vee \in \Sh_c(T')$. Thus the lemma follows if convolution by  $\cA_\tau \in \Sh_c(T')$ maps the stated categories to the respective stated categories.
 
 As in the proof of Lemma~\ref{lem symmetry}, it suffices to establish the assertion for the nearby category in the form 
 $$
\xymatrix{
\cA_\tau\star:\mu \Sh_{L^{>0}(\theta)}(T^{>0} T) \ar[r] & \mu \Sh_{L^{>0}(\theta+\tau)}(T^{>0} T)
}
$$
Moreover, by composition of convolutions, it suffices to assume $\theta = 0$, and $\tau\in [0, 2\pi)$, and establish
the assertion for 
 $$
\xymatrix{
\cA_\tau\star:\mu \Sh_{L^{>0}(0)}(T^{>0} T) \ar[r] & \mu \Sh_{L^{>0}(\tau)}(T^{>0} T)
}
$$

Since $\Sh_c(T')$ is a tensor category, convolution by $\cA_\tau\in \Sh_c(T')$ commutes in particular with convolution by objects
of $\Sh_{\Lambda_\Sigma}(T^\circ)$. Hence
by Lemma~\ref{lem symmetry} and Remark~\ref{rem tensor comp}, it suffices to see 
$$
\xymatrix{
\ssupp(\cA_\tau)  \cap T^{>0} T \subset L^{>0}(\tau)  
}
$$
But by Lemma~\ref{lemma micro interpretation slice}, and the conventions of Example~\ref{ex conv},
 we have that
$$
\xymatrix{
\ssupp(\cA_\tau)  \cap T^{>0} T = P^{>0}(\tau) \subset L^{>0}(\tau) 
}
$$
\end{proof}

Thanks to Corollary~\ref{cor open equiv} and  Lemma~\ref{lem par trans}, we have the following generalization of Corollary~\ref{cor open equiv}.
Note that the equivalence obtained here is not canonical since it depends on the choice of $\tau \in \BR$ through
its appearance in Lemma~\ref{lem par trans}. 

\begin{corollary}\label{cor open equiv gen}
Given $\theta\in S^1 = \BR/2\pi\BZ$, for the choice of a lift $\tau\in \BR$,
the nearby category admits a mirror equivalence
$$
\xymatrix{
\mu \Sh_{L^\times(\theta)}(M^\times) \ar[r]^-\sim & \Coh(\BP^{n-1})
}
$$
\end{corollary}

Finally, let us record the ambiguity of the equivalence of the corollary by analyzing what happens when $\theta = 0\in S^1= \BR/2\pi\BZ$
and $\tau\in 2\pi \BZ$.

When $\tau = 2\pi$, note that $\add(-1)_*\cA_{2\pi} \in \Sh_c(T')$ is supported on $T^\circ\subset T'$ 
and in fact
 $\add(-1)_*\cA_{2\pi}\in\Sh_{\Lambda_\Sigma}(T^\circ)$.
Recall that it corresponds to $\cO_{\BP^{n-1}}(-1) \in \Coh(\BP^{n-1})$
 under the equivalence of Example~\ref{ex ccc}. Thus we 
 have the following.


\begin{corollary}\label{cor monodromy calc}
Fix $\tau = 2\pi m \in  2\pi\BZ$.
Under the  equivalence %
of Corollary~\ref{cor open equiv},
convolution by $\cA_\tau\in \Sh_{c}(T')$ corresponds to tensoring with $\cO_{\BP^{n-1}}(-m)[2m]$.
\end{corollary}


\subsection{Adjoints to restriction}

We now arrive at the main technical result of this paper.


Let us focus on the skeleta $L(0), L(\pi/2) \subset M$ 
over the respective real rays $\BR_{\geq 0}, \BR_{\leq 0} \subset \BC$,
and  simplify our previous notation by setting
$$
\xymatrix{
L^\times_- = L^\times(\pi) 
&
L^\times_+ = L^\times(0) 
}
$$
$$
\xymatrix{
L_- = L^\times_- \cup L_0 
&
L_+ = L^\times_+ \cup L_0
&
L = 
L^\times_- \cup L_0 \cup L^\times_+
}
$$
 
Recall that the open embeddings  
$$
\xymatrix{
L^\times_-  \ar@{^(->}[r]^-{J_-} &  L & \ar@{_(->}[l]_-{J_+}   L^\times_+
}
$$
induce restriction functors
$$
\xymatrix{
\mu\Sh_{L^\times_-}(M^\times) & \ar[l]_-{J^*_-} \mu\Sh_{L}(M) \ar[r]^-{J^*_+} &  \mu\Sh_{L^\times_+}(M^\times) 
}
$$

Here is our main technical result which will be proved in this section.

\begin{thm}\label{thm adjoints}
1) The restriction functors $J_-^*, J_+^*$ admit fully faithful left and right adjoints fitting into adjoint triples
$$
\xymatrix{
 (J_{-!}, J_-^*, J_{-*})
&
(J_{+!}, J_+^*, J_{+*})
}
$$
and intertwining the natural convolution actions of $ \Sh_{\Lambda_\Sigma}(T^\circ)$.

3) The compositions 
$$
\xymatrix{
J^*_+J_{-*} &
J^*_-J_{+*}
}$$ are equivalences
with respective inverses the adjoint compositions  
$$
\xymatrix{
J^*_-J_{+!}
&
J^*_+J_{-!}
}
$$

4) The composition $J_+^* J_{-!} J^*_-J_{+!}$ is equivalent to convolution with $\cA_{2\pi}[-2]
\in \Sh_c(T')$.
\end{thm}

\begin{remark}
By Corollary~\ref{cor monodromy calc},
under the equivalence
$$
\xymatrix{
\mu \Sh_{L^\times_+}(M^\times) \ar[r]^-\sim & \Coh(\BP^{n-1})
}
$$
convolution with $\cA_{2\pi}[-2]\in \Sh_c(T')$ is given by 
tensoring with 
 $ \cO_{\BP^{n-1}}(-1)\in  \Coh(\BP^{n-1}).
 $
\end{remark}

The proof of the theorem will occupy the rest of this section.
To begin, let us use symmetry to simplify the assertion.

Consider the object $\cA_+  \in \mu\Sh_{L^\times_+}(M^\times)$,
corresponding 
to $\cA_0\in \Sh_{\Lambda_{\Sigma}}(T^\circ)$,
under the equivalence of Lemma~\ref{lem open equiv},
and the object
$\cA_- = \cA_\pi\star \cA_+\in \mu\Sh_{L^\times_-}(M^\times)$. 
 
Note their endomorphisms are scalars,
and they provide equivalences
$$
\xymatrix{
\Sh_{\Lambda_\Sigma}(T^\circ) \ar[r]^-\sim &  \mu\Sh_{L^\times_\pm}(M^\times)
&
\cF\ar@{|->}[r] & \cF\star\cA_\pm
}
$$ 

Introduce the fully faithful embeddings
$$
\xymatrix{
\cY_{\pm}:\Perf_k \ar@{^(->}[r] &  \mu\Sh_{L^\times_\pm}(M^\times)
&
\cY_\pm (V) = V\otimes \cA_\pm
}
$$
Note we have adjoint triples $(\cY^\ell_\pm, \cY_\pm, \cY^r_\pm)$ with adjoints given by
$$
\xymatrix{
\cY^\ell_\pm(\cF) = \Hom(\cF, \cA_\pm)^\vee
&
\cY^r_\pm(\cF) = \Hom(\cA_\pm, \cF)
}
$$

 Introduce the commutative diagram
$$
\xymatrix{
 \ar[d]_-{\cY^\ell_{-}} \mu\Sh_{L^\times_-}(M^\times) & \ar[l]_-{J^*_-} 
\ar[dl]^-{j_-^*} \mu\Sh_{L}(M) 
\ar[dr]_-{j_+^*}
\ar[r]^-{J^*_+} &  \mu\Sh_{L^\times_+}(M^\times)
 \ar[d]^-{\cY^r_{+}}  
 \\
\Perf_{k} &&\Perf_{k}
}
$$
where  we set $j^*_- = \cY^\ell_-  J_-^*$, $j^*_+ = \cY^r_+  J_+^*$.

\begin{prop}\label{prop reduction}
Suppose the restriction functors $j_-^*, j_+^*$ fit into adjoint pairs
$$
\xymatrix{
 ( j_-^*,  j_{-*})
&
(j_{+!},  j_+^*)
}
$$
the canonical maps  are equivalences
$$
\xymatrix{
J^*_-j_{-*}    \ar[r]^-\sim &  \cY_-
&
 \cY_+ \ar[r]^-\sim &J^*_+j_{+!}
}
$$
and there is an equivalence
$$
\xymatrix{
J^*_-j_{+!} \simeq \cY_-\otimes \ell[-1]
}
$$
where $\ell$ is a square-trivial line $\ell^{\otimes 2} \simeq k$.

Then the conclusions of Theorem~\ref{thm adjoints} hold.
\end{prop}

\begin{remark}
We include the  line $\ell$
and
 its square-trivialization $\ell^{\otimes 2} \simeq k$
  in the formulation and in what follows since it arises naturally as an orientation line. But
the validity of the proposition is independent of its appearance since we do not specify
 any characterizing or universal properties of the equivalence it participates in.
\end{remark}

\begin{proof} 
Let $k \in \Perf_k$ denote the rank one vector space. 

For $\cF\in \Sh_{\Lambda_\Sigma}(T^\circ)$, set
$$
\xymatrix{
\cF_- = \cF\star \cA_- \in \mu\Sh_{L^\times_-}(M^\times)
&
\cF_+ = \cF\star \cA_+ \in \mu\Sh_{L^\times_+}(M^\times)
}
$$
 and define candidate adjoints 
$$
\xymatrix{
J_{-*}(\cF_-) = \cF\star j_{-*}(k)
&
J_{+!}(\cF_+) = \cF\star j_{+!}(k)
}
$$
Note they evidently intertwine the natural convolution actions of  $\Sh_{\Lambda_\Sigma}(T^\circ)$.
Once we confirm they provide adjoints, we will have that they are fully faithful since
$$
\xymatrix{
J_-^*J_{-*}(\cF_-) = J_-^*(\cF\star j_{-*}(k))\simeq \cF\star J^*_-(j_{-*}(k))  \simeq \cF\star \cA_- = \cF_-
}
$$
$$
\xymatrix{
J_+^*J_{+!}(\cF_+) = J_+^*(\cF\star j_{+!}(k))\simeq \cF\star J^*_+(j_{+!}(k))  \simeq \cF\star \cA_+ = \cF_+
}
$$
using the assumed canonical equivalences  $J^*_-j_{-*}\simeq  \cY_-$, $ \cY_+\simeq J^*_+j_{+!}$.

Now to see they provide adjoints, for $\cG\in \mu\Sh_{L}(M)$, we  calculate 
$$
\xymatrix{
\Hom( \cG, J_{-*}(\cF_-)) = \Hom(\cG, \cF\star j_{-*}(k))
\simeq 
\Hom(\cF^\vee \star\cG,  j_{-*}(k))
}
$$
$$
\xymatrix{
\simeq \Hom(j_-^*(\cF^\vee \star\cG),  k)
\simeq
\Hom(\Hom(J_-^*(\cF^\vee \star\cG), \cA_-)^\vee,  k)
}
$$
$$
\xymatrix{
\simeq
\Hom(J_-^*(\cF^\vee \star\cG), \cA_-)
\simeq
 \Hom(\cF^\vee \star J_-^*(\cG), \cA_-)
}
$$
$$
\xymatrix{
\simeq \Hom(J_-^*(\cG),  \cF\star\cA_-)
\simeq \Hom( J_-^*(\cG),  \cF_-)
}
$$
and similarly calculate
$$
\xymatrix{
\Hom(J_{+!}(\cF_+), \cG) = \Hom(\cF\star j_{+!}(k), \cG)
\simeq 
\Hom( j_{+!}(k), \cF^\vee \star\cG)
}
$$
$$
\xymatrix{
\simeq \Hom(k, j_+^*(\cF^\vee \star\cG))
\simeq  \Hom(k, \Hom(\cA_-, J_+^*(\cF^\vee \star\cG)))
}
$$
$$
\xymatrix{
\simeq
\Hom(\cA_+, J_+^*(\cF^\vee \star\cG))
\simeq  \Hom(\cA_+, \cF^\vee \star J_+^*(\cG))
}
$$
$$
\xymatrix{
\simeq \Hom(\cF\star \cA_+, J_+^*(\cG)  )
\simeq \Hom( \cF_+, J_+^*(\cG))
}
$$

Next to see that $J^*_- J_{+!}$
 is an equivalence, and so with inverse equivalence its right adjoint
 $J^*_+ J_{-*}$, 
we calculate
$$
\xymatrix{
    J^*_- J_{+!}(\cF\star \cA_+) = J^*_- ( \cF\star j_{+!} (k)) 
\simeq
  \cF\star  J^*_- (j_{+!} (k)) 
\simeq  \cF\star  \cA_-\otimes\ell[-1]
}
$$
using the assumed equivalence $J^*_-j_{+!} \simeq \cY_-\otimes\ell[-1]$.
For later use, note in particular $J^*_- J_{+!}(\cA_+)\simeq \cA_-\otimes\ell[-1]$.

Finally, convolution with $\cA_\pi$, provides evident equivalences
$$
\xymatrix{
J^*_-(\cG) \simeq \cA_\pi\star J^*_+(\cA_\pi^\vee\star\cG)
&
J^*_+(\cG) \simeq \cA_\pi \star J^*_-(\cA_\pi^\vee\star\cG)
}
$$ 
and thus we have the other fully faithful adjoints
$$
\xymatrix{
J_{-!}(\cF_-) \simeq \cA_\pi \star J_{+!}(\cA_\pi^\vee\star\cF_-)
&
J_{+*}(\cF_+) \simeq \cA_\pi\star J_{-*}(\cA_\pi^\vee\star\cF_+)
}
$$
Moreover, 
 $J^*_+ J_{-!}$ is an equivalence,   and so with inverse equivalence its right adjoint $J^*_- J_{+*}$, 
 since
$$
\xymatrix{
J^*_+ J_{-!}(\cF_-)\simeq \cA_\pi\star J^*_-(\cA_\pi^\vee\star \cA_\pi\star J_{+!} (\cA_\pi^\vee\star \cF_-))
\simeq \cA_\pi\star J^*_- J_{+!} (\cA_\pi^\vee\star \cF_-)
}
$$
exhibits it as a composition of equivalences.
Note in particular that 
$$
\xymatrix{
J^*_+ J_{-!}(\cA_-)
\simeq \cA_\pi\star J^*_- J_{+!} (\cA_\pi^\vee\star \cA_-))
\simeq \cA_\pi\star J^*_- J_{+!} (\cA_+)
\simeq \cA_\pi\star  \cA_-\otimes\ell[-1]
}
$$
using the previously noted identity $J^*_-J_{+!}(\cA_+) \simeq \cA_-\otimes\ell[-1]$.

Using the previously noted identities
 $J^*_-J_{+!}(\cA_+) \simeq \cA_-\otimes\ell[-1]$ and  
   $J^*_+J_{-!}(\cA_-) \simeq \cA_\pi\star \cA_-\otimes\ell[-1]$,
   and the given isomorphism $\ell^{\otimes 2} \simeq k$,
we have equivalences
$$
\xymatrix{
J_+^* J_{-!} J^*_-J_{+!}(\cA_+) \simeq
J_+^* J_{-!} (\cA_-\otimes\ell[-1])}
\simeq \cA_\pi \star \cA_- [-2]\simeq \cA_{2\pi}\star \cA_+[-2]
$$ Since all of the functors intertwine convolution by objects of $\Sh_{\Lambda_\Sigma}(T^\circ)$,
this establishes  the last asserted equivalence.
\end{proof}

Now to prove Theorem~\ref{thm adjoints}, we will verify the assumptions of Proposition~\ref{prop reduction}.


Let us simplify our prior  notation by setting
$$
\xymatrix{
P^\times_- = P^\times(\pi)  \subset L_-^\times
&
P^\times_+ = P^\times(0)  \subset L_+^\times
}
$$
$$
\xymatrix{
P_- = P^\times_-  \cup L_0 \subset L_-
&
P_+ = P^\times_+ \cup L_0 \subset L_+
&
P =  P^\times_- \cup L_0 \cup P^\times_+ \subset L
}
$$

Recall the homeomorphism
$$
\xymatrix{
h:P \ar[r]^-\sim & \BR^n
}$$
along with its restrictions
$$
\xymatrix{
h_- = h|_{P^\times_-}:P^\times_- \ar[r]^-\sim & \BR^{n-1} \times \BR_{<0}
&
h_+ = h|_{P^\times_+}:P^\times_+ \ar[r]^-\sim & \BR^{n-1} \times \BR_{>0}
}
$$
Thus restriction gives equivalences
$$
\xymatrix{
 \mu\Sh_{P^\times_-}(M^\times)  & \ar[l]_-\sim  \mu\Sh_{P}(M) \ar[r]^-\sim & \mu\Sh_{P_+^\times} (M^\times)
}
$$

Assume for the moment  there is an object 
$$
\xymatrix{
\cA\in \mu\Sh_P(M) \subset \mu\Sh_L(M)
}
$$ 
whose restrictions satisfy
$$
\xymatrix{
\cA|_{P^\times_+} \simeq \cA_+ \in \mu\Sh_{P_+^\times}(M) \subset \mu\Sh_{L_+^\times}(M)
&
\cA|_{P^\times_-} \simeq \cA_-\otimes\ell[-1] \in \mu\Sh_{P_-^\times}(M) \subset \mu\Sh_{L_-^\times}(M)
}
$$
where $\ell$ is a square-trivial line $\ell^{\otimes 2} \simeq k$.
Note that such an object $\cA$, if it exists, must be unique up to equivalence.

Recall the fully faithful embeddings
$$
\xymatrix{
\cY_{\pm}:\Perf_k \ar[r]^-\sim   &\mu\Sh_{P^\times_\pm}(M^\times)   \ar@{^(->}[r] &  \mu\Sh_{L^\times_\pm}(M^\times)
&
\cY_\pm (V) = V\otimes \cA_\pm
}
$$
and introduce the  fully faithful  embedding 
$$
\xymatrix{
\cY:\Perf_k \ar[r]^-\sim &  \mu\Sh_{P}(M)\ar@{^(->}[r] &  \mu\Sh_{L}(M) 
&
\cY (V) = V\otimes \cA
}
$$
%

Set $j_{+!} = \cY$, $j_{-*} = \cY\otimes \ell[1]$ so
that
by assumption, there are canonical equivalences
$$
\xymatrix{
J^*_-j_{-!} \simeq \cY_- 
&
 J^*_+j_{+*} \simeq \cY_+
 &
 J^*_-j_{+!} \simeq \cY_-\otimes\ell[-1]
}
$$


Thus the following will allow us to invoke Proposition~\ref{prop reduction} and  
in turn establish Theorem~\ref{thm adjoints}.

\begin{thm}\label{thm core}
There is an object 
$$
\xymatrix{
\cA\in \mu\Sh_P(M) \subset \mu\Sh_L(M)
}
$$ 
whose restrictions satisfy
$$
\xymatrix{
\cA|_{P^\times_+} \simeq \cA_+ 
&
\cA|_{P^\times_-} \simeq \cA_-\otimes\ell[-1] 
}
$$
where $\ell$ is a square-trivial line $\ell^{\otimes 2} \simeq k$.

Furthermore, 
for $\cF\in \mu\Sh_L(M)$, there are functorial equivalences 
$$
\xymatrix{
 \Hom(J_-^*\cF, \cA_-) \simeq \Hom( \cF, \cA\otimes\ell[1]) 
&
 \Hom(\cA, \cF)  \simeq\Hom(\cA_+, J^*_+\cF) 
}
$$
\end{thm}

\begin{proof}
It is convenient to realize the symmetry between $L^\times_-= L^\times(\pi)$ and $L^\times_+
= L^\times(0)$
in a more explicit geometric form. Though convolution by $\cA_\pi$ gives an equivalence 
$$
\xymatrix{
\mu\Sh_{L^\times_+}(M^\times) \ar[r]^-\sim &
\mu\Sh_{L^\times_-}(M^\times)
}
$$
the underlying spaces $L^\times_-, L^\times_+$ are not even homeomorphic.
This is due to the special nature of the angle $0\in S^1$, and the resulting special nature of 
$ L^\times_+ $. Thus we will ``rotate"
all of our constructions by $-\pi/2$ and replace the  angles $0, \pi \in S^1$ with the generic angles
$-\pi/2, \pi/2 \in S^1$. 

To this end, let us simplify our prior notation by setting
$$
\xymatrix{
iL^\times_{-} = L^\times( \pi/2) 
&
iL^\times_{+} = L^\times( -\pi/2) 
}
$$
$$
\xymatrix{
iL_{-} = L^\times_- \cup L_0 
&
iL_{+} = L^\times_+ \cup L_0 
&
iL =  iL^\times_- \cup L_0 \cup iL^\times_+
}
$$%
%
%
%
%
and similarly 
$$
\xymatrix{
iP^\times_- = P^\times( \pi/2)  \subset iL_-^\times
&
iP^\times_+ = P^\times(- \pi/2)  \subset iL_+^\times
}
$$
$$
\xymatrix{
iP_- = P^\times_-\cup L_0  \subset iL_-
&
iP_+ = P^\times_+\cup L_0  \subset iL_+
&
iP =  iP^\times_- \cup L_0 \cup iP^\times_+ \subset iL
}
$$

\begin{remark}
We caution the reader that the above $\pm$ subscripts 
are chosen to be compatible starting from our prior $\pm$ subscripts and ``rotating" by $-\pi/2$,
but they are not compatible with the standard conventions for positive and negative imaginary numbers. 
For example, starting with  $L_{+}$ over the positive real ray $\BR_{\geq 0}\subset \BC$ and  ``rotating" by $-\pi/2$ leads to what we denote by $i L_{+}$
though it lies over the negative imaginary ray $i\BR_{\leq 0}\subset \BC$.
\end{remark}



By Lemma~\ref{lem par trans}, convolution with $\cA_{-\pi/2}$ provides canonical equivalences compatible with restriction
$$
\xymatrix{
\ar[d] \mu\Sh_{L}(M) \ar[r]^-\sim & \mu\Sh_{iL}(M)\ar[d]
&
\ar[d] \mu\Sh_{P}(M) \ar[r]^-\sim & \mu\Sh_{iP}(M)\ar[d]\\
\mu\Sh_{L^\times_\pm}(M^\times) \ar[r]^-\sim & \mu\Sh_{iL^\times_{\pm}}(M^\times)
&
\mu\Sh_{P^\times_\pm}(M^\times) \ar[r]^-\sim & \mu\Sh_{iP^\times_{\pm}}(M^\times)\\
}
$$

Consider the objects 
$$
\xymatrix{
\cB_+ = \cA_{-\pi/2}\star\cA_+ \in \mu\Sh_{iL^\times_+}(M^\times)
&
\cB_- = \cA_{-\pi/2}\star\cA_- \simeq \cA_{\pi/2}\star \cA_+ \in \mu\Sh_{iL^\times_-}(M^\times)
}
$$

It suffices to show there is an object 
$$
\xymatrix{
\cB\in \mu\Sh_{iP}(M) \subset \mu\Sh_{iL}(M)
}
$$ 
whose restrictions satisfy
$$
\xymatrix{
\cB|_{iP^\times_+}  \simeq \cB_+ 
&
\cB|_{iP^\times_-}  \simeq \cB_-\otimes\ell[-1] 
}
$$
where $\ell$ is a square-trivial line $\ell^{\otimes 2} \simeq k$,
and
such that for $\cF\in \mu\Sh_{iL}(M)$, there are functorial equivalences 
$$
\xymatrix{
 \Hom(J_-^*\cF, \cB_-) \simeq \Hom( \cF, \cB\otimes\ell[1]) 
&
 \Hom(\cB, \cF)  \simeq\Hom(\cB_+, J^*_+\cF) 
}
$$

We will explicitly construct $\cB$ by working with the specific Legendrian fibration
introduced in Section~\ref{ss prelim}
and finding a constructible sheaf that represents $\cB$.

Let us rapidly recall some of our prior constructions.

Set $X = \BR^{n}$ with coordinates $x_a$, for $a=1,\ldots, n$, and recall the linear Lagrangian fibration
$$
\xymatrix{
p: M = \BC^n \ar[r] & \BR^n = X & p(z_1, \ldots, z_n) = (x_1, \ldots, x_n)
}
$$
given by taking real parts, and its lift  to a Legendrian fibration 
$$
\xymatrix{
\displaystyle
q:  N = \BC^n \times \BR  \ar[r] & \BR^n \times \BR  = X\times \BR & q(z_1, \ldots, z_n, t) = (x_1, \ldots, x_n, t
+\frac{1}{2} \sum_{a=1}^n x_a y_a)
}
$$

 Recall the open subspace  
$$
\xymatrix{
\Upsilon_X = \{(x, t), [\xi, \eta]) \, |\, \eta>0\}   \subset S^\oo (X \times \BR)
}
$$ 
and the cooriented contactomorphism
$$
\xymatrix{
\psi:N\ar[r]^-\sim &  \Upsilon_X
}
$$
$$
\xymatrix{
\psi(z_1, \ldots, z_n, t) = ((x_1, \ldots, x_n), t+\frac{1}{2} \sum_{a=1}^n x_a y_a), [-y_1, \ldots, -y_n, 1])
}
 $$
 intertwining the Legendrian projection $q:N\to X\times \BR$ and the natural projection $\Upsilon_X \to X\times \BR$.

Recall  the symplectification of  $\Upsilon_X \subset S^\oo( X \times \BR)$
in the form of the  biconic open subspace 
$$
\xymatrix{
\Omega_X  = \{((x, t), (\xi, \eta))  \, |\, \eta>0\} \subset  T^*(X\times \BR) \setminus (X\times \BR)
}
$$

Following Definition~\ref{def assoc lag}, we associate to the conic Lagrangian subvarieties  $iL, i P_\pm, iP\subset M$
the respective biconic 
Lagrangian subvarieties $i\Lambda, i\Pi_\pm, i\Pi \subset \Omega_X$.
 Recall the biconic property encodes  invariance
 under the usual cotangent fiber scaling as well as
under the Hamiltonian  action 
induced by the scaling action
$$
\xymatrix{
\alpha:\BR_{>0} \times X\times \BR \ar[r] &  X \times \BR & 
\alpha(r,  (x, t)) = (rx, r^2 t)
}
$$

%
%
%
%
%

In order to construct $\cB$, we will record some elementary properties of  $i P_\pm\subset M$ and their behavior under the Legendrian projection 
$q:N\to X\times \BR$.
Analogous properties of $ i\Pi_\pm \subset \Omega_X $ will immediately hold
for the natural projection $\Omega_X \to X\times \BR$ thanks to the fact that
 the contactomorphism $\psi$
 interwines
   $q:N\to X\times \BR$ with the natural projection $\Upsilon_X \to X\times \BR$
  and $ i\Pi_\pm \subset \Omega_X $  are inverse images under the natural map $\Omega_X \to \Upsilon_X$.

  Introduce the closed positive quadrant 
$$ 
\xymatrix{
Q = \BR^n_{\leq 0} \subset \BR^n = X
}
$$
and more generally, for $\fJ \subset \{1, \ldots, n\}$, 
 the locally closed submanifold 
$$
\xymatrix{
\fJ Q\subset Q
}
$$ 
cut out by the equations $x_a = 0$, for $a\in \fJ$, and $x_a > 0$, for $a\not = 0$. 
Note that $\fJ Q$  is the interior of $Q$, when $\fJ = \emptyset$, and the union $\coprod_{|J| >0 } Q_J$
is the boundary $\partial Q$.

The restriction of $p$ to the isotropic subvariety $L_0 \subset M$ provides a homeomorphism
$$
 \xymatrix{
 L_0 \ar[r]^-\sim & \partial Q \subset X
 }
 $$
 and more precisely, diffeomorphisms
$$
\xymatrix{
 \fJ L_0 \ar[r]^-\sim & \fJ Q\subset X
 &
 |\fJ| >0
 }
 $$
 
 The restriction of $p$ 
to the Lagrangian subvariety $iP_\pm \subset M$ has image 
 $$
 \xymatrix{
 p(iP_\pm) = Q\subset X
 }
 $$
   and the further restriction 
  $$
\xymatrix{
 iP_\pm|_{\fJ Q} \ar[r] & \fJ Q\subset X
 }
 $$ 
is a diffeomorphism, when $|\fJ| \not = 1$,
and 
  a fibration with interval fibers, when $|\fJ | = 1$.

The restriction of $q$ 
 to the isotropic subvariety $L_0 \subset M$ also provides a homeomorphism
$$
 \xymatrix{
 L_0 \ar[r]^-\sim & \partial Q  \times \{0\}\subset X\times \BR
 }
 $$
 and more precisely, diffeomorphisms
$$
\xymatrix{
 \fJ L_0 \ar[r]^-\sim & \fJ Q\times \{0\}\subset X\times \BR
 &
 |\fJ| >0
 }
 $$

 The restriction of $q$ 
to the Lagrangian subvariety $iP_\pm \subset M$ has image  the graph
 $$
 \xymatrix{
 q(iP_\pm) = \Gamma_\pm 
 \subset X\times \BR
 }
 $$
 of a function 
$
f_\pm:Q\to \BR
$
such that 
$$
\xymatrix{
f_+\leq 0 & f_+|_{\partial Q}  = 0 & 
f_- = - f_+
}
$$

The explicit form of $f_\pm$ will not be important, but let us for example confirm the property $f_+\leq 0$.
By the definition of $q$, we have $f_+ =  \sum_{a=1}^n x_a y_a$ when evaluated on
$iP_+ \subset M$, and by the definition of $iP_+ \subset M$, it lies inside the locus  of points with $x_a \geq 0, y_a\leq 0$,
for $a=1, \ldots, n$.

Following across $\psi$, the restriction of  $\pi_{X\times \BR}:T^*(X \times \BR) \to X\times \BR$ to 
the Lagrangian subvariety $i\Pi_\pm \subset \Omega_X $ 
has image  the same graph
 $$
 \xymatrix{
 \pi_{X\times \BR} (i\Pi_\pm) = \Gamma_\pm \subset X\times \BR
 }
 $$
 
Let us describe the projection $i\Pi_\pm\to \Gamma_\pm$ in microlocal terms.
When $|\fJ| \not = 1$, over $\fJ Q \subset Q$, we find the positive codirection
within the conormal line bundle
$$
\xymatrix{
i\Pi_\pm |_{\fJ Q}  =  \{(x, f_\pm (x) ), (-rdf_\pm (x), r)) \, |\, x\in \fJ Q, r\in \BR_{> 0}\} \subset T^*_{\Gamma_\pm} (X\times \BR) 
}
$$ 
When $|\fJ| = 1$, over $\fJ Q \subset Q$, we find the positive two-dimensional cone bundle
$$
\xymatrix{
i\Pi_\pm |_{\fJ Q}  =  \{(x, f_\pm (x) ), (-rdf_\pm (x), s)) \, |\, x\in \fJ Q, r, s\in \BR_{\geq  0}, r+ s\in \BR_{>0} \} 
}
$$

Now consider the subspaces
$$
\xymatrix{
U = \{ (x, t) \in Q \times \BR \, |\, f_+(x) < t <f_-(x)\} 
}
$$
$$
\xymatrix{
V= \{ (x, t) \in Q \times \BR \, |\, f_+(x) \leq  t < f_-(x)\}
}
$$
and
 their iterated inclusions
$$
\xymatrix{
U\ar@{^(->}[r]^-u & V \ar@{^(->}[r]^-v & X\times \BR
}
$$

Let $\cL_U$ be a locally constant sheaf on $U$, and form  the iterated extension 
$$
\xymatrix{
\tilde \cB = v_! u_* \cL_{U} \in \Sh( X\times \BR)
}
$$
 Following the standard conventions recalled in Example~\ref{ex conv}, observe that 
 $$
 \xymatrix{
 \ssupp(\tilde \cB)  = i\Pi  \cup \ol{U}
 &
 \ssupp(\tilde \cB) \cap \Omega_X = i\Pi
 }
 $$ 
 
 Set $\cB\in \mu\Sh_{iP}(M)$ to be the object represented by $\tilde \cB\in\Sh(X\times\BR)$. 
  Following Example~\ref{ex cod 1}, we may normalize $\cL_U$ in order to have  the agreement
$$
\xymatrix{
\cB|_{iP^\times_+}  \simeq \cB_+ 
}
$$

It remains to show there is an equivalence 
$$
\xymatrix{
\cB|_{iP^\times_-}  \simeq \cB_-\otimes\ell[-1] 
}
$$
for a square-trivial line $\ell$, and
for $\cF\in \mu\Sh_{iL}(M)$, there are functorial equivalences 
$$
\xymatrix{
 \Hom(J_-^*\cF, \cB_-) \simeq \Hom( \cF, \cB\otimes\ell[1]) 
&
 \Hom(\cB, \cF)  \simeq\Hom(\cB_+, J^*_+\cF) 
}
$$

Recall the family of  conic Lagrangian subvarieties $P(\tau)\subset M$, for $\tau \in (-2\pi, 2\pi)$, for which 
 $iP_+ = P(-\pi/2)$,
$iP_- = P(\pi/2) $. 
Introduce the associated biconic Lagrangian subvarieties  $\Pi(\tau) \subset \Omega_X$,
for $\tau \in (-2\pi, 2\pi)$, for which 
 $i\Pi_+ = \Pi(-\pi/2)$,
$i\Pi_- = \Pi(\pi/2) $. 

In what follows, we will restrict the  parameter to assume that  $\tau \in (-\pi/2, \pi/2)$ to interpolate between 
the points of focus 
$\tau = \pm \pi/2$.

%
%

Generalizing the prior discussion, 
we find that the restriction of  $\pi_{X\times \BR}:T^*(X \times \BR) \to X\times \BR$ to 
the Lagrangian subvariety $\Pi(\tau) \subset \Omega_X $ 
has image  the graph
 $$
 \xymatrix{
 \pi_{X\times \BR} (\Pi(\tau)) = \Gamma_\tau \subset X\times \BR
 }
 $$
  of a function 
$
f_\tau:Q\to \BR
$
such that $f_\tau\leq  0$, when $\tau\leq 0$, and in general
$$
\xymatrix{
f_\tau|_{\partial Q}  = 0 & 
f_{-\tau} = - f_\tau
}
$$
  
In microlocal terms, 
 the projection $\Pi(\tau)\to \Gamma(\tau)$ 
is uniformly the positive codirection
within the conormal line bundle
$$
\xymatrix{
\Pi(\tau)  =  \{(x, f_\tau (x) ), (-rdf_\tau (x), r)) \, |\, x\in Q, r\in \BR_{> 0}\} \subset T^*_{\Gamma_\tau} (X\times \BR) 
}
$$ 


Next, for  a pair $\tau_1 < \tau_2 \in [-\pi/ 2, \pi/2]$, consider the subspaces
$$
\xymatrix{
U(\tau_1, \tau_2) = \{ (x, t) \in Q \times \BR \, |\, f_{\tau_1}(x) < t <f_{\tau_2}(x)\} 
}
$$
$$
\xymatrix{
V(\tau_1, \tau_2) = \{ (x, t) \in Q \times \BR \, |\, f_{\tau_1}(x) \leq t < f_{\tau_2}(x)\} 
}
$$
and
 their iterated inclusions
$$
\xymatrix{
U(\tau_1, \tau_2) \ar@{^(->}[rr]^-{u(\tau_1, \tau_2)} && V(\tau_1, \tau_2) \ar@{^(->}[rr]^-{v(\tau_1, \tau_2)}  && X\times \BR
}
$$

Set $\cL_{U(\tau_1, \tau_2)} = \cL_U|_{U(\tau_1, \tau_2)}$, and
introduce the object 
$$
\xymatrix{
\tilde \cB (\tau_1, \tau_2) =v(\tau_1, \tau_2)_!u(\tau_1, \tau_2)_* \cL_{U(\tau_1, \tau_2)} \in \Sh( X\times \BR)
}
$$
and note that 
$$
 \xymatrix{
 \ssupp(\tilde \cB)  = \Pi(\tau_1) \cup \Pi(\tau_2)  \cup \ol{U(\tau_1, \tau_2)}
 &
 \ssupp(\tilde \cB) \cap \Omega_X =  \Pi(\tau_1) \cup \Pi(\tau_2)
 }
 $$ 
 Set $\cB (\tau_1, \tau_2)\in \mu\Sh_{P(\tau_1) \cup P(\tau_2)}(M)$ to be the object represented by 
 $\tilde \cB (\tau_1, \tau_2)\in\Sh(X\times\BR)$.

Note the agreement $\tilde \cB = \tilde \cB (-\pi/2, \pi/2)$,
so that  for $\tau_1 = -\pi/2$, and any $\tau_2\in (-\pi/2, \pi/2]$, we have in particular
$$
\xymatrix{
 \cB (-\pi/2, \tau_2)|_{P^\times(\tau_1)}  \simeq \cB_+ 
}
$$
Thus by continuity in $\tau_1$, for any $\tau_1 <\tau_2 \in [-\pi/2, \pi/2]$, 
we have 
$$
\xymatrix{
 \cB (\tau_1, \tau_2)|_{P^\times(\tau_1)}  \simeq \cA_{\tau_1 + \pi/2}\star \cB_+ 
}
$$
Thus fixing $\tau_2 = \pi/2$, and following Example~\ref{ex cod 2}, we have
$$
\xymatrix{
 \cB (\tau_1, \pi/2)|_{P^\times(\pi/2)}  \simeq 
 \cA_{\pi}\star \cB_+\otimes\ell[-1]  \simeq
 \cB_-\otimes\ell[-1] 
}
$$
for the square-trivial line $\ell = \orient_\BR$ of orientations on the second factor of the base $X\times \BR$.

Finally, 
for  small $\epsilon>0$, and any $\tilde \cF\in \Sh_{i\Lambda}(X\times\BR, \Omega_X)$, 
representing
 $\cF\in \mu\Sh_{i\Lambda}(\Omega_X)$,
note that 
$\tilde \cB (-\pi/2, -\pi/2+\epsilon)$
represents the microlocal restriction to $i\Pi_+^\times \subset \Omega^\times_X$,
as discussed in Examples~\ref{ex cod 1} and~\ref{ex cod 2},
in the sense of a functorial equivalence
$$
\xymatrix{
 \Hom(\tilde \cB (-\pi/2, -\pi/2+\epsilon), \tilde \cF) \simeq \Hom( \cB_+, \cF)
}
$$

For any $\tau \in (-\pi/2, \pi/2)$, we have the key property 
$P^\times(\tau) \cap iL =  \emptyset$, 
and hence $\Pi^\times(\tau) \cap i\Lambda = \emptyset$.
Thus for any $\tau \in (-\pi/2, \pi/2)$, and $\tilde \cF\in \Sh_{i\Lambda}(X\times\BR, \Omega_X)$, we have 
a non-characteristic propagation equivalence,
 highlighted with $\dag$ in the following sequence
$$
\xymatrix{
\Hom(\tilde \cB,  \cF) = \Hom( u_! v_* \cL_{U},  \tilde \cF) 
\simeq \Hom(v_* \cL_{U},  u^!\tilde\cF)  
}
$$
$$
\xymatrix{
\simeq^{\dag} \Hom(v(-\pi/2, \tau)_*k_{U(-\pi/2, \tau)},   u(-\pi/2, \tau)^!\tilde\cF) 
}
$$
$$
\xymatrix{
\simeq
 \Hom(u(-\pi/2, \tau)_!v(-\pi/2, \tau)_*k_{U(-\pi/2, \tau)},   \tilde\cF) 
= \Hom(\tilde \cB (-\pi/2, \tau), \tilde\cF)
}
$$

Write $\tilde \cF \in \on{Ind} \Sh_{i\Lambda}(X\times \BR, \Omega_X)$
 for the ind-object  representing the right adjoint of the microlocalization
 of $\cF\in \mu\Sh_{iL}(M)$.
Then we can assemble a functorial equivalence 
$$
\xymatrix{
\Hom( \cB, \cF) 
\simeq
\Hom(\tilde \cB,\tilde \cF)
\simeq
 \Hom(\tilde \cB (-\pi/2, -\pi/2+\epsilon), \tilde \cF) 
 \simeq 
\Hom( \cB_+, \cF)
}
$$
We leave it the reader to  obtain an analogous functorial equivalence
$$
\xymatrix{
 \Hom(J_-^*\cF, \cB_-) \simeq \Hom( \cF, \cB\otimes\ell[1]) 
}
$$
by a similar argument or by duality.
This concludes the proof of the theorem.
\end{proof}


\subsection{Spherical  structure}
Let us return to the setting of Theorem~\ref{thm adjoints}, in particular 
the skeleta over the real rays $\BR_{\geq 0}, \BR_{\leq 0} \subset \BC$,
as  organized by the simplified notation
$$
\xymatrix{
L^\times_- = L^\times(\pi) 
&
L^\times_+ = L^\times(0) 
}
$$
$$
\xymatrix{
L_- = L^\times_- \cup L_0 
&
L_+ = L^\times_+ \cup L_0
&
L = 
L^\times_- \cup L_0 \cup L^\times_+
}
$$

The closed embeddings  
$$
\xymatrix{
L_-  \ar@{^(->}[r] &  L & \ar@{_(->}[l]   L_+ 
}
$$
induce fully faithful embeddings
$$
\xymatrix{
\mu\Sh_{L_-}(M)  \ar@{^(->}[r]^-{I_{-!}} &  \mu\Sh_{L}(M) & \ar@{_(->}[l]_-{I_{+!}}   \mu\Sh_{L_+}(M) 
}
$$
and we identify 
$\mu\Sh_{L_-}(M),   \mu\Sh_{L_+}(M)
$
with their images.
%
%

\begin{lemma}\label{lemma micro cons}
Inside of $\mu\Sh_{L}(M)$, we have
$$
\xymatrix{
\mu\Sh_{L_-}(M) = \ker(J_+^*) 
&
\mu\Sh_{L_+}(M) = \ker(J_-^*) 
&
\mu\Sh_{L_-}(M)  \cap \mu_{L_+}(M) = \{0\} 
}
$$
In particular, the compositions $J_+^* I_{-!}$, $J_-^* I_{+!}$ are conservative.
\end{lemma}

\begin{proof}
By definition, if a microlocal sheaf vanishes on an open subset, then its microsupport lies in the closed complement.
This proves the first two identities. For the third, recall that the dimension of the intersection $L_- \cap L_+ = L_0$ is less than $n = (\dim M)/2$ so does not support any nontrivial microlocal sheaves. Finally, 
the identities imply the kernels of $J_+^* I_{-!}$, $J_-^* I_{+!}$ vanish and so they are conservative. 
\end{proof}

\begin{thm}\label{thm main spherical}  The diagram of restriction functors
$$
\xymatrix{
\mu\Sh_{L^\times_-}(M^\times) & \ar[l]_-{J^*_-} \mu\Sh_{L}(M) \ar[r]^-{J^*_+} &  \mu\Sh_{L^\times_+}(M^\times) 
}
$$
forms a conservative spherical pair.
\end{thm}

\begin{proof}
Immediate from Lemma~\ref{lemma cons sph pair}, Theorem~\ref{thm adjoints}, and Lemma~\ref{lemma micro cons}.
\end{proof}

Recall the open embedding 
$$
\xymatrix{
L^\times_+  \ar@{^(->}[r] &  L_+
}
$$
with corresponding restriction functor
$$
\xymatrix{
 \mu\Sh_{L_+}(M) 
  \ar[r]^-{J^*} & 
 \mu\Sh_{L^\times_+}(M^\times) 
}
$$

\begin{corollary}  Restriction is a conservative spherical functor
$$
\xymatrix{
\mu\Sh_{L_+}(M) \ar[r]^-{J^*} &  \mu\Sh_{L^\times_+}(M^\times) 
}
$$
\end{corollary}

\begin{proof}
Immediate from Proposition~\ref{prop KS}, Lemma~\ref{lemma micro cons}, and Theorem~\ref{thm main spherical}.
\end{proof}


\section{Mirror symmetry}

Recall the dual torus $\check T^\circ = \Spec \BC[\chi_*(T^\circ)]$, and the fan $\Sigma\subset (\ft^\circ)^*$ determining the $\check T^\circ$-toric variety
$\mathbb P^{n-1}$. 

Consider the section 
$$
\xymatrix{
s:\cO_{\BP^{n-1}} \ar[r] & \cO_{\BP^{n-1}}(1)
&
s([x_1, \ldots, x_n]) = x_1 + \cdots + x_n
}$$
and the inclusion of its zero-locus
$$
\xymatrix{
i:\BP^{n-2} \simeq \{ s= 0\} \ar@{^(->}[r] & \BP^{n-1} 
}$$
The specific coefficients of $s$ will not not be important only the $\check T^\circ$-invariant fact that they are all non-zero.

%

\begin{thm}\label{thm main}
There is a commutative diagram with horizontal equivalences
$$
\xymatrix{
\ar[d]_{J^*} \mu\Sh_{L_+}(M) \ar[r]^-\sim & \Coh(\BP^{n-2})\ar[d]^-{i_*} \\
 \mu\Sh_{L^\times_+}(M^\times) \ar[r]^-\sim & \Coh(\BP^{n-1})
}
$$
\end{thm}

\begin{proof}
We will study the monad $A = J^*J_!$ of the adjunction $(J_!, J^*)$.

By Theorem~\ref{thm adjoints} and the spherical functor formalism, under the equivalence 
$$
\xymatrix{
 \mu\Sh_{L^\times_+}(M^\times) \ar[r]^-\sim & \Coh(\BP^{n-1})
}
$$
the monad $A = J^*J_!$ is given by tensoring with the cone of a morphism
$$
\xymatrix{
\cO_{\BP^{n-1}}(-1) \ar[r]^-s & \cO_{\BP^{n-1}}
}
$$

 Now let us calculate the morphism $s$. For each $\alpha= 1, \ldots, n$, let us focus on the Lagrangian skeleton $L_+ \subset M$ near the coordinate vector $e_\alpha  \in L_+$ with $z_a = 1$, for $a=\alpha$, and $z_a = 0$, for $a\not =\alpha$.  
 Observe that $L_+$ locally near $e_\alpha$ is homeomorphic to $\BR_{\geq 0} \times \BR^{n-1}$
 such that $L^\times_+$  corresponds to   $\BR_{> 0} \times \BR^{n-1}$. 
 Thus any object of $\mu\Sh_{L}(M)$ must vanish near $e_\alpha$, and in particular  any object 
 of $\mu\Sh_{L^\times}(M^\times)$ coming by restriction from  $\mu\Sh_{L}(M)$
 must vanish near $e_\alpha$.
 
 Recall the object  $\cA_+\in \mu\Sh_{L^\times}(M^\times)$ corresponding to the structure sheaf $\cO_{\BP^{n-1}} \in \Coh(\BP^{n-1})$.
 By the  above discussion, the object  $ J^*J_!(\cA_+) \in  \mu\Sh_{L^\times}(M^\times)$ 
 vanishes near $e_\alpha$. Thus by the compatibility recalled in Example~\ref{ex ccc},  the corresponding object  
   $\Cone(s) \in \Coh(\BP^{n-1})$ has vanishing stalk at the coordinate line $[e_\alpha] \in \BP^{n-1}$.
   Therefore the map $s$ must be non-zero at $[e_a] \in \BP^{n-1}$, 
 and so the zero locus of $s$ is a generic linear hypersurface 
 $$
 \xymatrix{
 i: \BP^{n-2} \ar[r] &   \BP^{n-1}
 }$$
 
 We have an equivalence of monads $A \simeq i_* i^*$,
  and hence an equivalence of modules
  $$
  \xymatrix{
  \Mod_A (\mu\Sh_{L^\times_+}(M^\times)) \ar[r]^-\sim &  \Coh(\BP^{n-2})
  }
  $$
Note  the comonad $A^\vee = J^* J_*$ is similarly equivalent to $i_*i^!$.

Recall that $J^*$ is conservative. Thus by Lurie's Barr-Beck Theorem~\cite{LurieHA}, to see
the canonical lift
$$
\xymatrix{
 \mu\Sh_{L_+}(M) \ar[r]^-{\tilde J^*} &  \Mod_A (\mu\Sh_{L^\times_+}(M^\times))
 \ar[r]^-\sim &  \Coh(\BP^{n-2})
}
$$
is an equivalence,
it suffices to check the following.  

Let $\cdots \to c_1 \to c_0$ be a complex of objects of $ \mu\Sh_{L_+}(M)$. Suppose the complex
$\cdots \to J^* c_1 \to J^* c_0$ of objects of  $ \mu\Sh_{L^\times_+}(M^\times) \simeq \Coh(\BP^{n-1})$ extends to a split colimit diagram
\vspace{0.5em}
$$
\xymatrix{
\cdots \ar[r] & \ar@/_1pc/[l] J^* c_1 \ar[r] & \ar@/_1pc/[l] J^* c_0 \ar[r]_-a & d\ar@/_1pc/[l]
}
$$
Then we must check that $\cdots \to c_1 \to c_0$ admits a colimit  in $ \mu\Sh_{L_+}(M)$.

First, observe that since $J^*c_0 \in \Coh (\BP^{n-2})$, the splitting implies $d\in \Coh(\BP^{n-2})$, or more precisely that
$d\simeq i_*d'$ where we regard $d'\in  \Coh(\BP^{n-2})$.
Choose an object $\tilde d\in \Coh(\BP^{n-1})$ together with an equivalence 
$$
\xymatrix{
f:d\ar[r]^-\sim & i_*i^!\tilde d \simeq J^* J_* \tilde d
}
$$ 
The counit $c$ of the adjunction $(J^*, J_*)$ provides an extended diagram
$$
\xymatrix{
\cdots \ar[r] & J^* c_1 \ar[r] & J^* c_0 \ar[r]^-a & 
d \ar[r]^-f & J^* J_* \tilde d \ar[r]^-c & \tilde d
}
$$
and then together with the unit $u$ of the adjunction $(J^*, J_*)$ an induced augmented complex
$$
\xymatrix{
\cdots \ar[r] & c_1 \ar[r] & c_0 \ar[rr]^-{J_*(c\circ f\circ a)\circ u} &&  J_*\tilde d
}
$$
We claim that this is the sought-after colimit diagram.

To check this, since $J^*$ is conservative, it suffices to see that the complex
$$
\xymatrix{
\cdots \ar[r] & J^*c_1 \ar[r] & J^*c_0 \ar[rr]^-{J^*J_*(c\circ f\circ a)\circ u} &&  J^*J_*\tilde d
}
$$
is a colimit diagram. Since $d$ is a colimit, it suffices to see the following diagram commutes
$$
\xymatrix{
 J^*c_0 \ar[rr]^-{J^*J_*(a)\circ  u} \ar[drr]_-a &&  J^* J_* d \ar[rr]^-{J^* J_*(c \circ f)} \ar[d]^-c &&  J^* J_* \tilde d \\
 && d \ar[urr]^-\sim_-f &&
}
$$
By standard identities for an adjunction, the triangle to the left is commutative.
Thus it suffices to show the  triangle to the right is commutative. With our previous identifications, it admits a reinterpretation
completely in terms of coherent sheaves
$$
\xymatrix{
  i_*i^! d \ar[rr]^-{ i_*i^! (c \circ f)} \ar[d]^-c &&  i_*i^! \tilde d \\
d \ar[urr]^-\sim_-f &&
}
$$
Its commutativity is a straightforward exercise we leave to the reader.
\end{proof}




\end{document}